\documentclass[reqno,10pt]{amsart}
\usepackage{latexsym,amsmath}
\usepackage{amsfonts}
\usepackage{amssymb}
\usepackage{enumerate}
\usepackage{latexsym}
\usepackage{color}

\usepackage[T1]{fontenc}
\usepackage{fourier}
\usepackage{bbm}

\usepackage{fullpage}

\numberwithin{equation}{section}

\DeclareFontFamily{U}{mathx}{\hyphenchar\font45}
\DeclareFontShape{U}{mathx}{m}{n}{
      <5> <6> <7> <8> <9> <10>
      <10.95> <12> <14.4> <17.28> <20.74> <24.88>
      mathx10
      }{}
\DeclareSymbolFont{mathx}{U}{mathx}{m}{n}
\DeclareFontSubstitution{U}{mathx}{m}{n}
\DeclareMathAccent{\widecheck}{0}{mathx}{"71}

\newcommand\ism{\cong}

\newcommand{\Bck}[1]{\left\llbracket{#1}\right\rrbracket}

\newcommand{\nix}{\,\cdot\,}
\newcommand{\Erdos}{Erdos}
\newcommand{\Renyi}{R\'enyi}

\DeclareMathOperator{\pr}{\mathbb P}

\newcommand\T{\vec T}
\def\vec#1{\mathchoice{\mbox{\boldmath$\displaystyle#1$}}
{\mbox{\boldmath$\textstyle#1$}}
{\mbox{\boldmath$\scriptstyle#1$}}
{\mbox{\boldmath$\scriptscriptstyle#1$}}}

\newcommand\ppartial{\Delta}

\newcommand\eps{\varepsilon} 
\newcommand\betac{\beta_{\mathrm{cond}}}
\newcommand\dk{d_{k\mathrm{-SAT}}}
\newcommand\dc{d_{\mathrm{cond}}(k)}
\newcommand{\vecone}{\vec{1}}

\newcommand\SIGMA{\vec\sigma}

\newcommand{\Bin}{{\rm Bin}}

\newcommand{\bck}[1]{\left\langle{#1}\right\rangle}
\newcommand\brk[1]{\left\lbrack{#1}\right\rbrack}

\newcommand\abs[1]{\left|{#1}\right|}

\newcommand\RR{\mathbb{R}}

\newcommand{\Whp}{W.h.p.}
\newcommand{\whp}{w.h.p.}

\newcommand{\tensor}{\otimes}

\newcommand\cA{\mathcal{A}}
\newcommand\cB{\mathcal{B}}
\newcommand\cC{\mathcal{C}}
\newcommand\cD{\mathcal{D}}
\newcommand\cF{\mathcal{F}}
\newcommand\cG{\mathcal{G}}
\newcommand\cE{\mathcal{E}}

\newcommand\cH{\mathcal{H}}

\newcommand\cT{\mathcal{T}}

\newcommand\cP{\mathcal{P}}

\def\hPhi{{\widehat{\Phi}}}

\def\hve{{\widehat{\varepsilon}}}
\def\w{\ell}

\def\GW{{\textrm{GW}}}

\def\hnu{{\widehat{\nu}}}
\def\hmu{{\widehat{\mu}}}
\def\heta{{\widehat{\eta}}}
\def\hh{{\widehat{h}}}
\def\hZ{{\widehat{Z}}}
\def\hz{{\widehat{z}}}
\def\hf{{\widehat{f}}}
\def\hpi{{\widehat{\pi}}}

\def\hq{{{\widehat{h}}}}
\def\hqq{{{\widehat{q}}}}
\def\dd{{{\mathrm{d}}}}
\def\ovf{{\bar{f}}}
\def\ovz{{\bar{z}}}
\def\hT{{{\widehat{T}}}}

\newcommand\PHI{\vec\Phi}

\newcommand\hhProb{{\mathbb{P}}}
\newcommand\hhErw{{\mathbb{E}}}
\def\Wvhp{{W.v.h.p.}}
\def\wvhp{{w.v.h.p.}}
\def\hhPhi{{\widetilde{\Phi}}}
\def\dplus{{d_{k-\textrm{SAT}}}}
\newcommand\Prob{\mathbb{P}}
\newcommand\hProb{{\mathbb{P}}}
\newcommand\cProb{{\mathbb{P}}}
\newcommand\Erw{\mathbb{E}}
\newcommand\hErw{{\mathbb{E}}}
\newcommand\cErw{{\mathbb{E}}}

\def\cD{{\mathcal{D}}}
\def\cB{{\mathcal{B}}}
\def\cE{{\mathcal{E}}}
\def\hcE{\widehat{\mathcal{E}}}
\def\cF{{\mathcal{F}}}
\def\cG{{\mathcal{G}}}
\def\cP{{\mathcal{P}}}

\def\hcF{{\widehat{\mathcal{F}}}}
\def\cH{{\mathcal{H}}}
\def\cA{{\mathcal{A}}}

\def\cT{{\mathcal{T}}}

\def\Whp{{W.h.p.}}
\def\whp{{w.h.p.}}
\def\core{{\rm{Core}}}
\def\trunk{{\rm{Trunk}}}

\def\um{{{\underline{m}}}}

\newtheorem{definition}{Definition}[section]
\newtheorem{claim}[definition]{Claim}

\newtheorem{remark}[definition]{Remark}
\newtheorem{theorem}[definition]{Theorem}
\newtheorem{lemma}[definition]{Lemma}
\newtheorem{proposition}[definition]{Proposition}
\newtheorem{corollary}[definition]{Corollary}

\newtheorem{fact}[definition]{Fact}

\newcommand\Claim{Claim}

\newcommand\Lem{Lemma}
\newcommand\Prop{Proposition}
\newcommand\Thm{Theorem}
\newcommand\Cor{Corollary}
\newcommand\Sec{Section}

\newcommand\Rem{Remark}

\newcommand{\beq}{\begin{equation}} \newcommand{\eeq}{\end{equation}}

\allowdisplaybreaks

\begin{document}

\title{The condensation phase transition in the regular $k$-SAT model}

\author[Bapst and Coja-Oghlan]{Victor Bapst$^*$, Amin Coja-Oghlan$^*$}
\thanks{$^*$The research leading to these results has received funding from the European Research Council under the European Union's Seventh 
Framework Programme (FP/2007-2013) / ERC Grant Agreement n.\ 278857--PTCC}
\date{\today} 

\address{Amin Coja-Oghlan, {\tt acoghlan@math.uni-frankfurt.de}, Goethe University, Mathematics Institute, 10 Robert Mayer St, Frankfurt 60325, Germany.}

\address{Victor Bapst, {\tt bapst@math.uni-frankfurt.de}, Goethe University, Mathematics Institute, 10 Robert Mayer St, Frankfurt 60325, Germany.}

\begin{abstract}
\noindent
Much of the recent work on phase transitions in random discrete structures has been inspired by ingenious but non-rigorous approaches from physics.
The physics predictions typically come in the form of distributional fixed point problems that are intended to mimic Belief Propagation,
a message passing algorithm.
In this paper we propose a novel method for harnessing Belief Propagation directly to obtain a rigorous proof of such a prediction, namely
the existence and location of a condensation phase transition in the random regular $k$-SAT model.

\bigskip
\noindent
%\emph{Key words:}	random structures, phase transitions, graph coloring.
\end{abstract}

\maketitle

\section{Introduction} % and results}
\label{sec_introduction}

\subsection{Background and motivation}
%\noindent
Over the past three decades the study of random constraint satisfaction problems has been driven by ideas from statistical physics~\cite{MPV,MPZ}.
This work has had a substantial impact on computer science (e.g., proofs that certain benchmark instances
are difficult for certain algorithms),  coding theory (``low density parity check codes'')
and probabilistic combinatorics (random graphs, hypergraphs and formulas); e.g., ~\cite{kSAT,DMSS,DSS3,GNS,David1,David2,RichardsonUrbanke}.
All of these disciplines deal with a common setup.
There are a large number of ``variables'' that interact through a similarly large number of ``constraints''. 
Each variable ranges over a finite domain (such as the Boolean values `true' and `false') and
every constraint binds a small number of variables, either encouraging or discouraging certain value combinations.

The striking feature of the physics work is that it is based on a non-rigorous but generic approach called the {\em cavity method},
centered around the {\em Belief Propagation} message-passing algorithm, that can be applied almost mechanically~\cite{MM}.
Hence the impact of a single technique on such a wide range of problems.
By comparison, the rigorous study of random problems has largely been case-by-case. 
This begs the question of whether the Belief Propagation calculations can be put on a rigorous basis more directly.

This is precisely 
the thrust of the present paper.
We show how the physics calculations can be turned into a 
rigorous proof in a highly non-trivial and somewhat representative case.
Specifically, we  determine the ``condensation phase transition''
in the random regular $k$-SAT model. 
The proof is based on a novel approach that demonstrates how
our recent general results on the connection between spatial mixing properties and the computation
of the free energy~\cite{betheupb} can be put to work.
The centrepiece of the proof is a fairly direct analysis of the Gibbs marginals by means of  Belief Propagation.
The arguments are rather generic and we expect them to extend to other problems.

The random regular $k$-SAT model is defined as follows~\cite{Rathi}.
There are Boolean variables $x_1,\ldots,x_n$ and $m$ contraints, namely propositional clauses of length $k$.
Each variable occurs precisely $d/2$ times as a positive and precisely $d/2$ times as a negative literal.
Hence, $m=dn/(2k)$; we assume tacitly that $d$ is even and that $k$ divides $dn$.
Let $\PHI=\PHI_{d,k}(n)$ signify a uniformly random such $k$-SAT formula.%
	\footnote{The regular $k$-SAT model shares many of the properties of the better known model where $m$ clauses are chosen uniformly and independently
		but avoids the intricacies that result from having a few variables of very high degree.}
For $k$ exceeding a certain constant $k_0$
the threshold where $\PHI$ ceases to be satisfiable is known~\cite{kSAT}.
	\footnote{In the sense that $\liminf_{n\to\infty}\pr\brk{\PHI\mbox{ is satisfiable}}>0$ if $d<\dk$ and 
		$\lim_{n\to\infty}\pr\brk{\PHI\mbox{ is satisfiable}}=0$ if $d>\dk$.}
While the exact formula is cumbersome, asymptotically $\dk/k=2^k\ln2-k\ln2/2+O(1)$ 
 for large $k$.
	
Of course, 
finding the satisfiability threshold is hardly the end of the story.
Much more precise information is encoded in
the Hamiltonian $\sigma\mapsto E_{\PHI}(\sigma)$ that maps each truth assignment $\sigma$ to the number of clauses that it violates.
We think of it as a ``landscape'' on the  Hamming cube.
For instance, if  $E_{\PHI}$ is riddled with local minima, we should expect that Markov processes such as
Simulated Annealing get trapped~\cite{Barriers,pnas,MPZ}.
Hence, $E_{\PHI}$ holds the key to understanding algorithms for finding, counting and sampling
	solutions~\cite{MontanariShah,RTS}.

The key quantity upon which the study of the Hamiltonian hinges is the {\em partition function} 
	\begin{align*}
	Z_{\PHI}:\beta\in(0,\infty)\mapsto\sum_\sigma\exp(-\beta E_{\PHI}(\sigma)).
	\end{align*}
As usual, the larger the {\em inverse temperature} $\beta$, the bigger the relative contribution of ``good'' assignments that violate few clauses.
Of course, we are interested in the asymptotics as $n\to\infty$.
Since $Z_{\PHI}(\beta)$ scales exponentially with $n$, we consider%
	\begin{equation}\label{eqFreeEnergy}
	\phi_{d,k}:\beta\in(0,\infty)\mapsto\lim_{n\to\infty}\frac1n\Erw[\ln Z_{\PHI}(\beta)].
	\end{equation}
Clearly, what makes $\phi_{d,k}$ vicious is that the log is {\em inside} the expectation.
The existence of the limit follows from the interpolation method~\cite{bayati}
and Azuma's inequality implies that $\ln Z_{\PHI}(\beta)$ concentrates about $\Erw[\ln Z_{\PHI}(\beta)]$.

A key question is how smoothly $\phi_{d,k}(\beta)$ varies as a function of $\beta$ for {\em fixed} $d,k$.
Formally, let us call $\beta_0\in(0,\infty)$ {\em smooth} if there exists $\eps>0$ such that the function $\beta\in(\beta_0-\eps,\beta_0+\eps)\mapsto\phi_{d,k}(\beta)$
admits an expansion as an absolutely convergent power series around $\beta_0$.
If $\beta_0$ fails to be smooth, a {\em phase transition} occurs at $\beta_0$.

\subsection{Results} \label{subsec_results}
According to (non-rigorous) physics predictions~\cite{pnas} for certain values of $d$ close to the 
satisfiability threshold $\dk$  there occurs a so-called {\em condensation phase transition} at a certain critical $\betac(d,k)>0$.
The main result of this paper 
proves this conjecture.
Let us postpone the precise definition of $\betac(d,k)$ for a moment.

\begin{theorem}\label{Thm_main}
There exists $k_0\geq3$ such that for all $k\geq k_0$, 
	$d\leq\dk$
there is $\betac(d,k)\in(0,\infty]$ such that
any $\beta\in(0,\betac(d,k))$ is smooth.
If $\betac(d,k)<\infty$, then
there occurs a phase transition at $\betac(d,k)$.
\end{theorem}

\noindent
Thus, if we fix $d,k$ such that $\betac(d,k)=\infty$, then the function $\phi_{d,k}$ is analytic on $(0,\infty)$.
But if $d,k$ are such that $\betac(d,k)<\infty$, then $\phi_{d,k}$ is non-analytic at the point $\betac(d,k)$.
In fact, we will see that $\betac(d,k)<\infty$ for $d$ 
exceeding a specific $\dc<\dk$. Crucially, \Thm~\ref{Thm_main} identifies the {\em precise} condensation threshold  $\betac(d,k)$;
it is the first such result in a model of this kind.

Let us take a look at the precise value of $\betac(d,k)$.
As most predictions based on the cavity method, $\betac(d,k)$ results from a {\em distributional fixed point problem}, i.e.,
a fixed point problem on the space of probability measures on the unit interval $(0,1)$.
The fixed point problem derives mechanically from the ``1RSB cavity equations''~\cite{MM}.
Specifically, writing $\cP(\Omega)$ for the set of probability measures on $\Omega$, we define two maps
	$$\cF_{k,d,\beta}:\cP(0,1)\to\cP(0,1),\qquad\hat\cF_{k,d,\beta}:\cP(0,1)\to\cP(0,1)$$
as follows.
Given $\pi\in\cP(0,1)$ let $\eta=(\eta_1,\ldots,\eta_{k-1})\in(0,1)^{k-1}$ be a random $k-1$-tuple drawn from the  distribution
	$(\hat z(\eta)/\hat Z(\pi))\,\dd\bigotimes_{j=1}^{k-1}\pi(\eta_j)$, where
	\begin{align*}
	\hat z(\eta)&=2-(1-\exp(-\beta))\prod_{j<k}\eta_j\qquad\mbox{and}\quad
	&\hat Z(\pi)&=\int\hat z(\eta)\dd\bigotimes_{j<k}\pi(\eta_j).
	\end{align*}
Then $\hcF_{k,d,\beta}(\pi)$ is the distribution of 
	$(1-(1-\exp(-\beta))\prod_{i=1}^{k-1} \eta_i) /{\hz(\eta)}.$
Similarly, given $\hat\pi\in\cP(0,1)$ draw $\hat\eta=(\hat\eta_1,\ldots,\hat\eta_{d-1})$ from
	$(z(\hat\eta)/Z (\hat\pi))\dd\bigotimes_{j=1}^{k-1}\hat\pi(\hat\eta_j)$, where
	\begin{align*}
	z (\heta)&= {\prod_{j<d/2}\heta_j \prod_{j\geq d/2}(1-\heta_j)+\prod_{j<d/2} (1-\heta_j) \prod_{j\geq d/2}\heta_j},
		&\  Z ( \hat\pi )=\int z(\heta)\dd\bigotimes_{j<k}\hat\pi(\hat\eta_j).
	\end{align*}
Then  $\cF_{k,d,\beta}(\hpi)$ is the distribution of $(\prod_{j<d/2} \heta_j \prod_{j\geq d/2}(1-\heta_j))/z(\heta)$.
Call a distribution $\pi\in\cP(0,1)$ {\em skewed} if 
the probability mass of the interval $(0,1-\exp(-k\beta/2))$ satisfies
$\pi(0,1-\exp(-k\beta/2))<2^{-0.9 k}$.

\begin{proposition} \label{prop_unique_fixed_point_tree_1}
Let $d_-(k)=\dk-k^5$ and $\beta_-(k,d)=k\ln2-10\ln k$.
The map $\cG_{k,d,\beta} = \cF_{k,d,\beta} \circ \hcF_{k,d,\beta}$ has a unique skewed fixed point $\pi^\star_{k,d,\beta}$, provided that
 $k\geq k_0$, $d\in[d_-(k),\dk]$ and $\beta>\beta_-(k,d)$.
\end{proposition}

To extract $\betac(d,k)$, let $\nu_1,\ldots,\nu_k,\hat\nu_1,\ldots,\hat\nu_d$ be independent random variables
such that the $\nu_i$ have distribution $\pi^\star_{k,d,\beta}$ and the $\hat\nu_i$ have distribution $\hcF_{k,d,\beta}(\pi^\star_{k,d,\beta})$.
Setting
	\begin{align*}z_1 &= {\prod_{j\leq d/2} \hat\nu_j \prod_{j>d/2} (1-\hat\nu_j)+\prod_{j\leq d/2} (1-\hat\nu_j) \prod_{j>d/2} \hat\nu_j}, &
	 z_2 &=  1-(1-\exp(-\beta))\prod_{j\leq k} \nu_j\end{align*}
and	$  z_3= \nu_1 \hat\nu_1 + (1-\nu_1)(1-\hat\nu_1)$, we let
	\begin{align}\label{eqmyBethe}
	\cF(k,d,\beta) &=\ln \Erw \left[ z_1\right] + \frac{d}{k}\ln \Erw \left[ z_2\right] -d \ln \Erw [z_3], &
	\cB(k,d,\beta)  &=  \frac{\Erw \left[ z_1\ln z_1\right]}{\Erw \left[ z_1\right]} +
		\frac{d}{k} \frac{\Erw \left[ z_2\ln z_2\right]}{\Erw \left[ z_2\right]} 
	-d \frac{\Erw [z_3 \ln z_3]}{\Erw [z_3]} . \end{align}
Finally, with the usual convention that $\inf\emptyset=\infty$ 
we let 
	$$\betac(k,d)=\begin{cases}
		\infty&\mbox{ if }d<d_-(k),\\
		\inf \{ \beta>\beta_-(k,d): \cF(k,d,\beta) < \cB(k,d,\beta) \}&\mbox{ if }d \in [d_-(k), \dk].
		\end{cases}$$

We proceed to highlight a few consequences of \Thm~\ref{Thm_main} and its proof.
The following result shows that $\betac(d,k)<\infty$, i.e., that a condensation phase transition occurs,
for degrees $d$ strictly below the satisfiability threshold.

\begin{corollary}\label{Cor_dc}
If $k\geq k_0$, then $\dc=\min\{d>0:\betac(d,k)<\infty\}<\dk-\Omega(k)$.
\end{corollary}

\noindent
Furthermore, the following corollary shows that the so-called ``replica symmetric solution'' predicted by the cavity method
yields the correct value of $\phi_{d,k}(\beta)$ for $\beta<\betac(d,k)$.

\begin{corollary}\label{Cor_rs}
If $k\geq k_0$, $d\leq \dk$ and $\beta<\betac(d,k)$, then $\phi_{d,k}(\beta)=
	\cF(k,d,\beta)$.
\end{corollary}

%\noindent
\Cor~\ref{Cor_rs} opens the door to studying the ``landscape'' $E_{\PHI}$ for $\beta<\betac(d,k)$.
Specifically, \Cor~\ref{Cor_rs} enables us to bring the ``planting trick'' from~\cite{Barriers} to bear so that we can analyse
typical properties of samples from the Gibbs measure. We leave a detailed discussion to future work.
Finally, complementing \Cor~\ref{Cor_rs}, the following result shows that 
 $\cF(k,d,\beta)$ overshoots  $\phi_{d,k}(\beta)$
for $\beta>\betac(d,k)$.

\begin{corollary}\label{Cor_1rsb}
If $k\geq k_0$, $d\leq \dk$ and $\beta>\betac(d,k)$, then there is $\betac(d,k) < \beta'<\beta$ such that $\phi_{d,k}(\beta')<\cF(k,d,\beta')$.
\end{corollary}

\subsection{Outline and related work}
Admittedly, the definition of $\betac(k,d)$ is not exactly simple.
For instance, even though the fixed point distribution 
from \Prop~\ref{prop_unique_fixed_point_tree_1} stems from a discrete problem,
it turns out to be a continuous distribution on $(0,1)$. Yet perhaps despite appearances, the analytic formula (\ref{eqmyBethe}) is conceptually {\em far} simpler than the definition of $\phi_{d,k}$.
For instance, we are going to see in \Sec~\ref{xsec_overview} that the fixed point problem 
can be understood elegantly in terms of a Galton-Watson tree. 
Thus, one could say that \Thm~\ref{Thm_main} reduces the condensation problem on the complex random formula $\PHI$ to a problem
on a random tree.

The proof of \Thm~\ref{Thm_main} 
builds upon an abstract result from~\cite{betheupb} that, roughly speaking, reduces the study of the partition function to two tasks.
First, to calculate the  marginals of the Gibbs measure induced by a 
random formula $\hat\PHI$ chosen from a reweighted probability distribution, the ``planted model''.
Second, to prove that the Gibbs measure of $\hat\PHI$ enjoys the non-reconstruction property, a spatial mixing property.
The technical contribution of the present work is to actually tackle these two tasks problems in a fairly generic way. Our principal tool is going to be the Belief Propagation algorithm, the cornerstone of the physicsts' cavity method. In particular, we are going to reduce the see 
that the distributional operator $\cG_{k,d,\beta}$ from \Prop~\ref{prop_unique_fixed_point_tree_1} mimics
Belief Propagation run on a Galton-Watson tree that captures the local geometry of the formula $\hat\PHI$. The predictions of the ``cavity method'' typically come as distributional fixed points 
 but there are few proofs that establish such predictions rigorously.
The one most closely related to the present work 
is the paper of Bapst et al.~\cite{condensationColoring} on condensation in random graph coloring. It determines the critical average degree $d$ for which condensation starts to occur 
with respect to the number of proper $k$-colorings of the \Erdos-\Renyi\ random graph.
Conceptually, this corresponds to taking the limit $\beta\to\infty$ in (\ref{eqFreeEnergy}), 
which simplifies the problem rather substantially.
 Thus, the main result of~\cite{condensationColoring} corresponds to \Cor~\ref{Cor_dc}. Other previous results on condensation, which dealt with random hypergraph $2$-coloring and the Potts model on the random graph, were only approximate
	\cite{BCOR,Lenka,CDGS}.

Interestingly, determining the satisfiability threshold on the random regular formula $\PHI$ is conceptually much easier than identifying the condensation threshold~\cite{kSAT}.
This is because the local structure of the random formula $\PHI$ is essentially deterministic, namely a tree comprising of clauses and variables in which every variable
appears $d/2$ times positively and $d/2$ times negatively.
In effect, 
the satisfiability threshold is given by a fixed point problem on the unit interval rather than on the space of probability measures on the unit interval.
Similar simplifications occur in other regular models~\cite{DSS1,DSS2}, and these proofs employed Belief Propagation 
in this simpler setting.
By contrast, we will see in \Sec~\ref{xsec_overview} that the condensation phase transition hinges on the reweighted distribution $\hat\PHI$, whose local structure is 
genuinely random.

Recent work on the $k$-SAT threshold in uniformly random formulas~\cite{kSAT,Kosta}, in particular the breakthrough paper by Ding, Sly and Sun~\cite{DSS3},
also harnessed the physicists' Belief Propagation or Survey Propagation calculations.%
	\footnote{Survey Propagation can be viewed as a Belief Propagation applied to a modified constraint satisfaction problem~\cite{MM}.}
In the uniformly random model a substantial technical difficulty is posed by the presence of variables of exceptionally high degree,
	an issue that is, of course, absent in the regular  model.
Specifically, \cite{kSAT,Kosta,DSS3} apply the second moment method to a random variable whose construction is guided by Belief/Survey Propagation.
By contrast, here we employ Belief Propagation in the more direct way enabled by~\cite{betheupb}.

\subsection{Notation and preliminaries}
We generally view a regular $k$-SAT instance $\Phi$ as bijections from sets of clause clones to sets of variable clones (``configuration model'').
That is, given $n,m,d,k$, we let $\{x_1,\ldots,x_n\}\times[d]$ be the set of variable clones and $\{a_1,\ldots,a_m\}\times[k]$ the set of clause clones.
Then $\Phi:\{x_1,\ldots,x_n\}\times[d]\to\{a_1,\ldots,a_m\}\times[k]$ is a bijection.
The first $d/2$ clones of each variable are considered its positive occurrences and the last $d/2$ ones its negative occurrences.

We denote the image of a clone $(x_i,j)$ by $\partial_\Phi(x_i,j)$ and the inverse image of $(a_i,j)$ by $\partial_\Phi(a_i,j)$.
Analogously, $\partial^\w_\Phi(v,j)$ is the depth-$\w$ neighborhood of clone $(v,j)$.
Moreover, we define $\PHI$  as a uniformly random bijection. 
By standard arguments this distribution is easily seen to be contiguous to the uniform distribution on regular formulas.

Suppose that the variables and clauses of $\Phi,\Phi'$ are $x_i,x_i',a_j,a_j'$ for $i\in[n]$, $j\in[m]$.
We distinguish (variable or clause) clones $r,r'$ of $\Phi,\Phi'$, which we consider their roots.
An {\em isomorphism} $\psi:\Phi\to\Phi'$ is a bijection
with the following properties.
	\begin{description}
	\item[ISM1] $r'=\psi(r)$.
	\item[ISM2] $\psi$ maps variable clones to variable clones and clause clones to clause clones.
	\item[ISM3] If $\psi(v,h)=(w,j)$, then $h=j$.
	\item[ISM4] We have $\psi\circ\Phi(v,h)=\Phi'\circ\psi(v,h)$ for all clones $(v,h)$.
	\end{description}

Let $\w\geq0$ 
and let $T$ be a regular $k$-SAT formula with a distinguished (variable or clause) clone $r$.
For each variable clone $(x,i)$ of $\PHI$ we have a random variable $\vecone\{\partial^\w T\cong\partial_{\PHI}^\w(x,i)\}$
that indicates that the depth-$\w$ neighborhood of $\PHI$ rooted at $(x,i)$ is isomorphic to $T$.
Similarly, for each clause cone $(a,j)$ of $\PHI$ we consider the random variable $\vecone\{\partial^{\w+1} T\cong\partial_{\PHI}^{\w+1}(a,j)\}$.
Let $\mathfrak T_\w$ be the $\sigma$-algebra generated by all these random variables.
Thus, $\mathfrak T_\w$ captures the ``local structure'' of the random formula up to depth $\w$.

\section{Outline}
\label{xsec_overview}

\subsection{Two moments do not suffice}
The default approach to studying the function $\phi_{d,k}(\beta)$ is the venerable ``second moment method''.
Cast on a logarithmic scale, if
	\begin{align}\label{eqVanillaSMM}
	\limsup_{n\to\infty}\frac1n\ln\Erw[Z_{\PHI}(\beta)^2]&\leq\lim_{n\to\infty}\frac2n\ln\Erw[Z_{\PHI}(\beta)],&\mbox{then }\\
	\phi_{d,k}(\beta)&=\lim_{n\to\infty}\frac1n\ln\Erw[Z_{\PHI}(\beta)].	\label{eqannealed}
	\end{align}
The last term 
is easy to study because the log is outside the expectation.
In particular, the function $\beta\in(0,\infty)\mapsto\lim_{n\to\infty}\frac1n\ln\Erw[Z_{\PHI}(\beta)]$ turns out to be analytic.
Consequently, the least $\beta\in(0,\infty)$ where (\ref{eqannealed}) fails to hold must be a phase transition.

From a bird's eye view, both the physics intuition and the second moment are all about
the geometry of the 
{\em Gibbs measure} of  $\PHI$ at a given $\beta\in(0,\infty)$. 
Let us encode truth assignments as points  $\sigma\in\{\pm1\}^n$ with the convention that
$1$ stands for `true' and $-1$ for `false'.
Then the Gibbs measure is  the distribution on $\{\pm1\}^n$  defined by
	$$\sigma\in\{\pm1\}^n\mapsto \exp(-\beta E_{\PHI}(\sigma))/Z_{\PHI}(\beta).$$
Thus, we weigh assignments according to the number of clauses that they violate, giving greater weight to `better' assignments as $\beta$ gets larger.
Let $\SIGMA,\SIGMA_{1},\SIGMA_{2},\ldots$ be independent samples from the Gibbs measure and write
$\bck{X(\SIGMA_1,\ldots,\SIGMA_l)}_{\PHI,\beta}$ for the expectation of 
 $X:(\{\pm1\}^{n})^l\to\RR$. 
Then 
according to the physics picture the condensation point $\betac(k)$ should be the supremum of all $\beta>0$ such that
	$\Erw\bck{\abs{\SIGMA_1\cdot\SIGMA_2}}_{\PHI,\beta}=o(n)$. In other words, if we choose a random formula $\PHI$ and then sample two assignments $\SIGMA_1,\SIGMA_2$ according to the Gibbs measure independently,
then $\SIGMA_1,\SIGMA_2$ will be about orthogonal. This decorrelation property is, roughly speaking, a {\em necessary} condition for the success of the second moment method as well~\cite{nae,yuval}.
Therefore, the prediction that $\Erw\bck{\abs{\SIGMA_1\cdot\SIGMA_2}}_{\PHI,\beta}=o(n)$ right up to $\betac(d,k)$
may inspire confidence that the same is true of~(\ref{eqVanillaSMM}).
In fact, we will prove in \Sec~\ref{sec_vanilla} that (\ref{eqVanillaSMM}) holds if either $d$ or $\beta$ is relatively small.

\begin{lemma}\label{Lemma_trivial}
If  $d\leq d_-(k)$ or 
$\beta\leq\beta_-(k,d)$ then~(\ref{eqVanillaSMM}) is true.
\end{lemma}

\noindent
However, for $\beta$ near $\betac(d,k)$ 
the second moment method turns out to fail rather spectacularly.
Formally, if $\betac(d,k)<\infty$, then 
there exists $\eps>0$ such that~(\ref{eqVanillaSMM}) is violated for all $\beta>\betac(d,k)-\eps$, i.e.,
the second moment overshoots the square of the first moment by a factor that is exponential in $n$.

\subsection{Quenching the average}
To understand what goes awry
 it is convenient to turn the second moment 
into a first moment under a reweighted distribution that we call the {\em planted model}.
This is the distribution on formula/assignment pairs  
under which the probability of $(\hat\Phi,\hat\sigma)$ equals 
	$\exp(-\beta E_{\hat\Phi}(\hat\sigma))/\Erw[Z_{\PHI}(\beta)]$.
Let $(\hat\PHI,\hat\SIGMA)$ be a random pair drawn from this distribution.
Then by symmetry the distribution of the assignment $\hat\SIGMA$ is uniform and we may assume without loss that $\hat\SIGMA=\vecone$
is the all-ones assignment.
Further, the probability that a specific formula $\hat\Phi$ comes up equals $\pr[\hat\PHI=\hat\Phi]=Z_{\beta}(\hat\Phi)/\Erw[Z_{\beta}(\PHI)]$.
 Thus, the planted distribution weighs formulas by their partition function.
In effect,
		$$\Erw[Z_{\PHI}(\beta)^2]=\Erw[Z_{\PHI}(\beta)]\cdot\Erw[Z_{\hat\PHI}(\beta)].$$

If we go over the proof of \Lem~\ref{Lemma_trivial}, we see that $\Erw[Z_{\hat\PHI}(\beta)]$ is dominated by two distinct contributions.
First, assignments that are more or less orthogonal to $\hat\SIGMA$ yield a term of order $\Erw[Z_{\PHI}(\beta)]$.
Second, there is a  contribution from $\SIGMA$ close to $\hat\SIGMA=\vecone$; say,
	 $\SIGMA\cdot\vecone\geq n(1-2^{-{k/10}})$.
Geometrically, this reflects the fact that 
 the planted assignment $\vecone$
sits in a ``valley'' of the Hamiltonian $E_{\hat\PHI}$ \whp\
The valleys are known as clusters in the physics literature and we let
	$$\cC_{\hat\PHI,\hat\SIGMA}(\beta)=Z_{\hat\PHI}(\beta)\langle{\vecone\{\SIGMA\cdot\vecone>n(1-2^{-k/10})\}}\rangle_{\hat\PHI,\beta}$$
be the (weighted) {\em cluster size}.
Performing an elementary calculation, we find that it is the expected cluster size that derails the second moment method for $\beta$ 
 near $\betac(d,k)$.

At a second glance, this is unsurprising.
For $\cC_{\hat\PHI,\hat\SIGMA}(\beta)$ scales exponentially with $n$
and is therefore prone to large deviations effects. To suppress these we ought 
to investigate
 $\Erw[\ln\cC_{\hat\PHI,\hat\SIGMA}(\beta)]$ 
instead of $\Erw[\cC_{\hat\PHI,\hat\SIGMA}(\beta)]$.
A similar issue (that the expected cluster size drives up the second moment) occurred in earlier work on condensation~\cite{condensationColoring,BCOR,Lenka,CDGS}.
Borrowing the remedy suggested in these papers, we prove in \Sec~\ref{sec_vanilla} that 
applying the second moment method to a carefully truncated random variable yields

\begin{lemma}\label{Lem_smmRemedy}
Equation (\ref{eqannealed}) holds iff 
	\begin{equation}\label{eqProp_smmRemedy1}
	\limsup_{n\to\infty}n^{-1}\Erw[\ln\cC_{\hat\PHI,\hat\SIGMA}(\beta)]\leq\lim_{n\to\infty}n^{-1}\ln \Erw[Z_{\hat\PHI}(\beta)].
	\end{equation}	
\end{lemma}

\noindent
Computing $\lim_{n\to\infty}n^{-1}\ln \Erw[Z_{\hat\PHI}(\beta)]$ is easy, as the following standard lemma shows.
\begin{lemma} \label{lemma_total_size} Assume that $d \leq \dplus$ and  $\beta \in \mathbb{R}$. Then 
$\lim_{n\to\infty}n^{-1}\ln \Erw[Z_{\hat\PHI}(\beta)] = \cF(k,d,\beta)$.
\end{lemma}

Hence, we are left to calculate $\Erw[\ln \cC_{\hat\PHI,\hat\SIGMA}(\beta)]$, the ``quenched average'' in physics jargon.
As the log and the expectation do not commute, this problem is well beyond the reach of elementary methods.
Tackling it is the main achievement of this paper.
Specifically, we are going to prove

\begin{proposition}\label{Prop_clusterSize}
Assume that $d \in [d_-(k), \dplus]$ and  $\beta > \beta_-(k,d)$.
Then
	$\lim_{n\to\infty}n^{-1}\Erw[\ln\cC_{\hat\PHI,\hat\SIGMA}(\beta)]=\cB(k,d,\beta).$
\end{proposition}

\subsection{Non-reconstruction and the Bethe free energy} \label{sec_non_reconstruction_and_bethe}

In the following we let for a formula $\Phi$ and $v \in V \cup F$ and $\w \geq 0$, $\partial_\Phi^{\w} v$ (resp. $\Delta_\Phi^{\w} v$) denote the set of vertices at distance exactly $\w$ (resp. less than $\w$) from $v$ in $\Phi$.

To prove \Prop~\ref{Prop_clusterSize} we investigate the spatial mixing properties of the conditional Gibbs measure 
	$$\Bck{\nix}_{\hat\Phi,\beta}=\bck{\nix\big|\cC_{\hat\Phi,\hat\sigma}(\beta)}_{\hat\Phi,\beta}.$$
Specifically, for a variable $x$, an assignment $\sigma\in\cC_{\hat\Phi,\hat\sigma}(\beta)$ and an integer $\ell\geq0$ let $\nabla(\hat\Phi,x,\ell)$
be the $\sigma$-algebra on $\cC_{\hat\Phi,\hat\sigma}(\beta)$ generated by the random variables $\SIGMA(y)$ for variable $y$
at distance greater either $\ell$ or $\ell+1$ from $x$.
Further, we define
	\begin{align}\label{eqNaiveMargs}
	\mu_{\hat\Phi,x}^{(\ell)}(\pm1)&=\Bck{\vecone\{\SIGMA(x)=\pm1\}|\nabla(\hat\Phi,x,\ell)}_{\hat\Phi,\hat\sigma}(\hat\sigma).
	\end{align}
In words, $\mu_{\hat\Phi,x}^{(\ell)}(\pm1)$ is the probability that $x$ gets assigned $\pm1$ in a random assignment
of its depth-$\ell$ neighborhood under the boundary condition induced by $\hat\sigma$.

We lift the distributions from (\ref{eqNaiveMargs}) to clauses.
In slightly greater generality, suppose that $\mu$ is a map that assigns each variable $x$ a probability distribution $\mu_x\in\cP(\{\pm1\})$.
Then for clause $a$ of a formula $\hat\Phi$ we let $\mu_{\hat\Phi,a}^{(2 \ell +1)}$ be the distribution on $\{\pm1\}^k$ with the following two properties.
	\begin{enumerate}[(i)]
	\item if $j\in[k]$ and $x=\partial_{\hat\Phi}(a,j)$, then the marginal distribution of the $j$th coordinate coincides with $\mu_{x}$.
	\item subject to (i), $H(\mu_{\hat\Phi,a}^{(2 \ell +1)})+\bck{\ln\psi_a}_{\mu_{\hat\Phi,a}^{(2 \ell +1)}}$ is maximum.
	\end{enumerate}
These two conditions determine $\mu_{\hat\Phi,a}^{(2\ell+1)}$ uniquely (because the entropy is concave).

Now, we say that a formula $\hat\PHI$ has the {\em non-reconstruction property} if for any $\eps>0$ there is $\ell>0$ such that
	\begin{align*}
	\lim_{n\to\infty}\pr\brk{\frac1n\sum_{x}\Bck{\abs{\mu_{\hat\Phi,x}^{(2\ell)}(1)-\Bck{\SIGMA(x)|\nabla(\hat\PHI,x,2\ell)}_{\hat\PHI,\beta}}}_{\hat\PHI,\beta}<\eps}&=1
	\end{align*}
The first half of the proof of \Prop~\ref{Prop_clusterSize} consist in proving the following.

\begin{proposition}\label{Prop_nonRe1}
Assume that $d \in [d_-(k), \dplus]$ and  $\beta > \beta_-(k,d)$.
Then $\hat\PHI$ has the non-reconstruction property.
\end{proposition}

Together with results from~\cite{betheupb} \Prop~\ref{Prop_nonRe1} implies an upper bound on $\Erw[\ln\cC_{\hat\PHI,\hat\SIGMA}(\beta)]$.
But to obtain a matching lower bound a little more work is needed. Specifically, we need to consider a further distribution on formula/assignment pairs that we call the {\em planted replica model}
 $(\tilde\PHI,\tilde\SIGMA)$ generated by the following experiment.

	\begin{description}
	\item[PR1] Choose a random formula $\hat\PHI$. 
	\item[PR2] For each variable $x$ choose $\tilde\SIGMA(x)$ from $\mu_{\hat\Phi,x}^{(2\ell)}$	independently.
	\item[PR3] For every clause $a$ choose $\tilde\SIGMA(a)\in\{\pm1\}^k$ independently from $\mu_{\hat\Phi,a}^{(2\ell+1)}$.
	\item[PR4] Choose 	$\tilde\PHI$
		uniformly at random subject to the following conditions.
		\begin{itemize}
		\item If $(a,j)$ is a clause clone and $\partial_{\tilde\PHI}(a,j)=(x,i)$, then $\tilde\SIGMA(a,j)=\tilde\SIGMA(x)$.
		\item For all clause clones $(a,j)$ we have $\ppartial_{\tilde\PHI}^{4\w+1}(a,j)\ism\ppartial_{\hat\PHI}^{4\w+1}(a,j)$.
		\end{itemize}
		If no such $\tilde\PHI$ exists, start over from {\bf PR2}.
	\end{description}
The {\em planted replica model has the non-reconstruction property} if for any $\eps>0$ there is $\w>0$ such that
	\begin{align*}
	\lim_{n\to\infty}\pr\brk{\frac1n\sum_{x}\Bck{\abs{\mu_{\tilde\Phi,x}^{(2\ell)}(1)-\Bck{\SIGMA(x)|\nabla(\tilde\PHI,x,2\w)}_{\tilde\PHI,\beta}}}_{\tilde\PHI,\beta}<\eps}&=1
	\end{align*}

\begin{proposition}\label{Prop_nonRe2}
Assume that $d \in [d_-(k), \dplus]$ and  $\beta > \beta_-(k,d)$.
Then the planted replica model has the non-reconstruction property.
\end{proposition}

The non-reconstruction property enables us to determine $\Erw[\ln\cC_{\hat\PHI,\hat\SIGMA}(\beta)]$.
Indeed, given a map $\mu:x\mapsto \mu_x\in\cP(\{\pm1\})$ that assigns each variable a distribution on $\pm1$ we define the
{\em Bethe free energy} of a formula $\hat\Phi$ as
	\begin{align*}
	\cB_{\hat\Phi}(\mu)&=\sum_x(1-d)H(\mu_x)+\sum_a\brk{H(a)+\bck{\ln\psi_a}_{\mu_a}}.
	\end{align*}
Of course, $\mu_a$ is the extension of $\mu$ from variables to clauses as defined above. We also let $\cB_{\hat \Phi, \w} = \cB_{\hat \Phi}((\mu_{\hat \Phi, x}^{(2\w)})_{x \in V})$.
Then by combining \Prop s~\ref{Prop_nonRe1} and~\ref{Prop_nonRe2} with~\cite[\Thm s 4.4 and 4.5]{betheupb} we obtain the following.

\begin{corollary}\label{Cor_nonRe3}
We have
	$
	\lim_{n\to\infty}\frac1n\Erw[\ln\cC_{\hat\PHI,\hat\SIGMA}(\beta)]=\lim_{\w\to\infty}\lim_{n\to\infty}\frac1n\Erw[\cB_{\hat\PHI,\w}].
	$
\end{corollary}

Furthermore, $\cB_{\hat\PHI,\w}$ is determined by the {\em local} structure of $\hat\PHI$.
Since $\hat\PHI$, the local structure of the random formula can be described in terms of a random tree.
In fact, tracing the Belief Propagation algorithm on this random tree enables us to relate $\Erw[\cB_{\hat\PHI,\w}]$
to the distributional fixed point problem from \Prop~\ref{prop_unique_fixed_point_tree_1}.
The result of this, derived in \Sec~\ref{sec_operator_tree}, is

\begin{proposition}\label{Prop_BetheOnTrees}
We have
	$\lim_{\w\to\infty}\lim_{n\to\infty}\frac1n\Erw[\cB_{\hat\PHI,\w}]=\cB(k,d,\beta)$.
\end{proposition}

\noindent
Finally, \Prop~\ref{Prop_clusterSize} follows from \Cor~\ref{Cor_nonRe3} and \Prop~\ref{Prop_BetheOnTrees}.

\section{Belief Propagation on random trees} \label{sec_galton_watson}

{\em In the following of the paper we assume that $d \in [d_-(k), \dplus]$ and that $\beta \geq \beta_-(k,d)$. We let $c_\beta = 1- \exp(-\beta) \in (0,1)$.}

In this section, we introduce a Galton-Watson process on trees, that will describe the local neighborhood of randomly chosen vertices in random formulas but also allow to analyze the probabilistic fixed point problem in \Sec~\ref{sec_operator_tree}.

\subsection{A Galton-Watson process on trees}

We consider the following Galton-Watson process. We first observe that there is a unique $q=q(k,d,\beta)\in(0,1)$ such that
	\begin{align} \label{eq_def_q}
	1-(1-\exp(-\beta))q^k&=2(1-q).
	\end{align}
We start from the tree $T_{2\w}$ of depth $2\w$ ($\w \geq 0$) such that\begin{itemize}
\item Each node at an even depth of the tree is a variable node and has for offspring $d-1$ clause nodes.
\item Each node at an odd depth of the tree is a clause node and has for offspring $k-1$ variable nodes.
\end{itemize}
Let $V_{2\w}$ be the set of $T_{2\w}$ variable nodes, $F_{2\w}$ be the set of its clause nodes. Let $\partial V_{2\w}$ denote the subset of variables at distance $2\w$ from the root, and for node $v$ of $T_{2\w}$, let $\partial v$ (resp. $\partial_\downarrow v$) denote the set of neighbors (resp. children) of $v$ in $T_{2\w}$. We further decorate $T_{2\w}$ has follows. Each node $v \in V_{2\w} \cup F_{2\w}$ carries a number $b_{v, \uparrow} \in \{-1,1\}$ determined by the following process.
\begin{itemize}
\item[(i)] For the root $r$, we have $b_{r, \uparrow}=1$ with probability $q$ and $b_{r, \uparrow}=-1$ with probability $1-q$.
\item[(ii)] The offspring of a variable node $x$ with $b_{x, \uparrow} =\pm1$ is $\frac d2-1$ clause nodes $a$ with $b_{a, \uparrow} = \pm1$ and $\frac d2$ clause nodes $a$ such that $b_{a, \uparrow} \mp1$.
\item[(iii)] If a clause node $a$ is such that $b_{a, \uparrow} = -1$, the number of $x \in \partial_\downarrow a$ such that $b_{x, \uparrow} = -1$ has distribution $\Bin(k-1,1-q)$.
\item[(iv)] If a clause node $a$ is such that $b_{a, \uparrow} = 1$, then with probability $\exp(-\beta)q^{k-1}/(1-(1-\exp(-\beta))q^{k-1})$
			the offspring is $k-1$ variables $x$ with $b_{x,\uparrow}=1$, and otherwise the number of $x \in \partial_\downarrow a$ such that $b_{x, \uparrow} = -1$ has a conditionnal distribution distribution $\Bin_{\geq 1}(k-1,1-q)$.
.
\end{itemize}
Then we define for a clause $a$ and $x \in \partial a$, $b_{a,x} = b_{x, \uparrow}$ if $x \in \partial_\downarrow a$ and $b_{a,x} = b_{a, \uparrow}$ otherwise. We let $\partial_{\pm 1} a = \{x \in \partial a, b_{a,x} = \mp 1 \}$ and for a variable $x$, $\partial_{\pm 1} x = \{a \in \partial x, b_{a,x} = \mp 1 \}$. We finally let, for $0 \leq l \leq k$, $\partial_{\pm 1, l}x = \{a \in \partial_{\pm 1}x, |\{y \in \partial a \setminus \{x\}, b_{a,y}=1\}|=l\}$.

Let $\T=\T(d,k,\beta,2\w)$ be the resulting random (decorated) tree, let $p_{k,d,\beta}^{(2\w)}$ denote its distribution, and let $\cT_{2\w}$ denote the support of $p_{k,d,\beta}^{(2\w)}$. Similarly, we denote by $\vec \hT$ the random tree pending below the first clause adjacent to the root of $\vec T(d,k,\beta,2\w+2)$, by $\widehat{V}_{2\w+1}$ its set of variables and by $\partial \widehat{V}_{2\w+1}$ its set of variable at distance $2\w+1$ from the root, by $\widehat{p}_{k,d,\beta}^{(2\w+1)}$ its distribution, and by $\widehat{\cT}_{2\w+1}$ the support of $\widehat{p}_{k,d,\beta}^{(2\w+1)}$.

We call a sequence $\partial \nu \in \cP(\{-1,1\})^{\partial V_{2\w}}$ a {\em boundary condition} over $T \in \cT_{2\w}$. Similarly, for $\hT \in \widehat{\cT}_{2\w+1}$, we call a sequence $\partial \nu \in \cP(\{-1,1\})^{\partial \widehat{V}_{2\w+1}}$ a boundary condition on $\hT$.

We define the Belief Propagation messages induced by the boundary condition $\partial \nu$ on $T$ as the families $(\nu_{x, \uparrow}^{T,\partial \nu})_{x \in V_{2\w}}$ and $(\hnu_{a,\uparrow}^{T,\partial \nu})_{a \in F_{2\w}}$, where $\nu_{x,\uparrow}^{T,\partial \nu} = (\partial \nu)_x$ for $x \in \partial V_{2\w}$, and otherwise
\begin{align} \label{eq_def_BP_simple_tree_vertex} \nu_{x, \uparrow}^{T,\partial \nu}(s) & = \frac{ \prod_{a \in \partial_\downarrow x} \hnu_{a, \uparrow}^{T,\partial \nu}(s)}{\sum_{s' \in \{-1,1\}} \prod_{a \in \partial_\downarrow x} \hnu_{a, \uparrow}^{T,\partial \nu}(s') } \qquad \textrm{for $x \in V_{2\w} \setminus \partial V_{2\w}$ and $s \in \{-1,1\}$},
\\ \label{eq_def_BP_simple_tree_clause} \hnu_{a,\uparrow}^{T, \partial \nu}(s) & = \frac{ \sum_{s_a \in \{-1,1\}^{\partial a}}\vecone_{s_x =s}  \psi_{a,\beta}(s_a)  \prod_{y \in \partial_\downarrow a} \nu_{y, \uparrow}^{T,\partial \nu}(s_y)}{ \sum_{s_a \in \{-1,1\}^{\partial a}}  \psi_{a,\beta}(s_a)  \prod_{y \in \partial_\downarrow a} \nu_{y, \uparrow}^{T,\partial \nu}(s_y) }\qquad \textrm{for $a \in F_{2\w}$ and $s \in \{-1,1\}$} . \end{align}

In the following of this section, we let $\w$ large enough be fixed. We will be interested in showing that, under reasonnable assumptions, the message exiting a tree $T \in \cT_{2\w}$ only weakly depends on the boundary conditon $\partial \nu$. More precisely, we define $\partial\nu^{(0)} \in \cP(\{-1,1\})^{\partial V_{2\w}}$ by $\partial \nu^{(0)}_x(1) = 1= 1-\partial \nu^{(0)}_x(-1)$ for all $x \in \partial V_{2\w}$.
For $\hT \in \widehat{\cT}_{2\w+1}$ we define $\partial \nu^{ (0)} \in \cP(\{-1,1\})^{\partial \hT}$
 similarly. For a tree $T \in \cT_{2\w}$ with root $r$, and a boundary condition $\partial \nu$ on $T$, we denote by
$$ \nu_{T}^{\partial \nu} = \nu_{r,\uparrow}^{T, \partial \nu}, \qquad \nu_{T}^{(2\w)} = \nu_{r,\uparrow}^{T, \partial \nu^{(0)}}.$$
Similarly for $\hT \in \widehat{\cT}_{2\w+1}$ with root $r$ and a boundary condition $\partial \nu$ on $\hT$ we let
$$ \hnu_{T}^{\partial \nu} = \hnu_{r,\uparrow}^{T, \partial \nu}, \qquad \hnu_{T}^{(2\w+1)} = \hnu_{r,\uparrow}^{T, \partial \nu^{ (0)}}.$$

 We are now ready to state the main results of this section. In the following, we denote by $\vec T$ a random tree drawn from the distribution $p_{k,d,\beta}^{(2\w)}$, and by $\vec {\partial \nu}$ a random boundary condition, independent of $\vec T$ and  that satisfies the following condition.
 \begin{enumerate}
 \item[{\bf H}] For any $x \in \partial V_{2\w}, \ \mathbb{P} \left[ (\vec{ \partial \nu})_x(1) \leq 1-\exp(- k \beta / 2) \left | \left((\vec {\partial \nu})_y\right)_{y \in \partial V_{2\w} \setminus \{x\}} \right. \right] \leq 2^{-0.9 k}$
 \end{enumerate}
Similarly, we denote by $\vec \hT$ a random tree drawn from the distribution $\hat{p}_{k,d,\beta}^{(2\w+1)}$, and by $\vec {\partial \nu}$ a random boundary condition, independent of $\vec \hT$ and  that satisfies the following condition.
 \begin{enumerate}
 \item[{\bf H}] For any $x \in \partial V_{2\w+1}, \ \mathbb{P} \left[ (\vec{ \partial \nu})_x(1) \leq 1-\exp(- k \beta / 2) \left | \left((\vec {\partial \nu})_y\right)_{y \in \partial V_{2\w+1} \setminus \{x\}} \right. \right] \leq 2^{-0.9 k}$
 \end{enumerate}
 
 \begin{proposition} \label{prop_contraction_tree_tensor} We have 
 \begin{align*} &\Prob \left[ \| \nu^{\vec{\partial \nu}}_{\vec T} - \nu^{(2\w)}_{\vec T} \|_\infty \geq 2 \w^{-1} \right] \leq \w^{-1},
 \\&\Prob \left[ \| \hnu^{\vec{\partial \nu}}_{\vec \hT} - \hnu^{(2\w+1)}_{\vec \hT} \|_\infty \geq 2 k^2 \exp(2 \beta) \w^{-1} \right] \leq  k^2 \w^{-1}. \end{align*}  
  \end{proposition}
 
 The following variant of the proposition will follow from similar steps, and will prove usefull when analyzing random graphs in Sec~\ref{sec_marginal_analysis}-\ref{sec_typical_properties}. We first need to slightly generalize the process considered up to now.  Let $\textrm{GW}(k,d,\beta,2\w)$ denote the Galton-Watson process introduced considered up to now. Let $\textrm{GW}'(k,d,\beta,2\w)$ be the multi-type random process defined by the same rules as $\textrm{GW}(k,d,\beta,2\w)$, except for the first and second rule which are replaced by
\begin{enumerate}
\item[(i)'] The root $r$ has for offspring $\frac{d}{2}$ clauses nodes $a$ with $b_{a,\uparrow}=1$ and $\frac{d}{2}$ clauses nodes $a$ with $b_{a,\downarrow}=-1$.
\item[(ii)'] The offspring of a variable node $x$ {\em different from the root} with $b_{x, \uparrow} =\pm1$ is $\frac d2-1$ clause nodes $a$ with $b_{a, \uparrow} = \pm1$ and $\frac d2$ clause nodes $a$ such that $b_{a, \uparrow} \mp1$.
\end{enumerate}
Let $\widetilde{\cT}_{2\w}$ denote the set of trees generated by the process and by $\widetilde{p}_{k,d,\beta}^{(2\w)}$ the associated probability distribution. Let, in the following, $\vec{T}'$ denote a random tree drawn from this distribution. Let $\partial V_{2\w}'$ denote the variables ar distance $2\w$ from $\vec T'$ root (which is, as previously, a deterministic quantity). Let us call as before $\partial \nu \in \cP(\{-1,1\})^{\partial V_{2\w}'}$ a boundary condition over $\vec T'$, and let us extend condition ${\bf H}$ into 
 \begin{enumerate}
 \item[{\bf H'}] For any $x \in \partial V_{2\w}', \ \mathbb{P} \left[ (\vec{ \partial \nu})_x(1) \leq 1-\exp(- k \beta / 2) \left | \left((\vec {\partial \nu})_y\right)_{y \in \partial V_{2\w}' \setminus \{x\}} \right. \right] \leq 2^{-0.9 k}$
 \end{enumerate}

  \begin{proposition}\label{prop_good_trees_for_our_setting} Assume that $\vec {\partial \nu}''$ is a random boundary condition over $\partial V_{2\w}'$, independent of $\vec T'$ and that satisfies {\bf H'}. the following assumption. Further assume that the random boundary condition $\vec {\partial \nu'}$ (whose distribution may also depend on $\vec T'$) satisfies for all $x \in \partial V_{2\w}'$, $\vec {\partial \nu'}_x(1) \geq \vec {\partial \nu}''_x(1)$. Then
  $$\Prob \left[ \| \mu^{\vec{\partial \nu'}}_{\vec T'} - \mu^{(2\w)}_{\vec T'} \|_\infty \geq k \exp(\beta) \w^{-1} \right] \leq \w^{-1}.$$ \end{proposition}
 
  \subsection{The (random) trunk of random trees : proof of \Prop~\ref{prop_contraction_tree_tensor}} 
  \label{sec_proof_prop_contraction_tree_tensor}
  
  In order to prove \Prop~\ref{prop_contraction_tree_tensor}, we shall identify a concrete condition on $(T, \partial \nu)$ under which the message $\nu_{T}^{\partial \nu}$ is close to $\nu_{T}^{(2\w)}$. We define the {\em Trunk} of $T$ under the boundary condition $\partial \nu$, $\trunk(T,\partial \nu)$, as the largest subset $W$ of $V_{2\w}$ such that for any $x \in W$ either
  \begin{description} 
  \item[TR0] $x \in \partial V_{2\w}$ and $\partial \nu_x(1) \geq 1- \exp(-k \beta /2 )$
  \end{description}
  or the five following conditions hold
  \begin{description}
  \item[TR1] there are at least $ \lfloor 0.9 k \rfloor$ clauses $a\in\partial_\downarrow x$ such that $\partial_1a=\{x\}$.
\item[TR2] there are no more than $ \lceil 0.1 k \rceil$ clauses $a\in\partial x$ such that $|\partial_{-1}a|=k$.
\item[TR3] for any $1\leq l \leq k$ the number of $a\in\partial_{-1} x$ such that $|\partial_1a| = l$ is bounded by $ k^{l+3} / l!\enspace$.
\item[TR4] there are no more than $k^{3/4}$ clauses $a\in\partial_{1}x$ such that $|\partial_1a|=1$ but $\partial a\not\subset W$.
\item[TR5]  there are no more than $k^{3/4}$ clauses $a\in\partial_{-1}x$ such that $|\partial_{-1}a|<k$ and $|\partial_1 a \setminus W| \geq |\partial_1 a|/4$.\end{description}

We will first observe that the following is true. 
 \begin{lemma} \label{lemma_most_likely_trunk}
  We have $$ \Prob \left[\textrm{the root of } \vec T \textrm{ under the boundary condition } \vec {\partial \nu} \textrm{ is cold} \right] \geq 1-\w^{-1}.$$
  \end{lemma}
  \begin{proof} The lemma is easily proved by induction. \end{proof}
   
 Further, we need to introduce the following definitions. For $T \in \cT_{2\w}$ with root $r$ and $x \in \partial V_{2\w}$, we denote by $[x \to r]$ the unique shortest path from $x$ to $r$ in $T$. 
  \begin{enumerate}
  \item[(i)]
  We say that a factor node $a \in F_{2\w}$ is {\em cold} if and only if $\partial_1 a \cap \trunk(T, \partial {\nu}) \neq \emptyset$.
  \item[(ii)] We say that a variable node $x \in V_{2\w}$ is {\em cold} if $x \in \trunk(T,  {\partial \nu})$.
  \item[(iii)] We say that $(x,a)$ (with $x \in \partial_\downarrow a$) is {\em cold} if $x$ is cold or $a$ is cold.
\item[(iv)]  We say that a path $[x \to r]$ ($x \in \partial V_{2\w}$) is {\em cold} if it contains at least $\lfloor 0.4 \w \rfloor$ cold pairs $(x,a)$.
 \item[(v)] Finally, we say that the pair $(T, \partial \nu) \in \cT_{2\w} \times \cP(\{-1,1\})^{\partial V_{2 \w}}$ is {\em cold} if all the paths $[x \to r]$, with $x \in \partial V_{2\w}$, are cold.
  \end{enumerate}
  
  The key result of this section is the following estimate.
  \begin{proposition} \label{prop_cold_trees_tensor_a}
  We have $$ \Prob \left[(\vec T, \vec {\partial \nu}) \textrm{ is cold} \right] \geq 1-\w^{-1}.$$
  \end{proposition}
  
  Then in \Sec~\ref{sec_proof_prop_cold_trees_tensor_b} we shall prove the following.
  \begin{proposition} \label{prop_cold_trees_tensor_b}
  If $(T, \partial \nu) \in \cT_{2\w} \times \cP(\{-1,1\})^{\partial V_{2\w}}$ is cold, then 
  $$ \| \nu^{\partial \nu}_{T,\uparrow} - \nu^{(2\w)}_{T,\uparrow} \|_\infty \leq \w^{-1} .$$
  \end{proposition}
  Let us see how this implies \Prop~\ref{prop_contraction_tree_tensor}.
  
  \begin{proof}[Proof of \Prop~\ref{prop_contraction_tree_tensor}] The first part of the proposition directly follows from the combination of \Prop~\ref{prop_cold_trees_tensor_a} and \Prop~\ref{prop_cold_trees_tensor_b}. For the second part of the proposition, we first need to introduce one more notation. For $\hT \in \widehat{\cT}_{2\w+1}$ and $j \in [k-1]$ we denote by $\hT[j] \in \cT_{2\w}$ the subtree of $\hT$ pending below the $j$-th neighbor of the root. For a boundary conditon $\partial \nu$ over $\partial \hT$, we denote by $\partial \nu[j]$ its restriction to $\hT[j]$.
  
  We observe that if $\partial \nu$ satisfies {\bf H}, then so does $\partial \nu[j]$. Moreover, for any $T \in \cT_{2\w}$ and $j \in [k-1]$, we have by definition of the Galton-Watson process
  $$ \mathbb{P} \left[ \vec \hT[j] = T \right] \leq 2 \mathbb{P} \left[ \vec T = T\right].$$
  For $\hT \in \widehat{\cT}_{2\w+1}$ and a boundary condition $\partial \nu$, using (\ref{eq_def_BP_simple_tree_clause}) and Taylor's theorem, we get
  $$ \|\hnu^{{\partial \nu}}_{ \hT} - \hnu^{(2\w+1)}_{ \hT}  \|_\infty \leq 8k \exp(2 \beta) \sup_{j \in [k-1]} \| \nu_{\hT[j]}^{\partial \nu[j]} -  \nu_{\hT[j]}^{(2\w)} \|_\infty  .  $$
  Thereby, we obtain, using the previous observations
  \begin{align*} \mathbb{P} \left[  \|\hnu^{\vec {\partial \nu}}_{ \vec \hT} - \hnu^{(2\w+1)}_{ \vec \hT}  \|_\infty \geq k^2 \exp(2 \beta) \w^{-1} \right] &\leq   \mathbb{P} \left[ \sup_{j \in [k-1]} \| \nu_{\vec \hT[j]}^{\vec{\partial \nu}[j]}- \nu_{\vec \hT[j]}^{(2\w)} \|_\infty \geq \w^{-1}  \right]  
  \\ &\leq  (k-1)  \mathbb{P} \left[  \| \nu_{\vec {\hT}[1]}^{\vec{\partial \nu}[1]} -\nu_{\vec {\hT}[1]}^{(2\w)} \|_\infty \geq \w^{-1}  \right]
  \\ &\leq  2(k-1)  \mathbb{P} \left[  \| \nu_{\vec {T}}^{\vec{\partial \nu[1]}} - \nu_{\vec {T}}^{(2\w)} \|_\infty \geq \w^{-1}  \right]. \end{align*}
  For $\w$ large enough, and using the first part of the proposition that we already proved, we have 
  $$\mathbb{P} \left[  \| \nu_{\vec {T'}}^{\vec{\partial \nu[1]}} - \nu_{\vec {T'}}^{(2\w)} \|_\infty \geq (\w-1)^{-1}  \right] \leq \w^{-1}.$$
  The second part of the proposition follows.
  \end{proof}
   
  \begin{proof}[Proof of \Prop~\ref{prop_good_trees_for_our_setting}] For $T \in \widetilde{\cT}_{2\w+2}$, let $T'$ be the tree obtained by removing the tree pending below the last children of $T$'s root. Then if $(T',\vec {\partial \nu}'')$ is good, then so is $(T', \vec {\partial \nu'})$. In particular 
  $$ \mathbb{P} \left[ (\vec{T'},\vec {\partial \nu'}) \textrm{ is cold } \right] \geq \mathbb{P} \left[ (\vec{T'},\vec {\partial \nu}) \textrm{ is cold } \right]  \geq 1- \w^{-1} .$$ In this case we have
  $$  \| \nu^{\vec{\partial \nu}}_{\vec T'} - \nu^{(2\w+1)}_{\vec T'} \|_\infty \leq \w^{-1}$$ and moreover, applying Taylor's theorem and by a similar token as previously,
  $$\| \mu^{\vec{\partial \nu'}}_{\vec T} - \mu^{(2\w+2)}_{\vec T} \|_\infty \leq k \exp( \beta) \w^{-1}.$$
  This concludes the proof of the proposition.
  
 \end{proof}
  
  \subsection{Proof of \Prop~\ref{prop_cold_trees_tensor_b}} \label{sec_proof_prop_cold_trees_tensor_b}
  
  We begin with the following lemma, that shows that the messages exiting vertices $x \in \trunk(T, \partial \nu)$ are under tight control.
  
  \begin{lemma} \label{lemma_if_frozen_then_cold} Let $(T, \partial \nu) \in \cT_{2\w} \times \cP(\{-1,1\})^{\partial V_{2\w}}$ be fixed. For all $x \in \trunk(T,\partial \nu)$, we have 
  $$\nu_{x, \uparrow}^{T,\partial \nu}(1) > 1- \exp(- k \beta /2 ) \qquad \textrm{and} \qquad \nu_{x, \uparrow}^{T, \partial \nu^{(0)}}(1) > 1-\exp(- k \beta /2).$$
 \end{lemma}
 
 The lemma will rely on a detailled analysis of terms of the form $\frac{\hnu^{T,\partial \nu}_{x,\uparrow}(1)}{\hnu^{T,\partial \nu}_{x,\uparrow}(-1)}$. In order to simplify the discussion, we shall isolate this analysis in the following lemma.
 
\begin{lemma} \label{lemma_estimate_ratio_generic} Let a clause $a$ be fixed along with its adjacents variables $x \in \partial a$ and a family $(\nu_{x \to a})_{x \in \partial a} \in \cP(\{-1,1\})^{k}$. Let 
$$\partial_{\rm good} a  = \left \{x \in \partial a, \nu_{x \to a}(1) \geq 1- \exp(-k \beta/2)  \right \}.$$
Let $(\hnu_{a \to x})_{x \in \partial a} \in \cP(\{-1,1\})^{k}$ be defined by the following equations.
\begin{align} \label{eq_crux_hnua} \hnu_{a \to x}(s) = \frac{ \sum_{s_a \in \{-1,1\}^k} \vecone_{s_x =s}  \psi_{a,\beta}(s_a) \prod_{y \in \partial a \setminus \{x\})} \nu_{y \to a}(s_y) }{ \sum_{s_a \in \{-1,1\}^k}  \psi_{a,\beta}(s_a) \prod_{y \in \partial a \setminus \{x\}} \nu_{y \to a}(s_y)} \end{align}
Then for $x \in \partial a$, the following estimates hold true.
\begin{enumerate}
\item[{\bf (a)}]$$ \exp(- \beta) \leq \frac{\hnu_{a \to x}(1)}{\hnu_{a \to x}(-1)} \leq \exp( \beta).$$
\item[{\bf (b)}] If $x \in \partial_1 a$, then  $$ \frac{\hnu_{a \to x}(1)}{\hnu_{a \to x}(-1)} \geq 1 .$$
\item[{\bf (c)}] If $\{x\}= \partial_1 a$ and $\partial_{-1} a \subset \partial_{\rm good} a$, then  $$ \frac{\hnu_{a \to x}(1)}{\hnu_{a \to x}(-1)} \geq \exp(0.99 \beta ) .$$
\item[{\bf (d)}] If $|\partial_1(a) \setminus \{x\} \cap \partial_{\rm good} a | \geq p$, then $$ \frac{\hnu_{a \to x}(1)}{\hnu_{a \to x}(-1)} \geq \exp(- \exp(- p k \beta /3))  .$$
\end{enumerate}  \end{lemma}
\begin{proof} Point (a) easily follows from the fact that for any $s_a \in \{-1,1\}^{\partial a}$, $\exp(-\beta) \leq \psi_{a,\beta}(s) \leq 1$, and that if there is $x \in \partial a$ such that $s_x =1$, then $\psi_{a,\beta}(s_a)=1$.

Point (b) follows from the observation that if $s_a, s'_a \in \{-1,1\}^{\partial a}$ satisfy $s_x =1$ and $s_y = s_y'$ for $y \in \partial a \setminus \{x\}$, then $\psi_{a,\beta}(s_a) \geq \psi_{a,\beta}(s'_a)$.

Point (c) follows from the observation that, if $\{x \} = \partial_1 a$ and $\partial_{-1} a \subset \partial_{\rm good} a$,
\begin{align*} \exp(- \vecone_{s \neq 1} \beta) &\leq \sum_{s_a \in \{-1,1\}^{\partial a}} \vecone_{s_x =s} \psi_{a,\beta}(s_a) \prod_{y \in \partial a \setminus \{x\}} \nu_{y \to a}(s_y) \leq \exp(- \vecone_{s \neq 1} \beta) + 2k \exp(-k \beta /2 ). \end{align*}

Finally, point (d) is obtained by observing that, if $|\partial_1 a \setminus \{x\} \cap \partial_{\rm good} a | \geq p$,
\begin{align*} \left |  \sum_{s_a \in \{-1,1\}^{\partial a}} \vecone_{s_x =s} \psi_{a,\beta}(s_a) \prod_{y \in \partial a \setminus \{x\}} \nu_{y \to a}(s_y) - 1 \right| &\leq \prod_{y \in \partial_1a \setminus \{x\}} \left(1 - \nu_{y \to a}(1) \right) \leq \exp(-p k \beta/2). \end{align*} \end{proof}
 
 \begin{proof}[Proof of \Lem~\ref{lemma_if_frozen_then_cold}] We first prove the statement concerning $\nu^{T,\partial \nu}_{\cdot, \uparrow}$. We prove it by induction over $t = \w - \frac{\textrm{dist}(x,r)}{2}$. For $t=0$ the result holds by definition of $\trunk(T, \partial \nu)$. Assume that the results hold for all $x \in V_{2\w}$ such that $\textrm{dist}(x,r) \geq 2(\w-t)$ and let $x \in V_{2\w}$ with $\textrm{dist}(x,r) = 2(\w-t-1)$ be fixed. We define 
 \begin{align*}
 \Delta_1 x &= \partial_\downarrow x \cap \partial_1 x,
 \\ \Delta_{1,0} x &= \left \{ a \in \partial_{\downarrow} x,\ \partial_{-1} a \cap \trunk(T,\partial \nu) = \partial_{\downarrow} a  \right \},
 \\ \Delta_{-1,0} x &= \left\{ a \in \partial_{\downarrow} x,\ a \in \partial_{-1,0} x \right \},
 \\ \textrm{ and for $1 \leq l \leq k$,} \qquad \Delta_{-1,l} x &= \left \{ a \in \partial_{-1}(x,a), | \partial_1(x) | = l, \ |\partial_1(x) \setminus \trunk(T, \partial \nu) | \leq |\partial_1(x)| /4 \right\}
 \end{align*}
 We have, by Eq. (\ref{eq_def_BP_simple_tree_vertex})
 \begin{align} \nonumber \frac{\nu_{x,\uparrow}^{T,\partial \nu}(1)}{\nu_{x,\uparrow}^{T,\partial \nu}(-1)}  =& \prod_{a \in \Delta_{1,0} x} \frac{\hnu^{T,\partial \nu}_{a, \uparrow}(1)}{\hnu^{T,\partial \nu}_{a, \uparrow}(-1)} \prod_{a \in \Delta_1 x \setminus \Delta_{1,0}x} \frac{\hnu^{T,\partial \nu}_{a, \uparrow}(1)}{\hnu^{T,\partial \nu}_{a, \uparrow}(-1)}    \prod_{l=1}^k \prod_{a \in \Delta_{-1,l} x} \frac{\hnu^{T,\partial \nu}_{a, \uparrow}(1)}{\hnu^{T,\partial \nu}_{a, \uparrow}(-1)} \prod_{a \in \Delta_{-1,0} x} \frac{\hnu^{T,\partial \nu}_{a, \uparrow}(1)}{\hnu^{T,\partial \nu}_{a, \uparrow}(-1)}
 \\ \label{eq_aux_ratio_messages_x_uparrow} & \prod_{a \in \partial_{\downarrow} x \setminus (\Delta_1 x \cup_{l=0}^k \Delta_{-1,l}x)} \frac{\hnu^{T,\partial \nu}_{a, \uparrow}(1)}{\hnu^{T,\partial \nu}_{a, \uparrow}(-1)} . \end{align}
 Because the messages $(\nu, \hnu)$ satisfy Eq. (\ref{eq_crux_hnua}), we can apply the result of \Lem~\ref{lemma_estimate_ratio_generic} to them. It follows that
 \begin{align*} \frac{\nu_{x,\uparrow}^{T,\partial \nu}(1)}{\nu_{x,\uparrow}^{T,\partial \nu}(-1)}  &\geq \exp \left( \beta \left[ 0.99  |\Delta_{1,0} x |  - | \Delta_{-1,0} x | \right] \right) \exp \left( \sum_{l=1}^k | \Delta_{-1,l}x | \exp(-k \beta l /4 ) \right) 
 \\ & \hspace{2 cm} \exp \left (-2 \beta |\partial_\downarrow x \setminus (\Delta_{1} x \cup (\cup_{l=0}^k \Delta_{-1,l} x)) | \right).\end{align*}
 By definition of $\trunk(T,\partial \nu)$, we have
 \begin{align*} |\Delta_{1,0} x | \geq  \lfloor 0.9 k  \rfloor, \qquad &|\Delta_{-1,0} x | \leq \lceil 0.1 k \rceil, \qquad |\partial_\downarrow x \setminus (\Delta_{1} x \cup (\cup_{l=0}^k \Delta_{-1,l} x )) | \leq  k^{3/4}. \end{align*}
 Moreover, for $1 \leq l \leq k$ we have $|\Delta_{-1,l} x| \leq |\partial_{-1,l} x| \leq  k^{l+3}/l!$ (by {\bf TR3}) and hereby
 \begin{align*} \sum_{l=1}^k | \Delta_{-1,l} x | \exp(-k \beta l /4 ) &\leq \sum_{l=1}^k \frac{ k^{l+3}}{l!} \exp(-k \beta l  / 4 ) \leq k^3 \exp \left(   k \exp(-k \beta / 4 ) \right) \leq 2 . \end{align*}
 Replacing with these four estimate in (\ref{eq_aux_ratio_messages_x_uparrow}), we obtain $\frac{\nu_{x,\uparrow}^{T,\partial \nu}(1)}{\nu_{x,\uparrow}^{T,\partial \nu}(-1)} \geq \exp(  2 k \beta /3 )$, as desired.
  
  The second part of the lemma, regarding $\nu_{x,\uparrow}^{(2\w)}$, follows from the observation that $\trunk(T,\partial \nu) \subset \trunk(T,\partial \nu^{(0)})$.
    \end{proof}
  
\begin{proof}[Proof of \Prop~\ref{prop_cold_trees_tensor_b}] 
Let $x \in V_{2\w} \setminus \partial V_{2\w}$ be fixed, and let for $a \in \partial_{\downarrow} x$ and a boundary condition $\partial \nu$, $\hve_{a, \uparrow}^{\partial \nu}: \{-1,1\} \to \mathbb{R}$ be defined by, for $s \in \{-1,1\}$
\begin{align} \hve_{a, \uparrow}^{T,\partial \nu}(s) &=\sum_{s_a \in \{-1,1\}^{\partial a}}  \vecone_{s_x =s} \psi_{a,\beta}(s_a) \prod_{y \in \partial_\downarrow a} \nu_{y, \uparrow}^{T,\partial \nu}(s_y).\end{align}
By applying Taylor's theorem to equation (\ref{eq_def_BP_simple_tree_vertex}), observing that $\nu_{x, \uparrow}^{T,\partial \nu}(1) = \left(1 + \prod_{a \in \partial_\downarrow x} \frac{\hve_{a,\uparrow}^{T,\partial \nu}(-1) }{\hve_{a,\uparrow}^{T,\partial \nu}(1)} \right)^{-1}$ we obtain
\begin{align}  \| \nu_{x, \uparrow}^{T,\partial \nu}-\nu_{x, \uparrow}^{T,\partial \nu^{(0)}} \|_\infty  \leq  \sup_{u \in [0,1]} &\frac{\left(  \frac{u \nu_{x,\uparrow}^{T,\partial \nu}(-1) + (1-u) \nu_{x,\uparrow}^{T,\partial \nu^{(0)}}(-1)}{u \nu_{x,\uparrow}^{T,\partial \nu}(1) + (1-u) \nu_{x,\uparrow}^{T,\partial \nu^{(0)}}(1)} \right)}{\left( 1+ \frac{u \nu_{x,\uparrow}^{T,\partial \nu}(-1) + (1-u) \nu_{x,\uparrow}^{T,\partial \nu^{(0)}}(-1)}{u \nu_{x,\uparrow}^{T,\partial \nu}(1) + (1-u) \nu_{x,\uparrow}^{T,\partial \nu^{(0)}}(1)} \right)^2} 
  \label{eq_contraction_tree_a} \sum_{a \in \partial_\downarrow x} \sup_{u \in [0,1]} \left \| \frac{ \hve^{T,\partial \nu}_{a, \uparrow} - \hve^{T,\partial \nu^{(0)}}_{a, \uparrow}} {u \hve_{a, \uparrow}^{T,\partial \nu} + (1-u) \hve_{a, \uparrow}^{T,\partial \nu^{(0)}}} \right \|_\infty . \end{align}
We observe that for $(x_1, x_2, x_3,x_4) \in [0,\infty)^4$ we have
\begin{align*}& \frac{(x_1+x_2+x_3)}{(1+x_1+x_2+x_3)^2} \leq \min \{ x_1 + x_2+x_3, (x_1 + x_2 + x_3)^{-1} \} \leq 1,
\\ &\sup_{u \in [0,1]} \frac{ux_1 + (1-u)x_2}{ux_3 + (1-u)x_4} \leq \frac{\max\{x_1,x_2\}}{\min\{x_3,x_4\}}. \end{align*}
Using this in cunjunction with \Lem~\ref{lemma_if_frozen_then_cold} we obtain
\beq \label{eq_contraction_tree_b}  \sup_{u \in [0,1]} \frac{\left( \frac{u \nu_{x,\uparrow}^{T,\partial \nu} (-1) + (1-u) \nu_{x,\uparrow}^{T,\partial \nu^{(0)}}(-1)}{u \nu_{x,\uparrow}^{T,\partial \nu}(1) + (1-u) \nu_{x,\uparrow}^{T,\partial \nu^{(0)}}(1)} \right)}{\left( 1+ \frac{u \nu_{x,\uparrow}^{T,\partial \nu}(-1) + (1-u) \nu_{x,\uparrow}^{T,\partial \nu^{(0)}}(-1)}{u \nu_{x,\uparrow}^{T,\partial \nu}(1) + (1-u) \nu_{x,\uparrow}^{T,\partial \nu^{(0)}}(1)} \right)^2}  \leq 6 \exp\left(-k \beta \vecone_{x \textrm{ is cold}} /2 \right). \eeq
We further observe that
$$ \sup_{u \in [0,1]} \left \| \frac{ \hve^{T,\partial \nu}_{a, \uparrow} - \hve^{T,\partial \nu^{(0)}}_{a, \uparrow}} {u \hve_{a, \uparrow}^{T,\partial \nu} + (1-u) \hve_{a, \uparrow}^{T,\partial \nu^{(0)}}} \right \|_\infty \leq \exp(\beta) \left \| { \hve^{T,\partial \nu}_{a, \uparrow} - \hve^{T,\partial \nu^{(0)}}_{a, \uparrow}}  \right \|_\infty.$$
In particular,
$$ \| \nu_{x, \uparrow}^{T,\partial \nu}-\nu_{x, \uparrow}^{T,\partial \nu_\tensor^{(0)}} \|_\infty \leq 6 \exp(\beta) \exp(-  k \beta \vecone_{x \textrm{ is cold}} /2 ). $$
Using again Taylor's theorem, for any $a \in F_{2\w}$ we have
\beq   \left \| { \hve^{T,\partial \nu}_{a, \uparrow} - \hve^{T,\partial \nu^{(0)}}_{a, \uparrow}}  \right \|_\infty \leq  4 \exp(-k \beta/2) \sum_{z \in \partial_\downarrow a} \min_{ \substack{y \in \partial_{\downarrow} a \cap \partial_{1} a \\ y \neq z} }  \max \left \{ \nu^{T,\partial \nu}_{y,\uparrow}  (-1),\nu^{T,\partial \nu^{(0)}}_{y,\uparrow}  (-1) \right \}  \| \nu^{T,\partial \nu}_{z,\uparrow} - \nu^{T,\partial \nu^{(0)}}_{z,\uparrow} \|_\infty. \eeq
Therefore, if $a$ is cold and the unique $x$ such that $\delta_{\uparrow} a = \{x\}$ is not cold, we have
\begin{align} \left \| { \hve^{T,\partial \nu}_{a, \uparrow} - \hve^{T,\partial \nu^{(0)}}_{a, \uparrow}}  \right \|_\infty &\leq  \exp( -k \beta/2) \sum_{x \in \partial_\downarrow a} \vecone_{x \textrm{ not cold}}  \| \nu^{T,\partial \nu}_{x,\uparrow} - \nu^{T,\partial \nu^{(0)}}_{x,\uparrow} \|_\infty \label{eq_contraction_tree_2_d} + \sum_{x \in \partial_\downarrow a} \vecone_{x \textrm{ is cold}} \| \nu^{T,\partial \nu}_{x,\uparrow} - \nu^{T,\partial \nu^{(0)}}_{x,\uparrow} \|_\infty. \end{align}
Combining (\ref{eq_contraction_tree_b}) with (\ref{eq_contraction_tree_2_d}), we obtain 
\begin{align*} \left \| \nu_{x, \uparrow}^{T,\partial \nu} - \nu_{x,\uparrow}^{T,\partial \nu^{(0)}} \right \|_\infty \leq 24 \exp(\beta) \sum_{a \in \partial_\downarrow x} \sum_{y \in \partial_\downarrow a} \exp \left(-k \beta (\vecone_{x \textrm{ is cold}}+\vecone_{x \textrm{ is not cold}} \vecone_{a \textrm{ is cold}} \vecone_{y \textrm{ is not cold}}) /2 \right) \left \| \nu_{y, \uparrow}^{T,\partial \nu} - \nu_{y,\uparrow}^{T,\partial \nu^{(0)}} \right \|_\infty. \end{align*}
Iterating this equation, we obtain, (using that for any $x \in \partial V_{2\w}$, the path from the root $r$ of $T$ to $x$ contains at least $\lfloor 0.4 \w \rfloor $ cold pairs $(x,a)$)
\begin{align*} \left \| \nu_{T, \uparrow}^{\partial \nu} - \nu_{T,\uparrow}^{(2\w)} \right \|_\infty &\leq 24^{\w} \exp(\beta \w) \sum_{x \in \partial V_{2\w}} \exp \left(-k \beta \lfloor 0.4 \w \rfloor /2 \right) \left \| \partial \nu_{x} -  \partial \nu_{x}^{(0)}  \right \|_\infty
\\ &\leq |\partial V_{2\w}| 24^\w \exp(\beta \w) \exp \left( - k \beta  \lfloor 0.4 \w \rfloor  /2 \right)
\\ &\leq (dk)^\w 24^\w \exp \left( - 0.01 k^2 \w \right) = o_\w(1).\end{align*}

 \end{proof}
 
 \subsection{Proof of \Prop~\ref{prop_cold_trees_tensor_a}}

\begin{proof} In order to prove the proposition, we will need to slightly extend the notion of cold variables and cold clauses. Let a pair $(T, \partial \nu) \in \cT_{2\w} \times \cP(\{-1,1\})^{d_\w}$ be fixed. Given $v \in T$ and $\{w\} = \partial_\uparrow v$, let $T_v$ denote the subtree of $T \setminus \{w\}$ rooted at $v$. Also recall that we denoted by $\partial V_{2\w} = \partial T$.

For $(a,x)$ with $\{x\} = \partial_\uparrow a$, we say that $x$ is strongly cold with respect to $a$ for the pair $(T,\partial \nu)$ if there exists no tree $T' \in \cT_{2\w}$ and no boundary condition $\partial \nu'$ over $\partial V_{2\w}$ such that the following is true. \begin{itemize}
 \item $x$ is not cold in $(T', \partial \nu')$,
 \item $T'_x = T_x$,
 \item $\forall x \in \partial  T_x \setminus \partial  T_{a}$, $(\partial \nu')_x = (\partial \nu)_x$.
 \end{itemize}
Observe that strongly cold variables are also cold. Let, for $a' \in \partial_{\downarrow} x$, ${p}_{x,a'}$ be the probability that $x$ is not strongly cold with respect to $a'$ when the pair $(\vec T, \vec {\partial \nu})$ is drawn at random. We shall prove by induction over $t \in \{1, \dots,\w\}$ that for $x$ at distance $2t$ from $\partial V_{2\w}$ and $a' \in \partial \downarrow x$, $p_{x,a'} \leq 2^{-0.9k}$. For $t=0$, the result follows from the assumption on the distribution of $\vec {\partial \nu}$. Let us now assume that we have proved the result up to $t \geq 0$ and consider $x$ at distance $2(t+1)$ from $\partial V_{2\w}$ and $a' \in \partial_\downarrow x$. For $x$ not to be strongly cold with respect to $a'$, one of the following must happen. Let ${V}_{\rm{cold}}^{(t)}$ be the set of strongly cold variables at distance $t$ from $\partial V_{2\w}$. Let $a$ be the clause such that $\partial_\uparrow x = \{a\}$.
  \begin{itemize}
  \item[(a)] there are less than $\lfloor 0.95 k \rfloor$ clauses $a \in \partial_{1} x$ such that $\partial_1 a = \{x\}$,
  \item[(b)] there are more than $\lceil 0.05 k \rceil$ clauses $a \in \partial x$ such that $|\partial_{-1} a |= k$,
  \item[(c)] there is $1 \leq l \leq k$ such that there are more than $0.5 k^{l+3}/l!$ clauses $a \in \partial_{-1}x$ with $|\partial_1 a| = l$,
 \item[(d)] $\left| \{b \in \partial_{1,0}x \setminus \{a\}, \partial b \setminus \{x\} \not \subset V_{\rm cold}^{(t-1)} \}  \right|  \geq  k^{3/4}-1$,
 \item[(e)] $| \{ b \in \partial_{-1}x,\  |\partial_{-1}b | \leq k, \  |\partial_{1} b \setminus \{x\} \setminus {V}_{\rm cold}^{(t-1)}| \geq | \partial_1 b|/4 \}| \geq k^{3/4}-1$.
 \end{itemize}
By definition of our random process, (a), (b) and (c) each hold with probability at most $2^{-0.95 k}$. The probability that a given clause $b \in \partial_{1, 0}x \setminus \{a\}$ contains at least one not strongly cold variable different from $x$ is $\sum_{y \in \partial \downarrow b} p_{y,b}+ \tilde O_k(2^{-k}) \leq k 2^{-0.9 k}$. By definition, the probability that each of the clauses $b_1, \dots, b_y \in \partial_{1,0}x \setminus \{a\}$ contain at least one not strongly cold variable different from $x$ is upperbounded by ${|\partial_{1,0} x \setminus \{a\}| \choose y} (k2^{-0.9k})^y$. Therefore, the probability that (d) happens is at most $2^{-1.5 k}$. Similarly, (e) happen with probability at most $2^{-1.5 k}$. Therefore we obtain $p_{x,a'} \leq 3. 2^{-0.95k} + 2. 2^{-1.5k} \leq  2^{-0.9k}$, as needed.

Let a clause $a \in F_{2\w}$ be fixed as well as $x \in \partial_\downarrow a$. Let $\partial_\uparrow a = \{y\}$. We say that $a$ is strongly cold with respect to $x$ if $|\partial_1 a \setminus \{x,y\} \cap {V}_{\rm cold}^{(t)}| \geq 1$. Let $q_{a,x}$ denote the probability that $a$ is not strongly cold with respect to $x \in \partial_\downarrow a$ when the pair $(\vec T,\vec {\partial \nu})$ is drawn from a distribution that satisfies the hypothesis of the proposition. By our previous estimate (g) and (h) happen with probability at most $2^{-1.5 k}$ while (f) happens with probability at most $2^{-0.9 k}$. In particular, $q_{a,x}\leq 2^{-0.9k}$. By construction of the strongly cold clauses, the probability that a pair $(x,a)$ with $\{a\} = \partial_{\uparrow} x$ is not strongly cold with respect to $b \in \partial \downarrow x$ is then upperbounded by $p_{x,b} q_{a,x} \leq 2^{-1.7 k}$.

Let $x \in \partial V_{2\w}$ be fixed. Recall that we denoted by $r$ the root of $T$. We denote sequence of variables and clauses on the path $[x \to r]$ by $(x_0=x,a_0, x_1, a_1, \dots, r)$. For the path $[x \to r]$ not to be cold, there must be at least $\lfloor 0.6 \w \rfloor$ pairs $(x_i,a_i)$ along this path that are not strongly cold. Moreover, for $i_1 \neq i_2 \neq \dots \neq i_l$, the probability that the $(x_{i_j},a_{i_j})$ are strongly cold is independent (by construction). Therefore we obtain
 \begin{align*} \mathbb{P} \left[ [x \to r] \textrm{ not cold} \right] &\leq \sum_{l \geq \lfloor 0.6 \w \rfloor} \sum_{0 \leq i_1 \leq \dots \leq i_l < \w} \mathbb{P} \left[ (x_{i_1},a_{i_1}) \textrm{ not strongly cold and } (x_{i_2},a_{i_2}) \textrm{ not strongly cold and } \right.
 \\ & \hspace{4 cm} \left. \dots \textrm{ and }  (x_{i_l},a_{i_l}) \textrm{ not strongly cold}\right] 
 \\ & \leq \sum_{l \geq \lfloor 0.6 \w \rfloor} \sum_{0 \leq i_1 \leq \dots \leq i_l < \w} \prod_{j=1}^l \mathbb{P} \left[ (x_{i_j},a_{i_j})   \textrm{ not strongly cold} \right]
 \\ &\leq 2^\w \left( 2^{-1.7k}\right)^{\lfloor 0.6 \w \rfloor} \leq 2^{-1.02 k \w}. \end{align*}
Consequently, we obtain with the union bound
 \begin{align*} \mathbb{P} \left[ ( \vec T, \vec {\partial \nu}) \textrm{ is not cold} \right] \leq \sum_{x \in \partial V_{2\w}}\mathbb{P} \left[ [x \to r] \textrm{ not cold} \right] \leq | \partial V_{2\w}| 2^{-1.02 k \w} \leq (dk)^\w 2^{-1.02 k \w} =o_\w(1). \end{align*}
 
 \end{proof}

\section{The fixed point problem on trees}
\label{sec_operator_tree}

{\em In this section we prove \Prop~\ref{prop_unique_fixed_point_tree_1} and \Lem~\ref{lemma_total_size}.}

We shall obtain the propositions by making the connection between the skewed fixed points of the operator $\cG_{k,d,\beta}$ and the analysis of Belief Propagation on random Galton-Watson trees studied in the previous section. We first identify $\cP(\{-1,1\})$ with $(0,1)$ through $\eta \mapsto \eta(-1)$. This also identifies $\cP(\cP(\{-1,1\}))$ with $\cP(0,1)$. With the notations of the previous section (and, in particular, $q$ given by (\ref{eq_def_q})), we shall prove that
\begin{proposition} \label{prop_fixed_point_to_galton_watson} Let $\pi$ be a skewed fixed point of $\cG_{k,d,\beta}$ and $\w \geq 1$ be fixed. Then we have $$  \pi  = \sum_{T \in \cT_{2\w}} p_{k,d,\beta}^{(2\w)}(T) \int_{\cP(\{-1,1\})^{\partial V_{2\w}}} \delta_{\nu_{T}^{\partial \nu}} \bigotimes_{x \in \partial V_{2\w}}  \left( \vecone_{b_{x, \uparrow} = -1} \frac{1-\partial \nu_x(-1)}{1-q}\dd \pi(\partial \nu_x) + \vecone_{b_{x,\uparrow}=1} \frac{\partial \nu_x(-1)}{q} \dd \pi(\partial \nu_{x}) \right).  $$ \end{proposition}

Let us see how \Prop~\ref{prop_unique_fixed_point_tree_1} follows from this proposition and from the result of the previous section.

\begin{proof}[Proof of \Prop~\ref{prop_unique_fixed_point_tree_1}] Let, for $\w \geq 0$, $\pi^{(2\w)} \in \cP(\{-1,1\})$ be the distribution of $\nu_{\vec T}^{(2\w)}$. 
Let $\vec \nu$ be distributed according to $\pi$. It follows from \Prop~\ref{prop_fixed_point_to_galton_watson} that $\vec \nu = \nu_{\vec T}^{\vec {\partial \nu}}$, where $\vec T$ and $\vec {\partial \nu}$ satisfies the assumptions of \Sec~\ref{sec_galton_watson}. In particular it follows from \Prop~\ref{prop_contraction_tree_tensor} that $\pi$ weakly converges towards $\pi^{(2\w)}$, hence the unicity of the fixed point. By a similar reasonning, we see that $\pi^{(2\w)}$ admits a weak limit, proving the existence of the fixed point. \end{proof}

\subsection{The multi-type Galton-Watson branching process: proof of \Prop~\ref{prop_fixed_point_to_galton_watson}}
\label{sec_multi_type_GW}

For $\pi, \hpi \in \cP(0,1)$ we define
$$ h(\pi) = \int_{(0,1)} \eta \dd \pi(\eta), \qquad \hh(\hpi) = \int_{(0,1)} \heta \dd \hpi(\heta).$$
We let $f : (0,1)^{d-1} \to (0,1)$ (resp. $\hf : (0,1)^{k-1} \to (0,1)$) be defined by \begin{align*}f (\heta_1, \dots, \heta_{d-1}) &= \frac{\prod_{j=1}^{d/2-1} \heta_j \prod_{j=d/2}^{d-1} (1-\heta_j) }{z(\heta_1, \dots, \heta_{d-1})}, \qquad
 \hf (\eta_1, \dots, \eta_{k-1}) = \frac{1-c_\beta \prod_{i=1}^{k-1} \eta_i }{\hz(\eta_1, \dots, \eta_{k-1})},
\end{align*}
and $f_{d},\hf_{k,\beta}  : [0,1] \to (0,1)$ be defined by
\begin{align*} f_{d}(\heta) = f(\heta, \dots, \heta) = 1-\heta, \qquad \hf_{k,\beta}(\eta) = \hf(\eta, \dots, \eta). \end{align*}
We say that $(\pi,\hpi)$ is a fixed point of $(\cF_{k,d,\beta},\hcF_{k,d,\beta})$ iff $\pi = \cF_{k,d,\beta}(\hpi)$ and $\hpi = \hcF_{k,d,\beta}(\pi)$.
\begin{lemma} \label{lemma_pi_hpi_to_h_hh} If $(\pi,\hpi)$ is a fixed point of $(\cF_{k,d,\beta},\hcF_{k,d,\beta})$, then we have
$$ h[\pi] = f_d (\hh[\hpi]), \qquad \hh[\hpi] = \hf_{k,\beta}(h[\pi]). $$ \end{lemma}
\begin{proof} We first observe that, using the multilinearity of $z$ (resp. $\hz$)
\begin{align*} Z[\hpi] &= \int_{(0,1)^{d-1}} z(\heta_1, \dots, \heta_{d-1}) \bigotimes_{j=1}^{d-1} \dd \hpi(\heta_j) = z(\hh[\hpi], \dots \hh[\hpi]),
\\ \hZ[\pi]  &= \int_{(0,1)^{k-1}} \hz(\eta_1, \dots, \eta_{k-1}) \bigotimes_{j=1}^{k-1} \dd \pi(\eta_j) = \hz(h[\pi], \dots, h[\pi]). \end{align*}
Using these equations, we obtain
\begin{align*} h[\pi] = \int_{(0,1)} \eta \dd \cF_{d, k, \beta}[\hpi](\eta) &= \frac{1}{Z[\hpi]} \int_{(0,1)^{d-1}} z(\heta_1, \dots, \heta_{d-1}) f(\heta_1, \dots, \heta_{d-1}) \bigotimes_{j=1}^{d-1} \dd \hpi(\heta_j) 
\\ & = \frac{1}{Z[\hpi]} \int_{(0,1)^{d-1}} \left[ \prod_{j=1}^{d/2-1} \heta_j \prod_{j=d/2}^{d-1} (1-\heta_j) \right]  \bigotimes_{j=1}^{d-1} \dd \hpi(\heta_j)
\\ & = f_d(\hh[\hpi]).
\end{align*}
Similarly, we have
\begin{align*} \hh[\hpi] = \int_{(0,1)} \heta \dd \hcF_{d, k, \beta}[\pi](\heta) &= \frac{1}{\hZ[\pi]} \int_{(0,1)^{k-1}} \hz(\eta_1, \dots, \eta_{k-1}) f(\eta_1, \dots, \eta_{k-1}) \bigotimes_{j=1}^{k-1} \dd \pi(\eta_j) 
\\ & = \frac{1}{\hZ[\pi]} \int_{(0,1)^{k-1}} \left[ 1-c_\beta \prod_{j=1}^{k-1} \eta_j \right]  \bigotimes_{j=1}^{k-1} \dd \pi(\eta_j) 
\\ & = \hf_{k,\beta}(h[\pi]).
\end{align*}
\end{proof}

Recalling the definition of $q = q(d,k,\beta)$ in Eq.~(\ref{eq_def_q}), and defining $\widehat{q} = 1-q$, the following is a simple observation.

\begin{fact} \label{fact_unique_solution_h_hh} The set of equations $$ 
y = 1-\hat{y},\qquad \hat{y} = \hf_{k,\beta}(y),$$
admits for unique solution in $[0,1]^2$ the pair $(q,\widehat{q})$.
\end{fact}
  
We define the measures $\pi_+, \pi_-, \hpi_+$ and $\hpi_-$ over $(0,1)$ by \begin{align} \label{eq_def_dd_pi_sat_unsat} \dd \pi_{+} (\eta) &= \frac{1- \eta}{1-q} \dd \pi(\eta), \qquad \dd \pi_{-} (\eta) = \frac{\eta}{q} \dd \pi(\eta),
\\ \label{eq_def_dd_hpi_sat_unsat} \dd \hpi_{+} (\eta) &= \frac{1- \heta}{1-\hqq} \dd \hpi(\heta), \qquad \dd \hpi_{-} (\heta) = \frac{\heta}{\hqq} \dd \hpi(\heta). \end{align} 

\begin{lemma} \label{lemma_fixed_point_pi_-_+} If $(\pi, \hpi)$ is a fixed point of $(\cF_{k,d,\beta},\hcF_{k,d,\beta})$, we have
\begin{align} \label{eq_fixed_point_pi_sat_unsat_a} \pi_{-}&= \int_{(0,1)^{d-1}}  \delta_{f(\heta_1, \dots, \heta_{d-1})} \bigotimes_{j=1}^{d/2-1} \dd \hpi_{-}(\heta_j) \bigotimes_{j=d/2}^{d-1} \dd \hpi_{+}(\heta_j),
\\ \label{eq_fixed_point_pi_sat_unsat_b} \pi_{+}&= \int_{(0,1)^{d-1}}  \delta_{f(\heta_1, \dots, \heta_{d-1})} \bigotimes_{j=1}^{d/2-1} \dd \hpi_{+}(\heta_j) \bigotimes_{j=d/2}^{d-1} \dd \hpi_{-}(\heta_j),
\\ \label{eq_fixed_point_pi_sat_unsat_c} \nonumber \hpi_{-} & =   \sum_{r=1}^{k-1} \binom{k-1}{r}  \frac{q^r (1-q)^{k-1-r}}{1-c_\beta q^{k-1}}    
\\ & \hspace{1cm} \int_{(0,1)^{k-1}} \delta_{\hf(\eta_1, \dots, \eta_{k-1})} \bigotimes_{j=1}^r\dd \pi_-(\eta_j) \bigotimes_{j=r+1}^{k-1}\dd \pi_+(\eta_j) 
\\ &\nonumber \hphantom{=} + \exp(-\beta) \frac{ q^{k-1} }{1-c_\beta q^{k-1}} \int_{(0,1)^{k-1}} \delta_{\hf(\eta_1, \dots, \eta_{k-1})} \bigotimes_{j=1}^{k-1} \dd \pi_-(\eta_j),
\\ \label{eq_fixed_point_pi_sat_unsat_d}\nonumber \hpi_{+} & =   \sum_{r=0}^{k-1} \binom{k-1}{r}  q^r (1-q)^{k-1-r} 
\\ & \hspace{1 cm} \int_{(0,1)^{k-1}} \delta_{\hf(\eta_1, \dots, \eta_{k-1})} \bigotimes_{j=1}^r \dd \pi_-(\eta_j) \bigotimes_{j=r+1}^{k-1}\dd \pi_+(\eta_j)  .
 \end{align}
\end{lemma}

\begin{proof} We first observe that, recalling the definition of $z$ in \Sec~\ref{subsec_results}, $q Z[\hpi] = \hqq^{d/2-1} (1-\hqq)^{d/2}$. We then compute
\begin{align*} \pi_{-}= \int_{(0,1)} \frac{\eta}{q} \delta_\eta  \dd \pi(\eta) &= \int_{(0,1)}  \frac{\eta}{q} \delta_\eta  \dd \cF_{k,d,\beta}[\hpi](\eta)
\\ & = \frac{1}{Z[\hpi]} \int_{(0,1)^{d-1}} \frac{1}{q} z(\heta_1, \dots, \heta_{d-1}) f(\heta_1, \dots, \heta_{d_1}) \delta_{f(\heta_1, \dots, \heta_{d_1})}  \bigotimes_{j=1}^{d-1} \dd \hpi(\heta_j)
\\ & = \int_{(0,1)^{d-1}} \frac{\prod_{j=1}^{d/2-1} \heta_j \prod_{j=d/2}^{d-1} (1-\heta_j)}{\prod_{j=1}^{d/2-1} \hqq \prod_{j=d/2}^{d-1} (1-\hqq)} \delta_{f(\heta_1, \dots, \heta_{d_1})}  \bigotimes_{j=1}^{d-1} \dd \hpi(\heta_j)
\\ & =  \int_{(0,1)^{d-1}}  \delta_{f(\heta_1, \dots, \heta_{d-1})} \bigotimes_{j=1}^{d/2-1} \dd \hpi_{-}(\heta_j) \bigotimes_{j=d/2}^{d-1} \dd \hpi_{+}(\heta_j)
\end{align*}
The equation on $\pi_{+}$ is proved similarly. We also compute
\begin{align*}  \hpi_{-} =\int_{(0,1)} \frac{\heta}{\hqq} \delta_{\heta}  \dd \hpi(\heta) &=\int_{(0,1)} \frac{\heta}{\hqq} \delta_{\heta}  \dd \hcF_{k,d,\beta}[\pi](\heta)
\\ & =   \int_{(0,1)^{k-1}}  \frac{ 1- c_\beta \prod \eta_j }{1 - c_\beta q^{k-1} } \delta_{\hf(\eta_1, \dots, \eta_{k-1})} \bigotimes_{j=1}^{k-1} \dd \pi(\eta_j)
\\ & =   \sum_{r=0}^{k-2} \binom{k-1}{r}    \frac{q^r (1-q)^{k-1-r}}{1-c_\beta q^{k-1}}   
\\ & \hspace{1cm} \int_{(0,1)^{k-1}}\delta_{\hf(\eta_1, \dots, \eta_{k-1})} \bigotimes_{j=1}^r\dd \pi_-(\eta_j) \bigotimes_{j=r+1}^{k-1}\dd \pi_+(\eta_j) 
\\ & \hphantom{=} + \exp(-\beta) \frac{ q^{k-1} }{1-c_\beta q^{k-1}} \int_{(0,1)^{k-1}} \delta_{\hf(\eta_1, \dots, \eta_{k-1})} \bigotimes_{j=1}^{k-1} \dd \pi_-(\eta_j) . \end{align*} 
The equation on $\hpi_{+}$ is proved in a similar manner.
 \end{proof}

\begin{proof}[Proof of \Prop~\ref{prop_fixed_point_to_galton_watson}] We observe that $\pi = (1-q) \pi_+ + q \pi_-$. Replacing with \Lem~\ref{lemma_fixed_point_pi_-_+} yields
$$  \pi  = \sum_{T \in \cT_2} p_{k,d,\beta}^{(2)}(T) \int_{\cP(\{-1,1\})^{\partial V_2}} \delta_{\nu_{T}^{\partial \nu}} \bigotimes_{x \in \partial V_2}  \left( \vecone_{b_{x, \uparrow} = -1} \frac{1-\partial \nu_x(-1)}{1-q}\dd \pi(\partial \nu_x) + \vecone_{b_{x,\uparrow}=1} \frac{\partial \nu_x(-1)}{q} \dd \pi(\partial \nu_{x}) \right).  $$
By induction over $1 \leq t \leq \w$, using repeatedly \Lem~\ref{lemma_fixed_point_pi_-_+}, we obtain that
$$  \pi  = \sum_{T \in \cT_{2t}} p_{k,d,\beta}^{(2t)}(T) \int_{\cP(\{-1,1\})^{\partial V_{2t}}} \delta_{\nu_{T}^{\partial \nu}} \bigotimes_{x \in \partial V_{2t}}  \left( \vecone_{b_{x, \uparrow} = -1} \frac{1-\partial \nu_x(-1)}{1-q}\dd \pi(\partial \nu_x) + \vecone_{b_{x,\uparrow}=1} \frac{\partial \nu_x(-1)}{q} \dd \pi(\partial \nu_{x}) \right).  $$ 
This concludes the proof of the proposition.
 \end{proof}
 
We define, for $(\nu_1, \dots, \nu_k, \hnu_1, \dots, \hnu_d) \in \cP(\{-1,1\})^{k+d}$ and $(b_1, \dots, b_k) \in \{-1,1\}^k$,
\begin{align*}z_1(\hnu_1, \dots, \hnu_d) &=  {\prod_{j\leq d/2} \nu_j(-1) \prod_{j>d/2} \nu_j(1)+\prod_{j\leq d/2} \nu_j(1) \prod_{j>d/2} \nu_j(-1)},
\\ z_2(\nu_1, \dots, \nu_k, b_1, \dots, b_k) &= 1-c_\beta \prod_{j =1}^k \nu_j(b_j),
\\ z_3(\nu_1,\hnu_1) &= \nu_1(1)\hnu_1(1)+\nu_1(-1)\hnu_1(-1). \end{align*}

In order to prove \Lem~\ref{lemma_total_size}, we also recall the following standard result, which we prove in \Sec~\ref{sec_vanilla}.

\begin{proposition} \label{prop_first_moment_vanilla} We have $$\frac{1}{n} \ln \Erw \left[ Z_\beta(\vec \Phi) \right] \sim \ln 2 + \frac{d}{k} \ln \left(1 - c_\beta q^k \right) - \frac{d}{2} \ln \left(\frac{1}{2 q} \right) - \frac{d}{2} \ln \left(\frac{1}{2 (1-q)} \right) .$$ \end{proposition}

\begin{proof}[Proof of \Lem~\ref{lemma_total_size}] We compute
 \begin{align*}\ln \Erw \left[ z_1(\hnu_1, \dots, \hnu_d)\right]  &= \ln \left(2q^{d/2}(1-q)^{d/2} \right),
 \\ \ln  \Erw \left[ z_2(\nu_1, \dots, \nu_k,b_1, \dots, b_k) \right] &= \ln \left(1-c_\beta q^k \right),
 \\ \ln \Erw [z_3(\nu_1,\hnu_1)]  &= \ln \left( 2 q (1-q) \right). \end{align*}
 Thereby we have $$  \cF(k,d,\beta)= \ln 2 + \frac{d}{k} \ln \left(1 - c_\beta q^k \right) - \frac{d}{2} \ln \left(\frac{1}{2 q} \right) - \frac{d}{2} \ln \left(\frac{1}{2 (1-q)} \right).$$
 The proposition then follows from \Prop~\ref{prop_first_moment_vanilla}.
\end{proof}

\subsection{Finite $\w$ approximations of $\cB(k,d,\beta)$}
\label{sec_finite_w_approx}

We finally present a simple approximation of $\cB(k,d,\beta)$ that will be useful in the following. Recall that $\textrm{GW}(k,d,\beta,2\w)$ and $\textrm{GW}'(k,d,\beta,2\w)$ were defined in the previous section. Let $\widehat{p}_{k,d,\beta}^{(2\w+1)}(\widehat{T}_1, \dots \widehat{T}_d)$ denote the probability that the neighborhood of the root in is equal to $(\widehat{T}_1, \dots, \widehat{T}_d)$ under the process $\textrm{GW}'(k,d,\beta,2\w+2)$. Denoting by $e_1$ the first edge exiting the root of the random process $\textrm{GW}'(k,d,\beta,2\w+2)$, let for $T \in \cT_{2\w}$ and $\widehat{T} \in \widehat{\cT}_{2\w+1}$ $\widecheck{p}(T, \widehat{T})$ be the probability that the $2\w$-neighborhood (resp. $2\w+1$-neighborhood) of this edge when removing its clause node (resp. variable node) is formed of the tree $T$ (resp. $\widehat{T}$). Similarly, denoting by $a_1$ the first clause connected to the root of the random process $\textrm{GW}'(k,d,\beta,2\w+2)$, let for $T_1, \dots, T_k \in \cT_\w$ $\widehat{p}_{k,d,\beta}(T_1, \dots T_k)$ be the probability that the $2\w$-neighborhood of $a_1$ is equal to $(T_1, \dots, T_k)$ under the process $\textrm{GW}'(k,d,\beta,2\w+2)$.

Finally, let $\pi^{\star}_{k,d,\beta}$ be the unique skewed fixed point of $\cG_{k,d,\beta}$ and, for the ease of notations, let $\pi_+, \pi_-, \hpi_+, \hpi_-$ denote the quantities associated to $\pi^{\star}_{k,d,\beta}$ through Eq. (\ref{eq_def_dd_pi_sat_unsat}-\ref{eq_def_dd_hpi_sat_unsat}).
\begin{lemma} We have
\begin{equation*}  \begin{split}  \cB(k,d,\beta) =& \sum_{(\hT_1, \dots, \hT_d) \in  \widehat{\cT}_{2\w+1}^d } p_{k,d,\beta}^{(2\w+1)}(\hT_1, \dots, \hT_d)\int_{(0,1)^d} \ln \left[ z_1(\widehat{\nu}^{\partial \nu_1 }_{\widehat{T}_1}, \dots, \widehat{\nu}^{\partial \nu_d }_{\widehat{T}_d} ) \right] 
\bigotimes_{j=1}^d \bigotimes_{x \in \partial_1 \hT_j} \dd \pi_+( (\partial \nu_j)_x ) \bigotimes_{x \in \partial_{-1} \hT_j} \dd \pi_-((\partial \nu_j)_x )
\\ &+\frac{d}{k} \sum_{(T_1, \dots, T_k) \in {\cT}_{2\w}^k} \widehat{p}_{k,d,\beta}^{(2\w)}(T_1, \dots, T_k) \int_{(0,1)^k} \ln \left[ z_2 (\nu_{T_1}^{\partial \nu_1}, \dots,\nu_{T_k}^{\partial \nu_k}) \right]
 \bigotimes_{j=1}^k \bigotimes_{x \in \partial_1 T_j} \dd \pi_+((\partial \nu_j)_x) \bigotimes_{x \in \partial_{-1} T_j} \dd \pi_-((\partial \nu_j)_x) 
\\ & - \sum_{\substack{T \in \cT_{2\w}, \hT \in \widehat{\cT}_{2\w+1}}} \widecheck{p}_{k,d,\beta}^{(2\w+1)}(T,\hT) \int_{(0,1)^2} \ln \left [ z_3 (\nu_{T}^{\partial \nu_1}, \hnu_{\hT}^{\partial \nu_2} )  \right] \bigotimes_{j=1}^2 \bigotimes_{x \in \partial_1 \hT_j} \dd \pi_+( (\partial \nu_j)_x ) \bigotimes_{x \in \partial_{-1} \hT_j} \dd \pi_-((\partial \nu_j)_x )+ o_\w(1). \end{split} \end{equation*}  \end{lemma}

\begin{proof} The proof is obtained by writing the expectation values in the definition of $\cB(k,d,\beta)$ explicitly in terms of $\pi_+, \pi_-, \hpi_+, \hpi_-$ and by  following steps similar to the one of the proof of \Prop~\ref{prop_fixed_point_to_galton_watson}. \end{proof} 
We define
\begin{equation*}  \begin{split} 
 \cB^{(\w)}(k,d,\beta) =& \sum_{(\hT_1, \dots, \hT_d) \in  \widehat{\cT}_{2\w+1}^d } \widehat{p}_{k,d,\beta}^{(2\w+1)}(\hT_1, \dots, \hT_d) \ln \left[ z_1(\widehat{\nu}^{(2\w+1)}_{\widehat{T}_1}, \dots, \widehat{\nu}^{(2\w+1) }_{\widehat{T}_d} ) \right] 
\\ &+\frac{d}{k} \sum_{(T_1, \dots, T_k) \in {\cT}_{2\w}^k} {p}_{k,d,\beta}^{(2\w)}(T_1, \dots, T_k)  \ln \left[ z_2 (\nu_{T_1}^{(2\w)}, \dots,\nu_{T_k}^{(2\w)}) \right]
 \\ & - \sum_{\substack{T \in \cT_{2\w}, \hT \in \widehat{\cT}_{2\w+1}}} \widecheck{p}_{k,d,\beta}^{(2\w+1)}(T,\hT)  \ln \left [ z_3 (\nu_{T}^{(2\w)}, \hnu_{\hT}^{(2\w+1)} )  \right]. \end{split} \end{equation*} 
\begin{proposition} \label{prop_approximation_internal_bethe_tree} We have
$$  \cB(k,d,\beta) = \cB^{(\w)}(k,d,\beta) + o_\w(1).$$
 \end{proposition}

 \begin{proof} The result easily follows from the weak convergence of $\pi^{(\w)}$ toward $\pi^\star_{k,d,\beta}$. \end{proof}

We now proceed to prove \Prop~\ref{Prop_BetheOnTrees}. In order to do so, we need to state here a standard lemma about the local convergence of the random formula $\hPhi$, that we will prove in \Sec~\ref{sec_typical_properties} (see \Lem~\ref{lemma_typical_w_neighb}). 

\begin{lemma} For all $\ell \geq 0$ and $\forall T \in \widetilde{\cT}_{2\w+2}$, we have $\rho_{\Phi}(T) \sim \widetilde{p}_{k,d,\beta}^{(2\w+2)}(T)$. \end{lemma}

\begin{proof}[Proof of \Prop~\ref{Prop_BetheOnTrees}] By the previous lemma we have, for $\w \geq 0$,
$$ \lim_{n\to\infty}\frac1n\Erw[\cB_{\hat\PHI,\w}] = \cB^{(\w)}(k,d,\beta).$$
The result then follows from \Prop~\ref{prop_approximation_internal_bethe_tree}.
 \end{proof}

\section{Marginal analysis} \label{sec_marginal_analysis}

\noindent
We will exhibit a number of {\em deterministic} conditions (six in total) that entail the non-reconstruction property.
Subsequently we are going to show that the random formulas $\hat\PHI$ and $\tilde\PHI$ enjoy these properties with high probability.

Then we will need some information on the local structure of a formula $\Phi$. For a variable node $x \in V$ and $t \geq 0$ we let $\ppartial^{(t)} \Phi(x)$ denote the $t$-neighborhood of $x$ in $\Phi$. For a tree $T \in \widetilde{\cT}_{2\w+2}$ (defined in \Sec~\ref{sec_finite_w_approx}) we define the empirical density of $T$ by
$$\rho_{\Phi}(T)= \frac{1}{n} \sum_{i \in [n]} \mathbf{1}_{\ppartial^{(2\w+2)} \Phi(x_i)\ism T}.$$
We shall say that a regular $k$-SAT formula $\Phi$ satisfies property {\bf Local Structure} if the following is true, for every $\w$ large enough.
\begin{description}
\item[{\bf Local Structure}]  $\forall T \in \widetilde{\cT}_{2\w+2}, \ \rho_{\Phi}(T) \sim \widetilde{p}_{k,d,\beta}^{(2\w+2)}(T).$ \end{description}

We shall also demand that $\Phi$ satisfies the {\bf Cycles} property: 
\begin{description}
\item[{\bf Cycles}]  There are $o(\sqrt n)$ cycles of length at most $\sqrt{\ln n}$. \end{description}

In order to proceed further, we need to introduce a few more notations, similar to the ones that we used in \Sec~\ref{sec_galton_watson}. Let $\Phi$ be fixed and $V$ denote its set of vertices, $F$ denote its set of edges and $E$ denote its set of (undirected) edges. For $x \in V$ (resp. $a \in F$), we let 
\begin{align*}
\partial_{1}x&= \{a \in \partial x, b_{a,x}=-1 \}, \qquad  \partial_{-1} x = \{a \in \partial x,  b_{a,x}=1 \}, \qquad\partial_{l}x& =\left \{ a \in \partial x, | \{ y \in \partial a \setminus \{x\}, b_{a,y} = 1 \}| = l \right\},
\\ \partial_1 a &= \{x \in \partial a, b_{a,x} = -1\}, \qquad \partial_{-1} a = \{x \in \partial a, b_{a,x}=1\}.  \end{align*}
We also introduce $\partial_{1,l}x= \partial_{1}x\cap \partial_l x$,  $\partial_{-1,l}x= \partial_{-1} x \cap \partial_l x$, and for $a_0 \in \partial x$ fixed, 
 $\partial_1(x,a_0) = \partial_{1} x \setminus \{a_0\}$, $\partial_{-1}(x,a_0) = \partial_{-1} x \setminus \{a_0\}$, $\partial_{1,l}(x,a_0) = \partial_{1,l} x \setminus \{a_0\}$, $\partial_{1,l}(x,a_0) = \partial_{1,l} x \setminus \{a_0\}$.
 
We define the {\em $\lambda$-core} of $\Phi$ (in symbols: $\core_\lambda(\Phi)$) as the largest set $W$ of variables such that all $x\in W$ satisfy the following conditions.
\begin{description}
\item[CR1] there are at least $k (1 - \frac{\lambda^{-1}}{100})$ clauses $a\in\partial_1x$ such that $\partial_1a=\{x\}$.
\item[CR2] there are no more than $k \exp(-\beta) \left(1+ \frac{\lambda}{100}\right)$ clauses $a\in\partial x$ such that $|\partial_{-1}a|=k$.
\item[CR3] for any $1\leq l \leq k$ the number of $a\in\partial_{-1} x$ such that $|\partial_1a| = l$ is bounded by $\lambda k^{l+3} / l!\enspace$.
\item[CR4] there are no more than $\lambda k^{3/4}$ clauses $a\in\partial_{1}x$ such that $|\partial_1a|=1$ but $\partial a\not\subset W$.
\item[CR5]  there are no more than $\lambda k^{3/4}$ clauses $a\in\partial_{-1}x$ such that $|\partial_{-1}a|<k$ and $|\partial_1 a \setminus W| \geq |\partial_1 a|/4$.
\end{description}
The $\lambda$-core is well-defined; for if $W,W'$ satisfy the above conditions, then so does $W\cup W'$.
Further, if $\lambda<\lambda'$, then $\core_{\lambda}(\Phi)\subset\core_{\lambda'}(\Phi)$. Also note the similarity with the trunk of trees defined in \Sec~\ref{sec_proof_prop_contraction_tree_tensor}.
We say that $\Phi$ satisfies the property {\bf Core} if and only if
\begin{description}
\item[{\bf Core}] $|\core_{1/2}(\Phi)| \geq (1-2^{-0.95})n$. \end{description}

Our aim will be to identify a large set $V_{\rm good} \subset V$ of vertices whose value under a typical assignment in the cluster is unlikely to be very far from the planted one.
A first candidate for vertices whose marginal may go wrong are those which do not belong to the $1$-core of $\Phi$. Yet, we are not guaranteed that vertices in the core have marginals sufficiently close to $\mu^{(0)}$. For instance, if the marginals of most of the neighbors of a given vertex $x \in \core_1(\Phi)$ went astray, there would be no reason for $x$'s marginal not to go astray itself. However, we see that the vertices in the core whose marginals are not what we think they should be must clump together. We say that a set $S \subset V$ is $\lambda$-\textit{sticky} if and only if for all $x \in S$, one of the following conditions holds true.
\begin{description}
\item[ST1] there are at least $\lambda k^{3/4}$ clauses $a\in\partial_1x$ such that $\partial_1a=\{x\}$ and $\partial_{-1}a\cap S\neq\emptyset$. 
\item[ST2] there are at least $\lambda k^{3/4}$ clauses $a\in\partial_{-1}x$ such that $|\partial_{-1} a| < k$ and $|\partial_1a\cap S|\geq |\partial_1a|/4$.
\end{description}

We say that $\Phi$ satisfies the property {\bf Sticky} if and only if
\begin{description}
\item[{\bf Sticky}] $\Phi$ has no $1/2$ sticky set of size between $2^{-0.95k}n$ and $2^{-k/20}n$. \end{description}

Finally, say that a variable $x \in V$ is {\em $(\eps,2\w)$-cold} if the following is true. Let $T=\ppartial^{2\w} x$.
Then $T$ is a tree.
Moreover, if we choose a boundary condition $\tau$ such that
	\begin{itemize}
	\item the values of variables $y$ that do not belong to the core are chosen adversarially,
	\item the values of the other variables are chosen i.i.d.\ such that the probability of $-1$ equals $\exp(-2\beta)$,
	\item subsequently an adversary is allowd to change some of the $-1$s to $+1$s,
	\end{itemize}
then with this boundary condition the BP marginal at the root of the tree is within $\eps$ of $\mu_T$ in total variation distance.

We say that $\Phi$ satisfies the property {\bf $(\eps,2\w)$-Cold} iff
\begin{description}
\item[{\bf $(\eps,2\w)$-Cold}] All but $\eps n$ variables are $(\eps,2\w)$-cold. \end{description}

We say that a formula $\Phi$ is {\em $(\eps, \w)$-tame} iff the properties {\bf Local Structure}, {\bf Cycles}, {\bf Core}, {\bf Sticky} and {\bf $(\eps,2\w)$-Cold} hold. Planted formulas are likely to be tame.

\begin{proposition} \label{proposition_typical_properties} For any $\eps >0$, there is $\w >0$ such that \whp~ $\vec \hPhi$ is $(\eps,\w)$-tame.
 \end{proposition}
 
 Similarly, formulas from the planted replica model are likely to be tame as well. 
\begin{proposition} \label{proposition_typical_properties_planted} For any $\eps >0$, there is $\w >0$ such that \whp~ $\vec \hhPhi$ is $(\eps,\w)$-tame.
 \end{proposition}

We prove \Prop s~\ref{proposition_typical_properties} and~\ref{proposition_typical_properties_planted} in \Sec~\ref{sec_typical_properties}.
We are going to show that $(\eps,\w)$-tame formulas have the non-reconstruction property. In the rest of this section, we assume that $\Phi$ is $(\eps,\ell)$-tame.
Let us briefly write $\Bck\nix=\Bck\nix_{\Phi,\beta}$ and $\bck\nix=\bck\nix_{\Phi,\beta}$.

We say that a set $T \subset\core_1(\Phi)\setminus S_1(\Phi)$ is \emph{$\sigma$-closed} if for any $x\in T$ and all $a\in\partial x$
we have $$\{y\in\partial a\cap\core_1(\Phi)\setminus S(\Phi):\sigma(y)=-1\}\subset T.$$
Moreover, for a clause $b$ we say $T \subset\core_1(\Phi)\setminus S(\Phi)$ is \emph{$(\sigma,b)$-closed} if
the above holds for all $x\in T$ and all $a\in\partial x\setminus b$.

\begin{lemma}\label{Lemma_flip}
Suppose that $\Phi$ is $(\eps,\ell)$-tame.
Then for any $\sigma$ such that $\vecone\cdot\SIGMA\geq(1-2^{-k/2})n$
and for any $(\sigma,b)$-closed set $T \subset\core_1(\Phi)\setminus S(\Phi)$ the following is true.
Let $\tilde\sigma(x)=(-1)^{\vecone\{x\in T\}}\sigma(x)$.
Then
	 \beq \label{eq_claim_victor_expansion_energy} E_{\Phi}(\widetilde{\sigma}) \leq E_{ \Phi}(\sigma) - k^{3/4} |T|. \eeq
\end{lemma}
\begin{proof}
Consider the following process:
	\begin{itemize}
	\item Let $\sigma_{0} = \sigma$, $V_0 = T$ and $U_0 = \sigma^{-1}(-1) \setminus V_0$.
	\item While there is $i_t \in V_t$ such that $E_{\Phi} ((-1)^{\vecone\{\nix=i_t\}}\sigma_t(\nix)) \leq E_{\Phi}(\sigma_t) - k^{3/4}$, 
			pick one such $i_t$ uniformly at random and let $\sigma_{t+1}(\nix) = (-1)^{\vecone\{\nix=i_t\}}\sigma_t(\nix)$ and $V_{t+1} = V_t \setminus \{i_t\}$.
	\end{itemize}
Clearly,
	\beq \label{eq_aux_lemma_key_property_flipping_a}
	E_{\Phi}(\sigma_{t}) \leq E_{\Phi}(\sigma) - k^{3/4}t . \eeq
Let $\tau$ be the stopping time of this process and assume that $\tau < |T|$, or, in other words, that $V_\tau \neq \emptyset$. We claim that $V_\tau$ is a $1$-sticky set. Indeed, because $T$ is $\sigma$-closed for $i \in V_\tau$ we have
	\begin{align*}
		-k^{3/4} \leq E_{\Phi}((-1)^{\vecone\{\nix=i\}}\sigma_t(\nix)) - E_{\Phi}(\sigma_{\tau}) \leq \vecone\{b\in\partial i\}-
			& |\partial_{1,0}(i)| + | \{ a \in \partial_{1,0} i, \partial_{-1}a \cap (V_\tau \cup U_0) \neq \emptyset \}| 
		\\ &+| \partial_{-1,0}i| + | \cup_{1 \leq l \leq k} \{a \in \partial'_{-1, l}, \partial_1 a \subset V_\tau\cup U_0)\}|.
		\end{align*}
Because $i \in \core_1(\Phi)$ we have $|\partial_{1,0} i|  \geq k^{7/8}$, $|\partial_{-1,0} i | \leq 3$, $ | \{a \in \partial_{1,0}, \partial_{-1} a \cap U_0 \neq \emptyset \}| \leq {k^{3/4}}$ and $| \{ a \in \partial_{1,0} i, |\partial_{-1}a \cap U_0| \geq |\partial_{-1}a|/4\}| \leq k^{3/4}$. Therefore, one of the following must hold.
\begin{itemize}
\item[(a)] $ | \{a \in \partial_{1,0}, \partial_{-1} a \cap V_\tau \neq \emptyset \}| \geq {k^{3/4}}$,
\item[(b)] $| \{ a \in \partial_{1,0} i, |\partial_{-1}a \cap V_{\tau}| \geq |\partial_{-1}a|/4\}| \geq k^{3/4}$,
\end{itemize}
It follows that the set $V_\tau$ is $1$-sticky. However, $\core_1(\Phi) \setminus S_1(\Phi)$ cannot contain a $1$-sticky set of size $|V_\tau| \leq |T| \leq 2^{-k/20}$ as this would contradict the maximality of $S(\Phi)$. It follows that $\tau = |T|$, and therefore $\sigma_\tau = \tilde\sigma$, from which (\ref{eq_claim_victor_expansion_energy}) follows using (\ref{eq_aux_lemma_key_property_flipping_a}).
\end{proof}

\begin{fact}\label{Fact_treeCount}
For any variable $x$ the following is true.
Let $\gamma(x,L)$ be the number of trees of order $L$ rooted at $x$ that are contained in the factor graph of $\Phi$.
Then
	$\gamma(x,L)\leq L(100dk)^L.$
\end{fact}

Write $T(x,\sigma)$ for the smallest $\sigma$-closed set that contains $x$.
In other words, this is the $-1$-component in $\core_1(\Phi)\setminus S_1(\Phi)$ that $x$ belongs to.
If $\sigma(x)=1$ we let $T(x,\sigma)=\emptyset$.

\begin{lemma}\label{Lemma_skewedWhp}
If $\Phi$ is $(\eps,\ell)$-tame, then for all $x\in\core_1(\Phi)\setminus S_1(\Phi)$ we have 
	\begin{align*}
	&\Bck{\SIGMA(x)}\geq1-\exp(-\beta k^{3/4}/2)\quad\mbox{and}\quad\Bck{\vecone\{|T(x,\SIGMA)|>\ln\ln n\}}\leq1/\ln n.
	\end{align*}
\end{lemma}
\begin{proof}
Let $N=2^{-k/2}n$.
Because $\Phi$ is tame we have $\Bck{\vecone\cdot\SIGMA< n-N}\leq\exp(-\Omega(n))$.
Therefore, $\Bck{\vecone\{|T(x,\SIGMA)|>N\}}\leq\exp(-\Omega(n))$.
Hence, let $t<N$ and let $\theta$ be a tree of order $t$ with root $x$ that is contained in the factor graph of $\Phi$
and whose vertices lie in $\core_1(\Phi)\setminus S_1(\Phi)$.
If $\sigma$ is such that $T(x,\sigma)=\theta$, then \Lem~\ref{Lemma_flip} implies that
$\tilde\sigma(x)=(-1)^{\vecone\{x\in T(x,\sigma)\}}\sigma(x)$ satisfies 
$E_{\Phi}(\widetilde{\sigma}) \leq E_{ \Phi}(\sigma) - k^{3/4}t$.
Consequently,
	\begin{align*}
		\frac{\bck{\vecone\{\SIGMA=\sigma\}}}{\bck{\vecone\{\SIGMA=\tilde\sigma\}}}\leq\exp(-\beta k^{3/4}t).
	\end{align*}
Hence, by Fact~\ref{Fact_treeCount}, the union bound and our assumptions on $\beta$ and $d$,
	\begin{align}\label{eqTreeUnionBound}
	\frac{\bck{\vecone\{|T(x,\SIGMA)|=t\}}}{\bck{\vecone\{\SIGMA(x)=1\}}}\leq t(100dk)^t\exp(-\beta k^{3/4}t)\leq\exp(-0.99\beta k^{3/4}t).
	\end{align}
This bound readily implies the second assertion.
To obtain the first assertion, we sum (\ref{eqTreeUnionBound}) over $1\leq t\leq N$.
\end{proof}

\begin{fact}\label{Fact_stochDom}
Let $q\in(0,1)$ and $L\geq1$.
Suppose that $\mu$ is a probability distribution on $\{\pm1\}^L$ such that for any $i\in[L]$ and any $y_1,\ldots,y_L\in\{\pm1\}$ we have
	\begin{align*}
	(1-q)\mu(y_1,\ldots,y_{i-1},-1,y_{i+1},\ldots,y_L)&\leq
		q\mu(y_1,\ldots,y_{i-1},1,y_{i+1},\ldots,y_L).
	\end{align*}
Furthermore, let $\nu$ be the distribution on $\{\pm1\}^l$ such that for all $y_1,\ldots,y_L\in\{\pm1\}$ we have
	\begin{align*}
	(1-q)\nu(y_1,\ldots,y_{i-1},-1,y_{i+1},\ldots,y_L)&=
		q\nu(y_1,\ldots,y_{i-1},1,y_{i+1},\ldots,y_L).
	\end{align*}
Moreover, let $B\subset\{\pm1\}^L$ be a set such that for all $b\in B$ and all $b'\geq b$ we have $b'\in B$.
Then $\mu(B)\geq\nu(B)$.
\end{fact}

\begin{lemma}\label{Lemma_noRe}
Let $r$ be a variable for which the following conditions hold.
	\begin{enumerate}
	\item $r$ is $(\eps,\ell)$-cold.
	\item $r$ has distance at least $\ln^{1/3} n$ from any cycle of length at most $\sqrt{\ln n}$.
	\end{enumerate}
Let $\Gamma_r$ be the event that $\SIGMA$ is a good boundary condition for $r$.
Then
	$\bck{\vecone\{\SIGMA\not\in\Gamma_r\}}\leq2\eps.$
\end{lemma}
\begin{proof}
Let $X$ be the set of boundary variables.
Moreover, let $\cA$ be the event that $\max_{x\in X}|T(x,\SIGMA)|\leq\ln\ln n$ and that $\SIGMA\cdot\vecone\geq(1-2^{-k/2})n$.
Because $\ell$ is bounded, \Lem~\ref{Lemma_skewedWhp} and the union bound imply that $\Bck{\vecone\{\SIGMA\in\cA\}}\sim1$.
Furthermore, if $\cA$ occurs, then our assumption ensures that the subgraph of the factor graph induced on
	$Y=\partial^{\ell}r\cup\bigcup_{x\in X}T(x,\SIGMA)$ is acyclic.

Now, fix a variable $x\in X$ and $\sigma\in\cA$ such that $\sigma(x)=-1$.
Let $a$ be the clause that is adjacent to $x$ on its shortest path to $r$ and let $T(x,a,\sigma)$ be the smallest $(\sigma,a)$-closed set that contains $x$.
Further, define $\tilde\sigma(y)=(-1)^{\vecone\{y\in T(x,\sigma)\}}$.
Then \Lem~\ref{Lemma_flip} shows that $E_{\Phi}(\tilde\sigma)\leq k^{3/4}|T(x,a,\sigma)|$.
Moreover, because the subgraph induced on $Y$ is acyclic we have $\tilde\sigma(x')=\sigma(x')$ for all $x'\in X\setminus\{x\}$.
Consequently, by Fact~\ref{Fact_treeCount} and the union bound,
	\begin{align}\label{eqControlBoundary1}
	\frac{\bck{\vecone\{\SIGMA(x)=-1\}\prod_{y\in X\setminus\{x\}}\vecone\{\SIGMA(y)=\sigma(y)\}\vecone\{\SIGMA\in\cA\}}}
		{\bck{\vecone\{\SIGMA(x)=1\}\prod_{y\in X\setminus\{x\}}\vecone\{\SIGMA(y)=\sigma(y)\}}}
			&\leq\sum_{t\leq\ln\ln n}t(100dk)^t\exp(-\beta k^{3/4}t)]\leq\exp(-\beta k^{3/4}/2).
	\end{align}
Since $\Bck{\vecone\{\SIGMA\in\cA\}}\sim 1$
and because  for all $\tau:X\to\{\pm1\}$ we have
	$$\bck{\prod_{y\in X\setminus\{x\}}\vecone\{\SIGMA(y)=\tau(x)\}}\geq\exp(-dk\beta|X|)=\Omega(1),$$
(\ref{eqControlBoundary1}) implies that for any $\tau$,
	$$
	\frac{\bck{\vecone\{\SIGMA(x)=-1\}\prod_{y\in X\setminus\{x\}}\vecone\{\SIGMA(y)=\tau(y)\}}}
		{\bck{\vecone\{\SIGMA(x)=1\}\prod_{y\in X\setminus\{x\}}\vecone\{\SIGMA(y)=\tau(y)\}}}\leq\exp(-\beta k^{3/4}/3).
	$$
Thus, the assertion follows from Fact~\ref{Fact_stochDom}.
\end{proof}

\noindent
Finally, \Prop s~\ref{Prop_nonRe1} and~\ref{Prop_nonRe2} follow from 
\Prop s~\ref{proposition_typical_properties} and~\ref{proposition_typical_properties_planted} and \Lem~\ref{Lemma_noRe}.

\section{Typical properties of the random formula}
\label{sec_typical_properties}

  {\em In this section we prove \Prop~\ref{proposition_typical_properties} and \Prop~\ref{proposition_typical_properties_planted} }. 
Let $\cE_{n,k,d}$ denote the set of regular $k$-SAT formulas. For $v \in V \cup F$ and $\w \geq 0$, we let $\partial^{\w} v$ (resp. $\Delta^{\w} v$) denote the set of vertices at distance exactly $\w$ (resp. less than $\w$) from $v$.

\subsection{Proof of \Prop~\ref{proposition_typical_properties} and \Prop~\ref{proposition_typical_properties_planted}}

We first deal with the easiest condition {\bf Local Structure}.

\begin{lemma} \label{lemma_typical_w_neighb} \Whp~ $\vec \hPhi$ satisfies {\bf Local Structure}. \end{lemma}

\begin{proof} Let $x \in V$ and $T \in \widetilde{\cT}_{2(\w+1)}$ be fixed. Let $X_x(T)$ be the number of formulas $\Phi$ such that $\ppartial^{2(\w+1)} \Phi(x) = T$. It is straightforward to compute that there are precisely $$\widetilde{p}_{k,d,\beta}^{2(\w+1)}(T) \frac{(nd/2)!^2}{(nd/2- \epsilon_+)! (nd/2- \epsilon_-)!} (1+o_n(1))$$ ways to construct a tree of depth $2(\w+1)$ around $x$, where $\epsilon_+$ (resp. $\epsilon_-$) is the number of positive (resp. negative) literals that appear in $T \setminus \partial T$. Once this as been done, it remains to connect the $(dn/2-\epsilon_+)$ positive litterals clones (resp. $(dn/2-\epsilon_-)$ negative litterals clones) together. This yield
$$ \frac{X_x(T)}{|  \cE_{n,k,d}|} = \widetilde{p}_{k,d,\beta}^{(\w+1)}(T) \frac{(nd/2)!^2}{(nd/2- \epsilon_+)! (nd/2- \epsilon_-)!} \frac{(nd/2-\epsilon_+)! (nd/2-\epsilon_-)!}{(nd/2)!^2}(1+o_n(1))  = \widetilde{p}_{k,d,\beta}^{(\w+1)}(T)  .$$
Consequently, we have
$$ \hErw \left[ \rho_{\vec \Phi}(T) \right] = \frac{X_x(T)}{|  \cE_{n,k,d}|} =   \widetilde{p}_{k,d,\beta}^{2(\w+1)}(T).$$
Moreover by standard concentration arguments $\rho_{\vec \Phi}(T)$ is concentrated around its mean and we have \whp
$$ \rho_{\vec \Phi}(T) \sim  \widetilde{p}_{k,d,\beta}^{2(\w+1)}(T).$$
This holds for any $T$ in the finite set $\widetilde{\cT}_{2(\w+1)}$, ending the proof of the lemma.
 \end{proof}
 
 In particular, this entails the following.
 
\begin{corollary} \label{cor_typical_w_neighb} \Whp~ $\vec \hhPhi$ satisfies {\bf Local Structure}. \end{corollary}
 
 The following is a standard result.
 
 \begin{fact} \Whp~ $\vec \hPhi$ and $\vec \hhPhi$ satisfy the property {\bf Cycles}. \end{fact}
 
 We will prove the following in \Sec~\ref{subsec_exp_properties}.
 
 \begin{proposition} \label{prop_formulas_safe_structure} \Whp~ $\vec \hPhi$ and $\vec \hhPhi$ satisfy {\bf Core} and {\bf Sticky}.
 \end{proposition}
 
 The remaining of this section is devoted to a proof of the two following lemmas.
 
\begin{lemma} \label{lemma_typical_properties} For all $\eps >0$, there is $\w >0$ such that \whp~ $\vec \hPhi$ is $(\eps,\w)$-cold. \end{lemma}
\begin{lemma} \label{lemma_typical_properties_planted}   For all $\eps >0$, there is $\w >0$ such that \whp~ $\vec \hhPhi$ is $(\eps,\w)$-cold \end{lemma}

\begin{proof}[Proof of \Prop~\ref{proposition_typical_properties} and \Prop~\ref{proposition_typical_properties_planted}] The propositions immediatly follow from the above lemmas. 
\end{proof}

Let $\alpha \geq 0$ and $(z_1, \dots, z_\alpha) \in V^\alpha$ be fixed. Let $\Delta = \{y \in V, \exists l \in [\alpha], y \in \partial^{(2 \w)} \vec \hPhi(z_l)\}$ and for $y \in \Delta$ let $\cC_y$ be the event that $y \in \core_1(\vec \hPhi) \setminus S_1(\vec \hPhi)$. Moreover, let $\mathcal{D}$ be the event that $z_1, z_2, \dots, z_\alpha$ are at distance strictly greater than $5\w$ one from the other in $\vec \hPhi$, and that their $5 \w$ neighborhoods are tree-like. For $y \in \Delta$, let also $\mathcal{F}_y$ denote the $\sigma$-algebra generated by the function $(\Phi, z_1, \dots, z_\alpha) \mapsto \left(\ppartial^{(2\w)} \Phi(z_1) \cup \dots \cup \ppartial^{(2\w)} \Phi(z_\alpha) \right)$.

 \begin{lemma}   \label{lemma_aux_good_3}   For $y \in \Delta$, we have
$$ \hProb \left[  \neg \cC_y  | \cD, \cF_{y} \right] \leq 2^{-0.95 k} .$$ 
 \end{lemma}

\begin{proof} Let $a_y \in \partial y$ be such that $a_y \in \cup_{l \in [\alpha]} \ppartial^{(2\w)} \vec \hPhi(z_l)$. Let $\vec \hPhi'$ be obtained from $\vec \hPhi$ by the following operations.
\begin{itemize}
\item Select $x \in \vec \hPhi$ and $a_x \in \partial x$ uniformly at random.
\item Replace the pair of edges $\{(y,a_y),(x,a_x)\}$ by the pair of edges $\{(y,a_x),(x,a_y)\}$.
\end{itemize}
Let $\cE$ be the event that $\vec \hPhi'$ satisfies $\cD$. We observe that $\hProb[\cD] = 1-o_n(1)$ and $\hProb[\cE] = 1-o_n(1)$. Conditioned on $\cD, \cE$ and $\cF_y$, $\vec \hPhi$ and $\vec \hPhi'$ are identically distributed.
 Moreover, we have
$$ \core_{1/2} \left( \vec \hPhi \right) \setminus S_{1/2}\left(\vec \hPhi \right) \subset \core_1 \left(\vec \hPhi' \right) \setminus_1  S\left(\vec \hPhi' \right). $$
It follows that
\begin{align} \nonumber \hProb \left[ \neg \cC_y | \cD,\cF_y \right] & = \hProb \left[ \neg \cC_y | \cD,\cE,\cF_y \right]  + o_n(1)
\\ \nonumber &\leq \hProb \left[ \left. x \notin \core_{1/2} \left( \vec \hPhi \right) \setminus S_{1/2}\left(\vec \hPhi\right)  \right| \cD,\cE, \cF_y \right]  + o_n(1)
\\  \label{eq_aux_proof_aux_good_3_a}  &\leq \hProb \left[ \left. x \notin \core_{1/2} \left( \vec \hPhi \right) \setminus S_{1/2}\left(\vec \hPhi\right)   \right| \cF_y \right] + o_n(1).
 \end{align}
 For a fixed $\sigma$-algebra $\cF$ generated by $(\Phi, z_1, \dots, z_\alpha) \mapsto \left(\ppartial^{(2\w)} \Phi(z_1) \cup \dots \cup \ppartial^{(2\w)} \Phi(z_\alpha) \right)$, let $\cH$ denote the event that there is $\hPhi''$ isomorphic to $\vec \hPhi$ such that $\cF_{y} = \cF$.
  Then, because $x$ is a random element of $V$, we have
 \begin{align}\nonumber \hProb \left[ \left. x \notin \core_{1/2} \left( \vec \hPhi \right) \setminus S_{1/2}\left(\vec \hPhi\right)   \right| \cF_y = \cF \right] &= \hProb \left[ \left. x \notin \core_{1/2} \left( \vec \hPhi \right) \setminus S_{1/2}\left(\vec \hPhi\right)   \right| \cH \right] 
 \\ \label{eq_aux_proof_aux_good_3_b} &= \hProb \left[ x \notin \core_{1/2} \left( \vec \hPhi \right) \setminus S_{1/2}\left(\vec \hPhi\right)    \right] + o_n(1), \end{align}
 where the last line used that $\hProb\left[ \cH \right] = 1-o_n(1)$. Finally, \Prop~\ref{prop_formulas_safe_structure} implies that
 \beq \label{eq_aux_proof_aux_good_3_c} \hProb \left[ x \notin \core_{1/2} \left( \vec \hPhi \right) \setminus S_{1/2}\left(\vec \hPhi\right)    \right] \leq 2^{1-0.96k} + o_n(1). \eeq
 Combining (\ref{eq_aux_proof_aux_good_3_a}), (\ref{eq_aux_proof_aux_good_3_b}) and (\ref{eq_aux_proof_aux_good_3_c}) concludes the proof of the lemma.
\end{proof}

\begin{proof}[Proof of \Lem~\ref{lemma_typical_properties}] For $\eps >0$ fixed, let $\vec Y =| \{ x \in V,\ \ppartial^{(2\w)} \vec \hPhi(x) \textrm{ is not $(\eps,2\w)$-cold in } \vec \hPhi  \}|$, and let $\alpha(n)$ be a slowly diverging function. We are going to show that there is a sequence $y_\w = o_\w(1)$ such that
\beq \label{eq_aux_typical_C6_a} \hErw \left[ \vec Y(\vec Y-1) \dots (\vec Y-\alpha +1 ) \right] \leq \left( y_\w n \right)^{\alpha}. \eeq
This bound implies the assertion; indeed,
\begin{align*} \hProb \left[ \vec Y > 3 y_\w n \right] & \leq \hProb \left[ \vec Y(\vec Y-1) \dots (\vec Y-\alpha+1) > (2 y_\w n)^\alpha \right]
\\ & \leq \frac{\hErw \left[ \vec Y(\vec Y-1) \dots (\vec Y-\alpha +1 ) \right] }{(2 y_\w n)^\alpha} \leq 2^{-\alpha}. \end{align*}
To prove (\ref{eq_aux_typical_C6_a}), we observe that $\vec Y(\vec Y-1) \dots (\vec Y-\alpha+1)$ is just the number of orderer $\alpha$-tuples of variables such that $ \ppartial^{(2\w)} \Phi(x)$ is not $(\eps,2\w)$-cold. Hence, by symmetry and linearity of expectation, 
$$ \hErw \left[ \vec Y (\vec Y-1) \dots (\vec Y-\alpha +1) \right] \leq n^\alpha \hProb \left[ \vec T_1, \dots, \vec T_\alpha \textrm{ are not $(\eps,2\w)$-cold} \right],$$
where $\vec T_1, \dots, \vec T_\alpha$ are the $2\w$- neighborhoods chosen of $\alpha$ random vertices $\vec x_1, \dots, \vec x_\alpha$ of $V$. Let $\cD$ be the event that $\vec x_1, \dots, \vec x_\alpha$ are at distance greater than $5 \w$ from each others and have tree-like $5 \w$ neighborhoods, and let $\vec \Delta = \partial \vec T_1 \cup \dots \cup \partial \vec T_\alpha$. Then \Lem~\ref{lemma_aux_good_3} implies that for $j \in \vec \Delta$
\beq \nonumber \hProb \left[ \left. \neg \cC_j \right|\cD, \cF_j \right] \leq 2^{-0.95k}.\eeq
In particular, using \Lem~\ref{lemma_typical_w_neighb} we can apply the result of \Prop~\ref{prop_good_trees_for_our_setting} to obtain that, for $1 \leq i \leq \alpha$,
\beq \label{eq_ccj_prime} \Prob \left[ \vec T_i \textrm{ is not $(\eps,2\w)$-cold} | \cD, \vec T_1, \dots, \vec T_{i-1}, \vec T_{i+1}, \vec T_\alpha \right] \leq \w^{-1}. \eeq
We have
\begin{align*}\hProb \left[ \vec T_1, \dots, \vec T_\alpha \textrm{ are not $(\eps,2\w)$-cold} \right] & \leq \hProb \left[ \vec T_1 \textrm{ is not $(\eps,2\w)$-cold} | \cD \right]  \hProb \left[ \vec T_2 \textrm{ is not $(\eps,2\w)$-cold} | \vec T_1 \textrm{ is not $(\eps,2\w)$-cold(, } \cD \right] \dots 
\\ & \hspace{1.5 cm} \hProb \left[ \vec T_\alpha \textrm{ is not $(\eps,2\w)$-cold} | \vec T_1, \dots, \vec T_{\alpha-1} \textrm{ are not $(\eps,2 \w)$-cold, }  \cD \right]. \end{align*}
Using (\ref{eq_ccj_prime}) this yields 
$$\hProb \left[ \vec T_1, \dots, \vec T_\alpha \textrm{ are not $(\eps,2 \w)$-cold} | \cD \right] \leq (o_\w(1))^\alpha. $$
Along with the observation that $\hProb \left[ \cD \right] = 1-o_n(1)$, this concludes the proof of the proposition. \end{proof}

In order to prove \Lem~\ref{lemma_typical_properties_planted}, we need to extend \Lem~\ref{lemma_aux_good_3} to the planted replica model. This will require a few more auxiliary results. We say that a tree $T \in \cT_{\w}$ is $2\w$-{\em pure} if and only if
$$ \nu^{(2\w)}_T(1) \geq 1- \exp(-100 \beta).$$
Let $\cT_{2\w}^+ \subset \cT_{2\w}$ denote the set of pure trees. Let, as before, $\alpha \geq 0$ and $(z_1, \dots, z_\alpha) \in V^\alpha$ be fixed, as well as a formula and an assignment $(\Phi,\sigma) \in \cE_{n,k,d} \times \{-1,1\}^n$. Let $\Delta = \{y \in V, \exists l \in [\alpha], y \in \partial^{(2\w)} \Phi(z_l) \}$ and for $y \in \Delta$, let $a_y$ be the unique clause in $\partial y \cap \cup_{l=1}^\alpha \ppartial^{(2\w)} \Phi(z_l)$ and let $l_y \in [k]$ (resp. $l'_y \in [d]$) be such that $y$ appears in $l_y$-th position in $a_y$ (resp. $a_y$ appears in $l_y'$ position in $y$). For $y \in \Delta$, let $\ppartial^{(2\w)} \vec \hhPhi(y \to a_y)$ denote the $2 \ell$ neighborhoof of $y$ in the formula where the edge between $y$ and $a_y$ has been removed and let \begin{itemize}
\item $\cA_y$ be the event that $\ppartial^{(2\w)} \vec \hhPhi(y \to a_y) $ is tree-like and is $2\w$-pure in $\vec \hhPhi$,
\item $\cB_y$ be the event that $\vec {\widetilde{\sigma}}(y) = 1$,
\item $\cC_y$ be the event that $y \in \core_{1/2}(\vec \hhPhi) \setminus S_{1/2}(\vec \hhPhi)$.
\end{itemize}
Moreover, let $\cD$ be the event that $z_1, \dots, z_\alpha$ are at distance greater than $5\w$ in $\vec \hhPhi$. For  $y \in \Delta$, let also $\cG_y$ denote the sigma algebra induced by the functions  $$(\Phi, z_1, \dots, z_\alpha,y) \mapsto \left(X=\ppartial^{(2\w)} \Phi(z_1) \cup \dots \cup \ppartial^{(2\w)} \Phi(z_\alpha) \setminus \ppartial^{(2\w)} \Phi(y \to a_y),\sigma_{|X} \right).$$ With these notations in mind, we will prove the following.

\begin{lemma} \label{lemma_aux_planted_good_1} For $y \in \Delta$, we have
$$ \hhProb \left[ \neg \cA_y | \cD, \cG_y \right] \leq 2^{-0.95k}.$$
\end{lemma}
\begin{proof}
This lemma follows from \Lem~\ref{lemma_most_likely_trunk} and \Lem~\ref{lemma_if_frozen_then_cold}, using in addition the fact that $\hhProb \left[ \cD\right] = 1-o_n(1)$.
\end{proof}

\begin{lemma}  \label{lemma_aux_planted_good_2}  For $y \in \Delta$, we have
$$ \hhProb \left[  \neg \cB_y  | \cA_y, \cD, \cG_y \right] \leq 4^{- k} .$$
 \end{lemma}

\begin{proof} Let $\widetilde{\widehat{\cT}}_{2\w+1}$ denote the set of $2\w+1$ neighborhoods of the first children of the root of the Galton-Watson process $\GW'(k,d,\beta,4\ell)$ (defined in \Sec~\ref{sec_galton_watson}). For a probability distribution $\mu$ over $\{-1,1\}^k$ and $1 \leq l \leq k$, let $\mu[l]$ denote the $l$-th marginal of $\mu$. Finally, recall that for $a \in F$, $\mu_a^{(2\w+1)}$ was defined in \Sec~\ref{sec_non_reconstruction_and_bethe}.

Recalling the definition of the replica planted model, we have
\begin{align*} \hhProb \left[ \neg \cB_y | \cA_y, \cD, \cG_y \right] &= \sum_{\hT \in \widetilde{\widehat{\cT}}_{\w}}  \hhProb \left[ \ppartial^{(2\w+1)} \vec \hhPhi(a_y) = \hT | \cA_y, \cD, \cG_y \right] \hmu^{(2\w+1)}_a[l_y](-1).
 \end{align*}
For $\hT \in  \widetilde{\widehat{\cT}}_{2\w+1} $ and $1 \leq l \leq k$, let $\hT[l]$ denote the subtree of size $2\w$ rooted at the $l$-th variable node adjacent to the root of $\hT$. For $T \in \cT_{2\w}^+$, let $\widehat{\cT}_{2\w+1}(T,l) \subset \widetilde{\widehat{\cT}}_{2\w+1} $ denote the set of trees compatible with $T$ on $l$-th position:
$$  \widehat{\cT}_{2\w+1}(T,l) =\left \{ \hT \in \widetilde{\widehat{\cT}}_{2\w+1},  \hT[l] = T \right \}  .$$
Then we immediately deduce from the previous equation that
$$\hhProb \left[ \neg \cB_y | \cA_y, \cF_y \right]  \leq  \sup_{T \in \cT_{2\w}^+ } \sup_{\hT \in \widehat{\cT}_{2\w+1}(T,l_y)}  \hmu^{(2\w+1)}_{\hT}[l](-1) (1 +o_n(1))$$
Using the definition of $\w$-pure trees, and the observation that marginals and messages cannot differ by a factor of more than $\exp(\beta)$, for any $T \in \cT_{2\w}^+$ and $\hT \in \widehat{\cT}_\w(T,l)$ we have $\hmu^{(2\w+1)}_{\hT}[l](-1) \leq \exp(-50 \beta) \leq 4^{- k}$, ending the proof of the lemma.
 \end{proof}

\begin{lemma}   \label{lemma_aux_planted_good_3}   For $y \in \Delta$, we have
$$ \hhProb \left[  \neg \cC_y  | \cA_y, \cB_y, \cD, \cG_{y} \right] \leq 2^{-0.95 k} .$$ 
 \end{lemma}

\begin{proof} Let $\vec \hhPhi'$ be obtained from $\vec \hhPhi$ by the following operation.
\begin{itemize}
\item Select $x \in \vec \hhPhi$ such that $\vec {\widetilde{\sigma}}_x = \vec {\widetilde{\sigma}}_y$ and (denoting by $a_x$ the $l_y'$-th clause adjacent to $x$) $\ppartial^{(4\w)}\vec \hhPhi(x,a_x) = \ppartial^{(4\w)} \vec \hhPhi(y,a_y) $ at random.
\item Replace the pair of edges $\{(y,a_y),(x,a_x)\}$ by the pair of edges $\{(y,a_x),(x,a_y)\}$.
\end{itemize}
Let $\cE$ be the event that $\vec \hhPhi'$ satisfies $\cD$. We observe that $\hProb \left[\cD\right] = 1-o_n(1)$ and $\hProb \left[\cE\right] = 1-o_n(1)$. Conditionned on $\cD, \cE$ and $\cG_y$, $\vec \hhPhi$ and $\vec \hhPhi'$ are identically distributed. Moreover, we have
$$ \core _{1/2} \left( \vec \hhPhi \right) \setminus S_{1/2}\left(\vec \hhPhi'\right) \subset \core _1\left(\vec \hhPhi' \right) \setminus  S_1\left(\vec \hhPhi '\right) . $$
It follows that
\begin{align} \nonumber \hhProb \left[ \neg \cC_y | \cA_y, \cB_y, \cD,\cG_y \right] & = \hhProb \left[ \neg \cC_y | \cA_y, \cB_y, \cD, \cE, \cG_y \right]  (1+o_n(1))
\\ \nonumber &\leq  \sum_{T' \in {\cT}_{4 \w}^+}  \hhProb \left[ \ppartial^{(4\w)} \vec \hhPhi(y,a_y) = T' \left| \cA_y, \cB_y, \cD, \cE, \cG_y \right. \right] 
\\ \label{eq_aux_lemma_aux_planted_good_a} & \hspace{2.5 cm}\hhProb \left[ x \notin \core_{1/2} \left( \vec \hhPhi \right) \setminus S_{1/2}\left(\vec \hhPhi\right)  \left| \cB_x, \cD, \cE, \cG_y, \ \ppartial^{(4\w)} \vec \hhPhi(x,a_x)= T' \right. \right].
 \end{align}
We define, for $T' \in \cT_{4 \w}^+$, $T(T') \in \widehat{\cT}_{4 \w+1}$ by $T[l_y'] = \ppartial^{(4\w)} \Phi(a_y \to y)$ and $T[l] = T'[l]$ for $l' \neq l_y'$ (where $T[l]$ denotes the $l$-th subtree pending on $T$'s root).
It follows from the same argument as previously for $T' \in \cT_{4 \w}^+$ we have $\mu^{(2\w+1)}_{T(T')}[l_y](1) \geq 1/2$. Therefore 
 \begin{align*} \hhProb \left[ \left . \ppartial^{(4\w)} \vec \hhPhi(y,a_y) = T' \right| \cA_y, \cB_y, \cD, \cE, \cG_y \right]  &= \frac{ \hProb \left[ \left . \ppartial^{(4\w)} \vec \hhPhi(y,a_y)  = T' \right|  \cA_y, \cG_y  \right]  \mu^{(2\w+1)}_{T(T')}[l_y](1)}{\sum_{T'' \in \cT_{4\w}^+} \hProb \left[ \left. \ppartial^{(4\w)} \vec \hhPhi(y,a_y)  = T'' \right| \cA_y, \cG_y  \right]  \mu^{(2\w+1)}_{T(T'')}[l_y](1)} (1+o_n(1))
 \\ & \leq \frac{ \hProb \left[ \left. \ppartial^{(4\w)} \vec \hhPhi(y,a_y)  = T' \right| \cA_y, \cG_y \right]  \mu^{(2\w+1)}_{T(T')}[l_y](1) }{\sum_{T'' \in \cT_{4\w}^+} \hProb \left[ \left. \ppartial^{(4\w)} \vec \hhPhi(y,a_y)  = T'' \right| \cA_y, \cG_y  \right]  1/2 }
 \\ & \leq 2 \hProb \left[ \left. \ppartial^{(4\w)} \vec \hhPhi(y,a_y)   = T' \right | \cA_y, \cG_y  \right], \end{align*}
 Moreover 
 \begin{align*}\hProb \left[ \left. \ppartial^{(4\w)} \vec \hhPhi(y,a_y)  = T' \right | \cA_y, \cG_y  \right] &\leq 
 \\ & \leq  {\hProb \left[  \ppartial^{(4\w)} \vec \hhPhi(y,a_y) = T' \right]} {\hProb \left[ \cA_y\right] }^{-1}  {\hProb \left[ \cG_y\right] }^{-1} 
 \\ & \leq 2 \hProb \left[ \ppartial^{(4\w)} \vec \hhPhi(y,a_y)  = T '\right] = 2 \hhProb \left[ \ppartial^{(4\w)} \vec \hhPhi(y,a_y)  = T'   \right].\end{align*}
 where we used \Lem~\ref{lemma_aux_planted_good_3} to obtain the second inequality. Using Baye's theorem once more, we have
 $$ \hhProb \left[\ppartial \vec \hhPhi^{(4\w)}(y,a_y)  = T'   \right] = \frac{ \hhProb \left[ \left. \ppartial^{(4\w)} \vec \hhPhi(y,a_y) = T' \right | \cB_y  \right] \hhProb \left[ \cB_y \right] }{ \hhProb \left[ \left.  \cB_y   \right |\ppartial \vec \hhPhi^{(4\w)}(y,a_y)  = T'\right]} \leq 2 \hhProb \left[ \left. \ppartial^{(4\w)} \vec \hhPhi(y,a_y)  = T' \right | \cB_y  \right].$$
 In order to deduce the last inequality, we used that by an argument similar to \Lem~\ref{lemma_aux_planted_good_2}, $ \hhProb \left[ \left.  \cB_y   \right | \ppartial^{(4\w)} \vec \hhPhi(y,a_y) \right] \geq 1/2$. It follows by replacing in (\ref{eq_aux_lemma_aux_planted_good_a}) that
 \begin{align*} \hhProb \left[ \neg \cC_y | \cA_y, \cB_y, \cD, \cG_y \right] &\leq  6 \sum_{T' \in {\cT}_{4\w}^+}  \hhProb \left[ \left. \ppartial^{(4\w)} \vec \hhPhi(y,a_y)  = T' \right| \cB_y \right] \hhProb \left[ x \notin \core_{1/2} \left( \vec \hhPhi \right) \setminus S_{1/2}\left(\vec \hhPhi\right)  \left| \cB_x, \cD,\cE, \ppartial^{(4\w)} \vec \hhPhi(x,a_x)  = T' \right. \right]
 \\ & \leq 6 \hhProb \left[ x \notin \core_{1/2} \left( \vec \hhPhi \right) \setminus S_{1/2}\left(\vec \hhPhi\right)  \left| \cB_x, \cD, \cE \right. \right] 
 \\ & \leq 6 {\hhProb \left[ x \notin \core_{1/2} \left( \vec \hhPhi \right) \setminus S_{1/2}\left(\vec \hhPhi\right)  \right] }{\hhProb \left[ \cB_x \right]}^{-1} {\hhProb \left[ \cD \right]}^{-1} {\hhProb \left[ \cE \right]}^{-1}
 \\ & \leq 8 {\hhProb \left[ x \notin \core_{1/2} \left( \vec \hhPhi \right) \setminus S_{1/2}\left(\vec \hhPhi\right)  \right] }.
 \end{align*}
We used \Lem~\ref{lemma_aux_planted_good_2} to deduce the last inequality. Along with \Prop~\ref{prop_formulas_safe_structure}, this ends the proof of the lemma.
\end{proof}

\begin{proof}[Proof of \Prop~\ref{proposition_typical_properties_planted}] We take a path similar to the proof of \Prop~\ref{proposition_typical_properties}. Let $\eps>0$ be fixed. Let $$\vec Y =| \{ x \in V,\ \ppartial^{(2\w)} \vec \hhPhi(x) \textrm{ is not $(\eps,2\w)$-cold in } \vec \hhPhi  \}|,$$ and let $\alpha(n)$ be a slowly diverging function. We are going to show that there is a sequence $y_\w = o_\w(1)$ such that
\beq \label{eq_aux_typical_C6_a_bis} \hhErw \left[ \vec Y(\vec Y-1) \dots (\vec Y-\alpha +1 ) \right] \leq \left( y_\w n \right)^{\alpha}. \eeq
This bound implies the assertion as previously. As before, we observe that
$$ \hhErw \left[ \vec Y (\vec Y-1) \dots (\vec Y-\alpha +1) \right] \leq n^\alpha \hhProb \left[ \vec T_1, \dots, \vec T_\alpha \textrm{ are not $(\eps,2\w)$-cold} \right],$$
where $\vec T_1, \dots, \vec T_\alpha$ are $2\w$- neighborhoods of $\alpha$ random vertices $\vec x_1, \dots, \vec x_\alpha$ of $V$. Let $\cD$ be the event that $\vec x_1, \dots, \vec x_\alpha$ are at distance greater than $5 \w$ from each others and have tree-like $5 \w$ neighborhoods, and let $\vec \Delta = \partial \vec T_1 \cup \dots \cup \partial \vec T_\alpha$. By combining \Lem~\ref{lemma_aux_planted_good_1}, \Lem~\ref{lemma_aux_planted_good_2}, \Lem~\ref{lemma_aux_planted_good_3} we obtain that for $y \in \vec \Delta$
\beq \nonumber \hhProb \left[ \left. \neg \cC_y \right|\cD, \cG_y \right] \leq 2^{-0.94k}.\eeq
In particular, using \Cor~\ref{cor_typical_w_neighb} we can apply the result of \Prop~\ref{prop_good_trees_for_our_setting} to obtain that, for $1 \leq i \leq \alpha$,
\beq \label{eq_ccj_prime_planted} \Prob \left[ \vec T_i \textrm{ is not $(\eps,2\w)$-cold} | \cD, \vec T_1, \dots, \vec T_{i-1}, \vec T_{i+1}, \vec T_\alpha \right] \leq \w^{-1}. \eeq
We have
\begin{align*}\hhProb \left[ \vec T_1, \dots, \vec T_\alpha \textrm{ are not $(\eps,2\w)$-cold} | \cD \right] & \leq \hhProb \left[ \vec T_1 \textrm{ is not $(\eps,2\w)$-cold} | \cD \right]  \hhProb \left[ \vec T_2 \textrm{ is not $(\eps,2\w)$-cold} | \vec T_1 \textrm{ is not $(\eps,2\w)$-cold, } \cD \right] \dots 
\\ & \hspace{1.5 cm} \hhProb \left[ \vec T_\alpha \textrm{ is not $(\eps,2\w)$-cold} | \vec T_1, \dots, \vec T_{\alpha-1} \textrm{ is not $(\eps,2\w)$-cold, }  \cD \right]. \end{align*}
Using (\ref{eq_ccj_prime_planted}) yields
$$\hhProb \left[ \vec T_1, \dots, \vec T_\alpha \textrm{ are not $(\eps,2\w)$-cold} | \cD \right] \leq (o_\w(1))^\alpha. $$
Along with the observation that $\hhProb \left[ \cD \right] = 1-o_n(1)$, this concludes the proof of the proposition.
 \end{proof}

\subsection{Proof of \Prop~\ref{prop_formulas_safe_structure}}

In order to prove \Prop~\ref{prop_formulas_safe_structure}, we will identify a set of simpler events that will imply the proposition. We will first need to control the number of vertices with unusual 2-neighborhood. To this end, we let for a formula $\Phi$, $U_0$ be the set of variables such that $\left| \{a \in \partial_{1,0} x \}  \right| < 2 k^{7/8}$, $|\partial_{-1,0} x | \geq 2$ or such that there exists $1 \leq l \leq k$ such that $ | \{\partial_{-1,l} x \} | \geq 0.01 k^{l+3}/l!$. Our first condition will ensure that $U_0$ is not too large:
\beq \label{tech_cond_0} \tag{$\mathcal{C}$0} |U_0| \leq 2^{-0.98k}n. \eeq
We now turn to expansion properties of $\Phi$. We define, for a set $T \subset V$ the sets
\begin{align*} F_0(T) &= \{ a \in F, |\partial_{1} a| = 1, \partial_{1} a \subset T, |\partial_{-1}(a) \cap T | \geq 1 \} ,
\\F_l(T) &= \{a \in F, \partial_{-1} a \cap T \neq \emptyset,  |\partial_{1}a | =l, |\partial_{1} a \cap T | \geq l/4 \}  \qquad \textrm{(for $1 \leq l \leq k$)}.
\end{align*}
The following conditions encompass bounds on the sizes of the sets $F_{i}(T)$ when $T$ has moderate size.
\begin{align}
\label{tech_cond_1} \tag{$\mathcal{C}$1} &\mbox{ There is no set $T \subset V$ of size $|T| \in [2^{-0.97 k} n, 2^{-k/20} n]$ and such that $|F_0(T)| \geq k^{3/4}|T|/100$.}
\\ \label{tech_cond_2} \tag{$\mathcal{C}$2} &\mbox{For each $1 \leq l \leq k$, there is no set $T \subset V$ of size $|T| \in [2^{-0.97 k} n ,  2^{-k/20} n ]$ }
\\ \nonumber & \hspace{5 cm} \mbox{ and such that $|F_l(T)| \geq k^{3/4}|T|/(100l^2)$.}
\end{align}

\begin{lemma}\label{lemma_aux_strongly_safe} Assume that $\Phi$ satisfies (\ref{tech_cond_0})-(\ref{tech_cond_2}). Then it satisfies {\bf Core} and {\bf Sticky}. \end{lemma}

\begin{proof} Let $\Phi$ be such that it satisfies (\ref{tech_cond_0})-(\ref{tech_cond_2}). We first prove that $\Phi$ does not admit a $1/2$-sticky set $S$ with $|S| \in [2^{-0.97 k}n,2^{-k/20}n]$. Indeed, let $S \subset V$ be a $1/2$-sticky set for $\Phi$ and let 
\begin{align*}
S_0 &= \left \{x \in S,\left | \{a \in \partial_{1,0} x, \partial(a,x) \cap S  \neq \emptyset \} \right| \geq k^{3/4}/2 \right \},
\\ S_l &= \left \{x \in S, \left | \ \{a \in \partial_{-1, l} x, ||\partial_{1}(a,x)| = l,  |\partial_{1}(a,x) \cap S| \geq l/4  \} \right| \geq k^{3/4}/2 \right \} \qquad \textrm{(for $1 \leq l \leq k-1$)}.
\end{align*}
We first observe that
\begin{align}  \label{eq_constraint_sticky_clauses} |F_0(S)| \geq {k^{3/4}} |S_0| /2, \qquad \textrm{and that for  $1 \leq l \leq k-1$ } |F_l(S)| \geq k^{3/4}{|S_l|}/(2l). \end{align}
Because $S$ is $1/2$-sticky, we have $S \subset S_0 \bigcup \cup_{l=1}^{k-1} S_l $ and therefore either $|S_1| \geq |S|/100$ or there is $1 \leq l \leq k-1$ such that $|S_l |\geq |S| / (100 l^2)$. In either case, it follows from (\ref{eq_constraint_sticky_clauses}) and (\ref{tech_cond_1})-(\ref{tech_cond_2}) that $|S_l| \notin [2^{-0.98k}n,2^{-k/20}n]$. Using that (for $0 \leq l \leq k-1$) $|S_l| \leq |S| \leq k |S_l |$ shows that $S$ has size outside the range $[2^{-0.97 k}n, 2^{-k/20}n]$.

We now turn to the study of the $1/2$-core of $\Phi$. Given $\Phi$, we consider the following {\em whitening} process. Let $ U=  U_0$ initially. While there is a variable $x \not\in U$ such that one of the following conditions occurs, add $x$ to $ U$.
\begin{itemize}
\item[(a)] $|\{a \in \partial_{1,0} x, |\partial_{-1,0} a \cap U| \geq 1\}| > k^{3/4}/2$.
\item[(b)] $| \{a \in \partial_{{-1}} x, |\partial_{1} a \cap U | \geq |\partial_1 a|/4 \} | > k^{3/4}/2$.
\end{itemize}
It is easily seen that the process converges. Let $ U_{\infty}$ be the resulting subset of $V$, then we have 
\beq \label{eq_core_whitening} \core( \Phi)_{1/2} = V \setminus U_\infty . \eeq
We are going to show that $ U_\infty$ cannot be too large. By condition (\ref{tech_cond_0}), we can assume that $|U_0| \leq 2^{-0.98k}n$. Assume for contradiction that $|U_\infty| \geq 2^{-0.97 kn}$ and let $U$ be the set obtained when precisely $ 2^{-0.97 k}n - |U_0|$ variables have been added to $U_0$. By construction each variable $x \in U$ has one of the following properties.
\begin{itemize}
\item[(00)] $x$ belongs to $U_0$,
\item[(0)] $x$ belongs to more than $k^{3/4}/2$ clauses $a$ with $\partial_{1}a = \{x\}$ and $|\partial_{-1} a \cap  U | \geq 1 $,
\item[($l$)] $x$ belongs to more than $k^{3/4}/2$ clauses $a\in  \partial_{{-1},l} x $ with $|\partial_{1} a \cap  U| \geq l/4$.
\end{itemize}
Let $ U_0 \subset  U$ be the set of variables $x \in  U$ that satisfy (00),  $ V_0 \subset  U$ be the set of variables $x \in  U$ that satisfy (0),  and for $1 \leq l \leq k-1$ $ V_l \subset  U$ be the set of variables $x \in  U$ that satisfy ($l$). As $| U| \leq | U_{0}| + | V_0| + \sum_{l=1}^k | V_l| $ and $|U_0| \leq |U| /k$, either $|V_0| \geq |U|/100 \geq 2^{-0.98 k}n$ or there is $l$ such that 
$|V_l| \geq |U|/(100 l^2) \geq 2^{-0.98k}n$. Either case is impossible by a similar reasonning as previously and we obtaind that$| U_\infty| \leq 2^{-0.97k}n$ \whp.
 \end{proof}

Studying $\vec \hPhi$ will be enough to obtain the information needed about $\vec \hhPhi$. Indeed, we shall obtain sufficiently strong estimates of the probability of events under the random formula $\vec \hPhi$ to transfer them into high probability statements for the biased distribution generating $\vec \hhPhi$. More precisely, say that $\vec \hPhi$ satisfies a property ($\mathcal{P}$) \textit{with very high probability} ($\wvhp$) iff ($\mathcal{P}$) has probability larger than $1-\exp\left(-2^{-0.99k}n \right)$ under $\vec \hPhi$. Then we can infer that ($\mathcal{P}$) has a large probability under the random formula $\vec \hhPhi$.

\begin{lemma}\label{lemma_wvhp_implies_whp} Let $\cA$ be an event. Assume that $\vec \hPhi$ satisfies $\cA$ \wvhp. Then $\vec \hhPhi$ satisfies $\cA$ \whp. \end{lemma}
\begin{proof} Without loss of generality we can assume that $\cA$ contains the event $$\{\textrm{all but $2^{-0.999k}n$ of the $2\w$ neighborhood of variables $x \in V$ consists of a pure tree}\}.$$ Reformulating the definition of the planted replica model, we see that
\begin{align}\nonumber  \hhProb\left[ \neg \cA \right] &=  \frac{\sum_{\Phi} \mathbf{1}\left[ \Phi \notin \cA \right] \hProb \left[ \vec \hPhi = \Phi \right] \exp( n B_{\Phi,\w})}{\sum_{\Phi} \hProb \left[ \vec \hPhi = \Phi \right] \exp( n B_{\Phi,\w})} 
\\ \label{eq_aux_lemma_whvp_whp} & \leq \frac{\sup_{\Phi \in \cA} \exp( n B_{\Phi,\w})}{\sum_{\Phi} \hProb \left[ \vec \hPhi = \Phi \right] \exp( n B_{\Phi,\w})} \hProb \left[ \neg \cA \right ].
\end{align}
We observe that $B_{\Phi,\w} \leq H(\mu_{\Phi}^{(2\w)})$. Moreover, for all pure trees $T$ $H(\mu_T^{(\w)}) \leq 4^{-k}$. It therefore follows that
$$ \sup_{\Phi \in \cA} \exp( n B_{\Phi,\w}) \leq \exp \left[ n H(\mu_\Phi^{(2\w)})\right] \leq \exp \left (2^{-0.999k} n \right). $$
Returning to the definition of $\cB^{(\w)}(k,d,\beta)$ in \Sec~\ref{sec_finite_w_approx} we obtain on the other hand
\begin{align*}  \sum_{\Phi} \hProb \left[ \vec \hPhi = \Phi \right] \exp( n B_{\Phi,\w}) \geq 1/2 \exp \left( n \cB^{(\w)}(k,d,\beta) \right) \geq \exp \left( -2^{-0.999k}n \right),  \end{align*}
where the last estimate follows from an analysis similar as previsously. Replacing in (\ref{eq_aux_lemma_whvp_whp}) yields 
$$\hhProb \left[ \neg \cA \right] \leq 2 \exp\left( 2^{1-0.999k}n \right) \hProb \left[ \neg \cA \right] \leq 2\exp\left( 2^{1-0.999k}n -2^{-0.99 k}n\right)= o_n(1),$$
as desired.
\end{proof}

In order to obtain our result, we are thus left with proving the following proposition.

\begin{proposition} \label{prop_typical_properties_splitted}  \Wvhp~ $\vec \hPhi$ satisfies (\ref{tech_cond_0}), (\ref{tech_cond_1}) and (\ref{tech_cond_2}). \end{proposition}

\begin{proof}[Proof of \Prop~\ref{prop_formulas_safe_structure}] The propositions follow from combining \Lem~\ref{lemma_aux_strongly_safe} combined with \Lem~\ref{lemma_wvhp_implies_whp} and \Prop~\ref{prop_typical_properties_splitted}. \end{proof}

\subsection{Proof of \Prop~\ref{prop_typical_properties_splitted}}
\label{subsec_exp_properties}

In this section we shall study typical properties of the random formula $\vec \hPhi$. For a formula $\Phi$ and $0 \leq l \leq k$, we let $m_{l}(\Phi)$ count the number of clauses $a$ of $\Phi$ such that $| \delta_{1} a | = l$. 

 \begin{lemma}\label{lemma_typical_C0} \Wvhp~ $\vec \hPhi$ satisfies (\ref{tech_cond_0}). \end{lemma}

\begin{proof} Let $\vec Y_1$ denote the number of variables $x \in V$ such that $\left| \{a \in \partial_{1,0} x\}  \right| \leq 4 k^{7/8}$. Let $p$ denote the probability that a binomial of parameters $(k-1,1- q)$ takes values $0$. Using \Lem~\ref{lemma_typical_w_neighb} and recalling the definition of $\widetilde{p}_{k,d,\beta}^{(\w)}$ in \Sec~\ref{sec_operator_tree} gives 
$$ \hErw[\vec Y_1 ] \leq \sum_{r =0}^{4 k^{7/8}-1} \binom{d/2}{r} p^r (1-p)^{d/2-r}. $$
A simple computation reveals that $p = 2^{1-k} + \tilde O_k(4^{-k})$. This implies that the summand is maximal for $r = 4^{k^{7/8}}$ and allows to bound $\hErw[\vec Y_1]$ as
\begin{align*} \hErw[\vec Y_1 ] &\leq  O_k(k^{10}) \left( \frac{k e (\ln 2)}{k^{7/8}} \right)^{4 k^{7/8}} 2^{-k} + \tilde O_k(4^{-k}) = \tilde O_k(2^{-k}). \end{align*} A standard concentration argument then yields that $\vec Y_1 \leq 2^{-0.99k}$ \wvhp.

Similarly, let  $\vec Y_2$ denote the number of variables $x \in V$ such that $\partial_{-1} x \geq 2$. Let $Q$ denote the probability that a binomial of parameter $(d/2, \exp(-\beta) q^{k-1}/(1-c_\beta q^{k-1})$ takes a value larger than $2$. By another simple computation, we find $Q = \tilde O_k(2^{-k})$. It follows from \Lem~\ref{lemma_typical_w_neighb} that $\hErw[\vec Y_2] = Q = \tilde O_k(2^{-k})$. Again, by concentration this implies $\vec Y_2= \leq 2^{-0.99 k}$ \wvhp.

Finally, for $1 \leq l \leq k$ let $\vec Y_3(l)$ be the number of variables $x \in V$ with $|\partial_{{1},l} i| \geq k^{l+3}/l!$. By similar computations, we obtain
\begin{align*} \hErw[\vec Y_3(l) ] &\leq \sum_{r =k^{l+3}/l!}^{d/2} \binom{d/2}{r}  \binom{k-1}{l} \frac{1}{2^k} \left(1+ \tilde O_k(2^{-k}) \right) = \tilde O_k(2^{-k}). \end{align*}
It follows that $\hErw[\vec Y_3] = \tilde O_k(2^{-k})$, and by the same concentration argument as previously $\vec Y_3 \leq 2^{-0.99k}$ \wvhp.

The proof of the lemma is completed by noting that $| \vec U_0| \leq \vec Y_1 + \vec Y_2 + \vec Y_3$.
\end{proof}

We define $m'_l = \frac{1}{200} \frac{d}{2^k} k^{l+3}/l!$. The previous estimates can easily be (slightly extended and) recast as follows.

\begin{remark} \label{remark_control_clauses_types} \Wvhp~ we have for all $0 \leq l \leq k$, $m_l(\vec \hPhi) \leq m_l'$. \end{remark}

We are now ready to complete
 \begin{lemma}\label{lemma_typical_C1} \Wvhp~ $\vec \hPhi$ satisfies (\ref{tech_cond_1}). \end{lemma}

\begin{proof}
Given $\vec \hPhi$, let $\vec X_0(t,r,y)$ count the number of sets $T \subset V$ of size $|T|  = t n$, such that 
\begin{itemize}
\item $|F_0(T)| = r t n$.
\item $\sum_{a \in F_0(T)} | \partial_{-1}a \cap T| = yrtkn$.
\end{itemize}
By definition of $F_1(T)$, $\vec X_0(t,r,y)=0$ if $y < k^{-1}$. The expected value of $\vec X_0(t,r,y)$ can be computed in the following manner. First choose the sets $T$ and $F_0(T)$. The latter has to be chosen among the $m_{0} (\vec \hPhi)$ satisfied clauses.  Among the $tdn$ literal clones from $T$, choose the $rtn$ positive literal clones that will be connected to the positive literal clones of clauses in $F_0(T)$, and the $ytdnk$ literal clones that will be connected to negative literal clones of clauses in $F_0(T)$. Make the same choices among the negative and positive literal clones of the clauses in $F_0(T)$. Then match these $rtn$ positive literal clones (resp. $yrtkn$ negative literal clones) at random, and then match the remaining $dn/2-rtn$ remaining positive literal clones (resp. $dn/2-rtn$ remaining negative literal clones) at random. The normalizing factor is the total number of graphs that can be obtained from the configuration model, $(dn/2)!^2$. Without words, and using in addtion \Rem~\ref{remark_control_clauses_types} to observe that we can assume $m_0(\vec \hPhi) \leq m_0' = \frac{d}{2^k} k^3$, this gives
\begin{align*} \hErw[\vec X_0(t,r)] &\leq \binom{n}{tn} \binom{m_0'}{ r t n } \binom{tdn}{rtn} \binom{tdn}{yrtkn} \binom{rtkn}{rtn} \binom{rtkn}{yrtkn} \frac{(rtnk)!  (dn/2 -rtnk)! (dn/2)!}{(dn/2)!^2}
\\&\leq \binom{n}{tn} \binom{m_0'}{ r t n } \binom{tdn}{rtn} \binom{tdn}{yrtkn} \binom{rtkn}{rtn} \binom{rtkn}{yrtkn} \binom{dn/2}{rtn}^{-1} \binom{dn/2}{yrtkn}^{-1}.  \end{align*}
We shall bound this quantity by using the bounds, for $1 \leq a \leq b$ and $n > 0$
\beq \label{eq_aux_bound_binomial} b \ln \left( \frac{a}{b} \right) \leq \frac{1}{n} \ln \binom{an}{bn} \leq b \ln \left( \frac{ae}{b} \right)  . \eeq
This yields
\begin{align*} \frac{1}{n} \ln \hErw[\vec X_0(t,r,y)] &\leq t \ln \left( \frac{e}{t} \right) +  rt \ln \left( \frac{dk^3 }{2^{k} t r } \right) +rt \ln \left ( 2k e^2 t \right) + y r t k  \ln \left( \frac{2e^2 t}{y} \right).
 \end{align*}
In particular for $r \geq k^{3/4}$, $t \in  [2^{-0.98k}n, 2^{-k/20}n]$, and $y \geq 1/k$, we get
\begin{align*}  \frac{1}{n} \ln \hErw[\vec X_0(t,r,y)]& \leq - t \ln t + t + k^{3/4} t \ln \left( k^{10} t \right) \leq - 2^{-0.98k} n. \end{align*}
In particular $$\sum_{\substack{t \in  [2^{-0.98k}, 2^{-k/20}] \\ tn \in \mathbb{N}}}  \sum_{\substack{r \in [k^{3/4}, d]\\ rtn \in \mathbb{N} }}  \sum_{\substack{y \in [0,1]\\ yrtn \in \mathbb{N} }} \hErw \left[ \vec X_0(t, r,y) \right] \leq \exp \left[ - 2^{-0.985k}n \right].$$ This implies by Markov's inequality that \wvhp~there are no sets $T$ of size $|T| \in [2^{-0.98k}n, 2^{-k/20}n]$ such that $|F_0(T)| \geq  k^{3/4} |T|$.
\end{proof}

\begin{lemma}\label{lemma_typical_C2} \Wvhp~ $\vec \hPhi$ satisfies (\ref{tech_cond_2}). \end{lemma}

\begin{proof}
Given $\vec \hPhi$ and $1 \leq l \leq k$, let $\vec X_l(t,r,x,y)$ count the number of sets $T \subset V$ of size $|T|  = t n$ and such that the following condition are true.
\begin{itemize}
\item $|F_l(T)| = r t n$.
\item $\sum_{a \in F_l(T)} | \partial_{1}a \cap T| = x rt k n$.
\item $\sum_{a \in F_l(T)} | \partial_{-1}a \cap T| = y rt k n$.
\end{itemize}
By definition of $F_l(T)$, $\vec X_l(t,r,x,y) = 0$ if $x < l k^{-1}/4$ or $y < k^{-1}$. With \Rem~\ref{remark_control_clauses_types} we can assume $$m_l(\vec \hPhi)  \leq m'_l.$$
Reasoning as before, we obtain
\begin{align*} \hErw[\vec X_l(t,r,x,y)] &\leq \binom{n}{tn} \binom{m_{l}'}{rtn} \binom{tdn}{xrtkn} \binom{tdn}{yrtkn} \binom{rtkn}{xrtkn} \binom{rtkn}{yrtkn} \binom{dn/2}{xrtnk}^{-1} \binom{dn/2}{yrtkn}^{-1}.\end{align*}
Taking logarithm and using (\ref{eq_aux_bound_binomial}), we obtain
\begin{align*} \frac{1}{n} \ln \hErw[\vec X_l(t,r,x,y)] &\leq t \ln \left( \frac{e}{t} \right) + rt \ln \left( \frac{d k^{l+3}}{2^{k} l! rt } \right)+ xrtk \ln \left( \frac{2e^2 t}{x} \right) + yrtk \ln \left( \frac{2e^2 t}{y} \right). \end{align*}
In particular, for $r \geq k^{3/4}/(100l^2)$, $t \in [ 2^{-0.98k}, 2^{-k/20}]$ and $x \geq l k^{-1} / 4, y \geq k^{-1}$, we have
\begin{align*} \frac{1}{n} \ln \hErw[\vec X_l(t,r,x,y)] &\leq t \ln \left( \frac{e}{t} \right) +  rt \ln \left( \frac{d k^{l+6} t^{l/4+1}}{ 2^kr /l!} \right) 
 \leq - t \ln t + t + \frac{k^{3/4}}{100 l^2} t \ln \left( k^{l+8} t^{l/4+1} \right) \end{align*}
 For any $1 \leq l \leq k$ we have $k^{l+6} t^{l/4+1} \leq k^{l^2}$ and we thereby obtain that
 $$ \frac{1}{n} \ln \hErw[\vec X_l(t,r,x,y)] \leq  - 2^{-0.98k} n.$$
This entails that, for any $1\leq l \leq k$, $$\sum_{\substack{t \in [2^{-0.98k} ,2^{-k/20}] \\ tn \in \mathbb{N}}} \sum_{\substack{r \in [k^{3/4}/(100l^2), d]\\ rtn \in \mathbb{N} }} \sum_{\substack{x \in [0, 1] \\ xrtn \in \mathbb{N}}} \sum_{\substack{y \in [0, 1] \\ yrtn \in \mathbb{N}}} \hErw \left[ \vec X_l(t,r,x,y) \right] \leq \exp \left[ - 2^{-0.985k} n \right].$$ This implies by Markov's inequality that \wvhp~there are no $1 \leq l \leq k$ and no sets $T$ of size $|T| \in [2^{-0.98k}n, 2^{-k/20} n]$ such that $|F_l(T)| \geq k^{3/4}  |T|/(100 l^2)$.
\end{proof}

\section{Moment computations}
 \label{sec_vanilla}
 
 {\em In this section we prove \Lem~\ref{Lemma_trivial}, \Lem~\ref{Lem_smmRemedy} and \Prop~\ref{prop_first_moment_vanilla}. We recall that $q = q(k,d,\beta)$ was defined in \Sec~\ref{sec_galton_watson}, Eq. (\ref{eq_def_q}).}
 
 \subsection{Preliminaries}
  
 We will need the following version of the inverse function theorem.
 
 \begin{lemma} \label{lemma_implicit_functions_theorem} Let $U \subset \mathbb{R}^h$ be an open set and let $f \in {C}^\infty(U)$. Assume that $u \in U$ and $r>0$ are such that
 $$ \{x \in \mathbb{R}^h : \| x - u \|_2 \leq r \} \subset U.$$
 Let $Df(x)$ be the Jacobian matrix of $f$ at $x$, $\rm{id}$ the identity matrix, and $\| \cdot \|$ the operator norm over $L^2(\mathbb{R}^h)$. Assume that $Df(u) = \rm{id}$ and
 $$ \| Df(x)- \rm{id} \| \leq \frac{1}{3} \qquad \textrm{for all $x \in \mathbb{R}^h$ such that $\|x-u\|_2 \leq r$}.$$
 Then for each $y \in \mathbb{R}^h$ such that $\|y-f(u) \| \leq r/2$ there is precisely one $x \in \mathbb{R}^h$ such that $\|x-u\| \leq r$ and $f(x)=y$. Furthermore, the inverse map $f^{-1}$ is $C^\infty$ on $\{x \in \mathbb{R}^h : \| x-u\| < r\}$, and $Df^{-1}(x) = (Df(x))^{-1}$ on this set.
   \end{lemma}
   
   We will also need the following result on the large deviation function of the multinomial distribution.
   
 \begin{lemma} \label{lemma_deviation_binomial} Let $l \geq 2$ and $(p_1, \dots, p_l) \in (0,1)^l$ satisfying $\sum_{j=1}^l p_j = 1$ be fixed. We have, for any $(q_1, \dots, q_l) \in (0,1)^l$ satisfying $\sum_{j=1}^l q_j = 1$
 $$ \frac{1}{n} \ln P \left[ \forall j \in [l], \ \left| \mathrm{Multinomial}(n, p_1, \dots, p_l)_j - n q_j \right| \leq 0.01 \sqrt{n} \right] = \sum_{j=1}^l q_j \ln \left( \frac{p_j}{q_j} \right)  + o_n(1).$$ \end{lemma}

Finally, we will need the following concentration result.

\begin{lemma} \label{lemma_concentration} Let $d \leq \dplus$ and $\beta \in \mathbb{R}$ be fixed. For any $\alpha >0$ there is $\delta >0$ such that
\begin{align*} & \Prob \left[ \left | \frac{1}{n} \ln Z_{\vec \Phi}( \beta)-\frac{1}{n} \Erw \ln \left[ Z_{\vec \Phi}(\beta) \right] \right| > \alpha \right] < \exp(-\delta n),
\\ & \hProb \left[ \left | \frac{1}{n} \ln \cC_{\vec \hPhi,\hat\SIGMA}(\beta)-\frac{1}{n} \hErw \ln \left[ \cC_{\vec \hPhi,\hat\SIGMA}(\beta) \right] \right| > \alpha \right] < \exp(-\delta n).
\end{align*} 
  \end{lemma}
\begin{proof} The proof follows from the fact that if two formula $\Phi, \Phi' $ differ by at most one switch of edges, the associated partition functions satisfy
$$ |\ln Z_\Phi (\beta) -\ln Z_{\Phi'}(\beta)| \leq 2\beta \qquad \textrm{and, for $\sigma \in \{-1,1\}^n$} \qquad |\ln \cC_{\Phi,\sigma}( \beta) -\ln \cC_{\Phi',\sigma} (\beta)| \leq 2\beta .$$
The stated concentration result is then a consequence of Azuma's inequality (applied to the configuration model). 
\end{proof}

 \subsection{The first moment computation}

Let $z_{1,k} : \mathbb{R} \times (0,1) \to (0,1]$ be defined by
$$ z_{1,k}(\beta,h) = 1-c_\beta (1-h)^k ,$$
the Kullback-Leibler divergence $D_1 : (0,1)^2 \to \mathbb{R}$ be defined by
$$ D_1(\alpha,h) =  \alpha \ln \left( \frac \alpha h \right) + (1-\alpha) \ln \left( \frac {1-\alpha}{1-h} \right),$$
and $f_{1,k} : \mathbb{R}^2 \to \mathbb{R}$ be defined by
$$ f_{1,k}(d, \beta) = \ln 2 + \frac{d}{k} \ln z_{1,k}(\beta,1-q) + d D_1\left( \frac{1}{2},1-q \right).$$

\begin{proof}[Proof of \Prop~\ref{prop_first_moment_vanilla}] We need to compute the expected value of $\prod_{a \in F} \psi_{a,\beta}(\sigma)$ under a random assignment $\sigma \in \{-1,1\}^n$. To do this, we introduce a different probability space formed of all vectors in $\{-1,1\}^{km} \times \{0,1\}^m$
$$ (\phi_{al})_{a \in [m], l \in [k]} , (y_a)_{a \in [m]}$$
with a probability distribution $\mathbb{P}$ such that the $(\phi_{al})_{a \in [m], l \in [k]}$ are independent random variables distributed as $\mathbb{P}(\phi_{al}=1)=1-q$ and the $(y_a)_{a \in [m]}$ are independent Bernoulli random variables of parameter $\exp(-\beta)/(1+\exp(-\beta))$.
We consider the two events
$$S=\left \{ \forall \ a \in [m] \ (y_a = 0 \textrm{ and }  \exists \ l \in [k], \phi_{al}=1) \textrm{ or } \left(y_{a}=1 \textrm{ and } \forall \ l \in [k], \phi_{al}=-1 \right) \right \}$$
and 
$$B= \left \{ \left | |\{(a,l):\phi_{al}={1}\}|-\frac{d}{2}n \right| \leq \sqrt{n} \right \}.$$
Then, using that $\psi_{a,\beta}(\phi_{al}) =\left(1 + \exp(-\beta) \right) \left( \mathbb{P} \left[ y_a = 0 \right] \mathbf{1}_{\exists \ l \in [k], \phi_{al}=1}+\mathbb{P} \left[ y_a = 0 \right] \mathbf{1}_{\forall \ l \in [k], \phi_{al}=-1} \right)$, we see the expected value of $\prod_{a \in F} \psi_{a,\beta}$ under \textit{any} given assignment $\sigma \in \{-1,1\}^n$ is given by $\mathbb{P}[S|B](1+\exp(-\beta))^m$. In particular,
\beq \label{eq_aux_vanilla_rate_function_1_a} \frac{1}{n} \ln \Erw \left[ Z_\beta(\vec \Phi) \right] \sim \ln2 + \frac{1}{n} \ln \mathbb{P} [S|B] + \frac{d}{k} \ln(1+\exp(-\beta)) .\eeq
By Bayes' theorem we have
\beq \label{eq_aux_vanilla_rate_function_1_aa} \mathbb{P} [S|B] = \frac{\mathbb{P}[S] \mathbb{P}[B|S]}{\mathbb{P}[B]} . \eeq
It follows from \Lem~\ref{lemma_deviation_binomial} that
\beq \label{eq_aux_vanilla_rate_function_1_b} \frac{1}{km} \ln \mathbb{P}[B] \sim -D_1 \left(\frac{1}{2}, 1-q \right).\eeq
It is also straightforward to obtain that
\beq \label{eq_aux_vanilla_rate_function_1_c} \frac{1}{m} \ln \mathbb{P}[S] \sim z_{1,k}(\beta,1-q) -\ln (1+\exp(-\beta)),\eeq
and by definition of $q$ we have (using the central limit theorem)
\beq \label{eq_aux_vanilla_rate_function_1_d} \frac{1}{m} \ln \mathbb{P}[B|S] = o_n(1).\eeq
The proposition is obtained by combining Eq.(\ref{eq_aux_vanilla_rate_function_1_a}-\ref{eq_aux_vanilla_rate_function_1_d}).
\end{proof}

 \subsection{The second moment computation}

Recall that $c_\beta = 1-\exp(-\beta)$. Let $\cT = \{(h,\hq) \in (0,1)^2, \hq < h \}$. Let $z_{2,k} : \mathbb{R} \times \cT \to (0,1]$ be defined by
$$ z_{2,k}(\beta,h,\hq) = 1-2 c_\beta (1-h)^k + c_\beta^2(1-2h+\hq)^k. $$

\begin{lemma}\label{lemma_vanilla_2_implicit} Let $g_{2,k,\beta}: \cT \to \mathbb{R}^2$ be defined by
$$ g_{2,k,\beta}(h,\hq) = \left( \frac{\hq + (h-\hq ) \left[ 1-c_\beta(1-h)^{k-1} \right]}{z_{2,k}(\beta, h, \hq)},  \frac{\hq}{z_{2,k}(\beta, h, \hq)} \right).$$
Let $\alpha \in (0,1)$ and let $\mathcal{U} = \{(x,y) \in \mathbb{R}^2, \| (x,y)-  (\frac12, \frac{1-\alpha}{2} )\|_2 \leq  k 2^{-k} \}$. Then the equation $g_{2,k,\beta}(h,\hq) = ( \frac{1}{2},\frac{1-\alpha}{2})$ admits a unique solution in $\cT \cap \mathcal{U}$ that we denote by $(h_\beta(\alpha), \hq_\beta(\alpha))$.
Moreover, $\alpha \to h_\beta(\alpha)$ (resp. $\alpha \to \hq_\beta(\alpha)$) is of class $C^\infty$ on $(0,1)$ and the following is true.
\begin{align} \label{eq_lemma_v_2_imp_1}\hq_\beta \left(1/2 \right) &= h_\beta \left( 1/2\right)^2 = q^2,
\\ \label{eq_lemma_v_2_imp_2}  \hq_\beta'(\alpha) &=  h_\beta'(\alpha) -1/2 + \tilde O_k(2^{-4k/3}) \qquad \textrm{for $|\alpha-1/2| \leq 2^{-k/3}$}.
\end{align}
\end{lemma}

\begin{proof} The Jacobian matrix $D g_{2,k,\beta}(h,\hq)$ of $g_{2,k,\beta}$ at $(h, \hq) \in\cT$ is given by $D g_{2,k,\beta}( h, \hq) = \textrm{id} + \tilde O_k(2^{-k})$; in particular it satisfies $\| Dg_{2,k,\beta}(h,\hq) - \textrm{id} \| \leq 1/3$. Then \Lem~\ref{lemma_implicit_functions_theorem} applied to $g_{2,k,\beta} \in C^\infty(\cT)$ with $y=u = \left(\frac{1}{2}, \frac{1-\alpha}{2} \right)$ and $r = k 2^{-k}$ imply that there is exactly one $(h_\beta(\alpha),\hq_\beta(\alpha)) \in \cT$ such that $\| (h_\beta(\alpha),\hq_\beta(\alpha))- \left( \frac{1}{2}, \frac{1-\alpha}{2} \right) \|_2 \leq r$ and $g_{2,k,\beta}(h_\beta(\alpha),\hq_\beta(\alpha)) = \left( \frac12, \frac{1-\alpha}2 \right)$. Moreover, the map $\alpha \to (h_\beta(\alpha), \hq_\beta(\alpha))$ is of class $C^\infty$ and $(h_\beta'(\alpha),\hq_\beta'(\alpha))  =(0,-1/2)+\tilde O_k(2^{-k})$. A more detailled computation (using $\frac{d}{d \alpha} g_{2,k,\beta}(h_\beta(\alpha),\hq_\beta(\alpha)) = (0,-1/2)$ and the chain rule for computing derivatives) reveals that, for $|\alpha -1/2| \leq 2^{-k/3}$ 
\begin{align*} \hq_\beta'(\alpha) +\frac{1}{2} z_{2,k}(\beta, h_\beta(\alpha), \hq_\beta(\alpha)) &= \frac{1}{4} \left( h_\beta'(\alpha) \frac{\partial z_{2,k}}{\partial h} + \hq_\beta'(\alpha) \frac{\partial z_{2,k}}{\partial \hq}\right) (\beta, h_\beta(\alpha), \hq_\beta(\alpha))  = \tilde O_k(4^{-k}),
\\  h_\beta'(\alpha) + c_\beta \hq'_\beta(\alpha) 2^{1-k} + \tilde O_k(2^{-4k/3}) &= \frac{1}{2} \left( h_\beta'(\alpha) \frac{\partial z_{2,k}}{\partial h} + \hq_\beta'(\alpha) \frac{\partial z_{2,k}}{\partial \hq}\right) (\beta, h_\beta(\alpha), \hq_\beta(\alpha))  = \tilde O_k(4^{-k}).  \end{align*}
In particular
$$ \hq_\beta'(\alpha) - h_\beta'(\alpha) +\frac12 = \frac{1}{2} \left(1-z_{2,k}(\beta, h_\beta(\alpha), \hq_\beta(\alpha)) \right)-2^{-k} c_\beta + \tilde O_k(2^{-4k/3}) = \tilde O_k(2^{-4k/3}).$$
Finally, (\ref{eq_lemma_v_2_imp_1}) is easily proved by inspection.
 \end{proof}
 
 In particular, we observe that \Prop~\ref{prop_first_moment_vanilla} and the above lemma imply the following.
 \begin{corollary} \label{cor_first_moment} We have $\lim_{n \to \infty} \frac{1}{n} \ln \Erw \left[ Z_\beta(\vec \Phi) \right]  = \frac{1}{2} f_{2,k}(d,\beta,1/2)$. \end{corollary}

Let the Kullback-Leibler divergence $D_2 : (0,1)^3 \to \mathbb{R}$ be defined by
$$ D_2(\alpha, h,\hq) = \alpha \ln \left( \frac{\alpha}{2 (h-\hq)} \right) + \frac{1-\alpha}{2} \ln \left( \frac{1-\alpha}{2 \hq} \right)+ \frac{1-\alpha}{2} \ln \left( \frac{1-\alpha}{2 (1-2h+\hq)} \right) ,$$
and $f_{2,k} : \mathbb{R}^2 \times (0,1) \to \mathbb{R}$ be defined by
$$f_{2,k}(d,\beta,\alpha) = \ln 2 + H(\alpha) + \frac{d}{k} \ln z_{2,k}(\beta, h_\beta(\alpha), \hq_\beta(\alpha)) + d D_2\left( \alpha,h_\beta(\alpha), \hq_\beta(\alpha)\right).$$
We also let \begin{align*} Z(d,\beta,\alpha)= \Erw \left [ \sum_{\substack{\sigma,\tau \\ \sigma \cdot \tau=(2\alpha-1) n}}  \prod_{a \in F} \left( \psi_{a,\beta}(\sigma_a)  \psi_{a,\beta}(\tau_a) \right) \right ].\end{align*}
so that $$\Erw \left[  Z^2_\beta (\vec \Phi) \right] = \sum_{\alpha \in \{0,1/n, \dots, 1\}} Z(d,\beta, \alpha).$$
\begin{proposition} \label{prop_second_moment_rate_function} Let $d>0, \beta \in \mathbb{R}$ and $I \subset [0,1]$ be fixed. We have
$$ \frac{1}{n}\ln \left(  \sum_{\alpha \in \{0,1/n, \dots, 1\}\cap I} Z(d,\beta, \alpha) \right) \sim \sup_{\alpha \in I} f_{2,k}(d,\beta,\alpha) $$ and in particular
$$ \frac{1}{n} \ln \Erw \left[  Z_\beta^2 (\vec \Phi) \right] \sim \sup_{\alpha \in (0,1)} f_{2,k}(d,\beta,\alpha).$$
\end{proposition}
\begin{proof}
We need to compute the expected value of $\prod_{a \in F} \psi_{a,\beta}(\sigma) \psi_{a,\beta}(\tau)$ under a random pair of assignements $(\sigma,\tau) \in \{-1,1\}^{2n}$. To do this, we introduce a different probability space formed of all vectors in $\{-1,1\}^{2km}\times \{0,1\}^{2m}$
$$ (\phi_{al})_{a \in [m], l \in [k]} , (y_a^{(1)},y_a^{(2)})_{a \in [m]}$$
with a probability distribution $\mathbb{P}$ such that the $(\phi_{al})_{a \in [m], l \in [k]}$ independent random variables satisfying
$$\phi_{al}=
\begin{cases}
(1,1), &\text{with probability } \hq_{\beta}(\alpha)\\
(1,-1), &\text{with probability } h_{\beta}(\alpha)-\hq_{\beta}(\alpha)\\
(-1,1), &\text{with probability } h_{\beta}(\alpha)-\hq_{\beta}(\alpha)\\
(-1,-1), &\text{with probability }  1-2h_{\beta}(\alpha)+\hq_{\beta}(\alpha)\\
\end{cases}$$
independently for all $a,l$, and the $(y_a^{(1)})_{a \in [m]}$ (resp. $(y_a^{(2)})_{a \in [m]}$) are independent Bernoulli random variables of parameter $\exp(-\beta)/(1+\exp(-\beta))$.
We consider the following events.
\begin{align*} S_1&=\{ \forall \ a \in [m] \ (y_a^{(1)} = 0 \textrm{ and }  \exists \ l \in [k], \phi_{al}\in \{(1,1),(1,-1)\})  
\\ & \hspace{4 cm}  \textrm{ or } \left(y_{a}^{(1)}=1 \textrm{ and } \forall \ l \in [k], \phi_{al}\in \{(-1,1),(-1,-1)\} \right)\},
 \\ S_2&=\{ \forall \ a \in [m] \ (y_a^{(2)} = 0 \textrm{ and }  \exists \ l \in [k], \phi_{al}=\in \{(1,1),(-1,1)\})
 \\ & \hspace{4 cm} \textrm{ or } \left(y_{a}^{(2)}=1 \textrm{ and } \forall \ l \in [k], \phi_{al}\in \{(1,-1),(-1,-1)\} \right)\},
 \\ S^{(2)} &= S_1 \cap S_2, \end{align*}
and 
\begin{align*}B^{(2)}=& \left \{ \left| |\{(a,l):\phi_{al}=(1,1)\}|-\frac{1-\alpha}{2}n\right| \leq \frac{1}{\sqrt{n}},\ \left| |\{(a,l):\phi_{al}=(-1,1)\}|-\frac{\alpha}{2}n\right| \leq \frac{1}{\sqrt{n}} \right. \\ 
& \left. \phantom{\{} \left| |\{(a,l):\phi_{al}=(-1,1)\}|-\frac{\alpha}{2}n\right| \leq \frac{1}{\sqrt{n}},\ \left| |\{(a,l):\phi_{al}=(-1,-1)\}|-\frac{1-\alpha}{2}n\right| \leq \frac{1}{\sqrt{n}} 
\right\}.\end{align*}
Then the expected value of $\prod_{a \in F} \psi_{a,\beta}(\sigma)  \psi_{a,\beta}(\tau)$ under \textit{any} given pair of assignments $(\sigma,\tau) \in \{-1,1\}^{2n}$ that satisfies 
\begin{align*} & \left| |\{i \in [n]:(\sigma_{i},\tau_i)=(1,1)\}|-\frac{1-\alpha}{2}n\right| \leq {\sqrt{n}},\ \left| |\{i\in [n]:(\sigma_{i},\tau_i)=(1,-1)\}|-\frac{\alpha}{2}n\right| \leq {\sqrt{n}}  \\ 
& \ \phantom{\{} \left| |\{i \in [n]:(\sigma_{i},\tau_i)=(-1,1)\}|-\frac{\alpha}{2}n\right| \leq {\sqrt{n}},\ \left| |\{i\in [n]:(\sigma_{i},\tau_i)=(-1,-1)\}|-\frac{1-\alpha}{2}n\right| \leq {\sqrt{n}} 
.\end{align*} is given as previously by $\mathbb{P}[S^{(2)}|B^{(2)}](1+\exp(-\beta))^{2m}$. In particular,
\beq \label{eq_aux_vanilla_rate_function_2_a} \frac{1}{n} \ln \Erw \left[ Z^2_\beta(\vec \Phi) \right] \sim \ln2 +H(\alpha)+ \frac{1}{n} \sup_{\alpha \in (0,1)} \ln \mathbb{P}[B^{(2)}|S^{(2)}] + \frac{2d}{k}\ln(1+\exp(-\beta)) .\eeq
By Bayes' theorem we have
\beq \label{eq_aux_vanilla_rate_function_2_aa} \mathbb{P} [ S^{(2)}|B^{(2)}] = \frac{\mathbb{P}[S^{(2)}] \mathbb{P}[B^{(2)}|S^{(2)}]}{\mathbb{P}[B^{(2)}]} . \eeq
It follows from \Lem~\ref{lemma_deviation_binomial} that
\beq \label{eq_aux_vanilla_rate_function_2_b} \frac{1}{km} \ln \mathbb{P}[B^{(2)}] \sim -D_2 \left(\alpha, h_\beta(\alpha),\hq_\beta(\alpha) \right).\eeq
It is also straightforward to obtain that
\beq \label{eq_aux_vanilla_rate_function_2_c} \frac{1}{m} \ln \mathbb{P}[S^{(2)}] \sim z_{2,k}(\beta,h_\beta(\alpha), \hq_\beta(\alpha)) - 2\ln (1+\exp(-\beta)),\eeq
and by definition of $h_\beta$ we have
\beq \label{eq_aux_vanilla_rate_function_2_d} \frac{1}{m} \ln \mathbb{P}[B^{(2)}|S^{(2)}] = o_n(1).\eeq
The proposition is obtained by combining Eq.(\ref{eq_aux_vanilla_rate_function_2_a}-\ref{eq_aux_vanilla_rate_function_2_d}).\end{proof}

\begin{lemma} \label{lemma_vanilla_non_local_0} Assume that $d \leq \dplus$ and $\beta \in \mathbb{R}$. Then we have
$$ \sup_{\alpha \geq 2^{-k/10}} f_{2,k}(d,\beta,\alpha) \leq f_{2,k}(d,\beta, 1/2) $$ \end{lemma}

\begin{lemma} \label{lemma_vanilla_non_local} Assume that $d \leq d_-(k)$ or that $d \in [d_-(k), \dplus]$ and that $\beta \leq \beta_-(k)$.Then we have
$$ \sup_{\alpha \in (0,1)} f_{2,k}(d,\beta,\alpha) \leq f_{2,k}(d,\beta, 1/2) $$ \end{lemma}

We defer the proof of these lemma to \Sec~\ref{sec_proof_lemma_vanilla_non_local}.

\begin{proof}[Proof of \Lem~\ref{Lemma_trivial}] The proposition follows by combining \Cor~\ref{cor_first_moment}, \Prop~\ref{prop_second_moment_rate_function} and \Lem~\ref{lemma_vanilla_non_local}.  \end{proof}

\subsection{Proof of \Lem~\ref{Lem_smmRemedy}}

To facilitate the proof of \Lem~\ref{Lem_smmRemedy} we introduce a random variable that explicitly controls the ``cluster size'' $\cC_{\Phi,\sigma}(\beta)$. More precisely, we call $\sigma \in \{-1,1\}^n$ {\em tame} in $\Phi$ iff 
	$$\cC_{\Phi,\sigma}(\beta)\leq \Erw[Z_\beta(\vec \Phi)].$$
Now, let
	$$Z_{\rm tame}(\Phi, \beta) = \sum_{\sigma \in \{-1,1\}^n}  \prod_{a \in F} \psi_{a,\beta}(\sigma_a) \mathbf{1}_{\sigma \textrm{\footnotesize \ is tame}}.$$
	
We shall also need to introduce a few more notations: we denote by $\um = (m_0, \dots, m_k)$ a vector of $[m]^{k+1}$, and by $\um(\Phi) = (m_0(\Phi), \dots, m_k(\Phi))$, with the $m_j(\Phi)$ as defined in \Sec~\ref{subsec_exp_properties}. Also recall that $\vec {\hat \sigma}$ denotes the all $1$ vector of length $n$.

\begin{lemma}\label{Lemma_antiPlanting}
Let $d \leq \dplus$ and $\beta \in \mathbb{R}$ be fixed.
Assume that
		$$\cProb \left[ \vec {\hat \sigma} \textrm{ isa tame in } \vec \hPhi \right] \geq \exp(o_n(n)).$$
Then
$$\Erw \left[Z_{\rm tame}(\vec \Phi,\beta \right] \geq \exp(o_n(n)) \Erw \left[Z_\beta(\vec \Phi) \right].$$ 
\end{lemma}

\begin{proof} 

Given that $\um(\vec \Phi) = \um(\vec \hPhi)$ the two formula $\vec \Phi$ and $\vec \hPhi$ are identically distributed. Thus we have for any $\um \in [m]^{k+1}$ \begin{align*} \Prob [ \vec {\hat \sigma} \textrm{ is not a tame in } \vec \Phi | \um(\vec \Phi) = \um] &= \cProb [ \vec {\hat \sigma} \textrm{ is not a tame in } \vec \hPhi | \um(\vec \hPhi ) = \um] . \end{align*}
In particular this implies that
\begin{align*} \Erw \left[ Z_\beta (\vec \Phi) - Z_{\rm{tame}} (\vec \Phi, \beta) \right] &= 2^n \sum_{\um \in [m]^{k+1}}  \Prob \left[ \vec {\hat \sigma} \textrm{ is not a tame in } \vec \Phi | \um(\vec \Phi) = \um  \right] \Prob \left[ \um(\vec \Phi) = \um \right] \exp(-\beta m_0)
\\ & \leq  2^n \sum_{\um \in [m]^{k+1}} \cProb [ \vec {\hat \sigma} \textrm{ is not a tame in } \vec \hPhi | \um(\vec \hPhi ) = \um]  \Prob \left[ \um(\vec \hPhi) = \um \right] \Erw \left[ Z_\beta(\vec \Phi )\right] 
\\ & \leq\cProb \left[ \vec {\hat \sigma} \textrm{ is not a  a tame in } \vec \hPhi \right]  \Erw \left[ Z_\beta(\vec \Phi) \right].  \end{align*}
This concludes the proof of the lemma.
 \end{proof}

\begin{lemma}\label{Lemma_smm}
Assume that $d\leq \dplus$ and $\beta \in \mathbb{R}$ are such that 
	$$\frac{\Erw[Z_{\rm tame}(\vec \Phi, \beta)]}{\Erw[Z_\beta(\vec \Phi)]}>\exp(o_n(n)).$$
Then $$\frac{\Erw[Z_{\rm tame}(\vec \Phi, \beta)]^2}{\Erw[Z_{\rm tame}(\vec \Phi, \beta)^2]}>\exp(o_n(n)).$$
\end{lemma}
\begin{proof}
We let
\begin{align*} Z_{\rm tame}(d,\beta,\alpha)= \Erw \left [ \sum_{\substack{\sigma,\tau \\ \sigma \cdot \tau=(2\alpha-1) n}}  \prod_{a \in F} \left( \psi_{a,\beta}(\sigma_a)  \psi_{a,\beta}(\tau_a) \right) \mathbf{1}_{\sigma \textrm{\footnotesize is a tame}} \mathbf{1}_{\tau \textrm{\footnotesize is a tame}} \right ] .\end{align*}
Then we have, by the definition of a ``tame''
\begin{align*} \sum_{\alpha \leq 2^{-k/10}} Z_{\rm tame}(d,\beta, \alpha) &\leq \Erw \left[ \sum_{\substack{\sigma \in \{-1,1\}^n}}  \prod_{a \in F}  \psi_{a,\beta}(\sigma_a) \mathbf{1}_{\sigma \textrm{\footnotesize is a tame}} \cC_{\Phi,\sigma}(\beta)\right]
\\ & \leq \Erw \left[ Z_\beta (\vec \Phi) \right]^2 \end{align*}
On the other hand we have with \Lem~\ref{lemma_vanilla_non_local_0}
$$ \sum_{\alpha \geq 2^{-k/10}} Z_{\rm tame}(d, \beta, \alpha) \leq \sum_{\alpha \geq 2^{-k/10}} Z(d, \beta, \alpha) =  \exp(o_n(n)) \Erw \left[ Z_\beta(\vec \Phi)\right]^2.$$
This implies that
\begin{align*} \Erw[Z_{\rm tame}(\vec \Phi, \beta)^2] &= \sum_{\alpha < 2^{-k/10}} Z_{\rm tame}(d, \beta, \alpha) +  \sum_{\alpha \geq 2^{-k/10}} Z_{\rm tame}(d, \beta, \alpha)
\\& = \exp(o_n(n)) O\left(\Erw \left[ Z_\beta(\vec \Phi)\right]^2\right).\end{align*}
The lemma then follows from the assumption that $\Erw \left[ Z_\beta(\vec \Phi)\right] \leq \exp(o_n(n)) \left( \Erw \left[ Z_{\rm tame}(\vec \Phi,\beta)\right] \right)$.
\end{proof}

The reverse direction of \Lem~\ref{Lem_smmRemedy} will be given by the following lemma.

\begin{lemma} \label{lemma_planting_direct} Assume that $d \in [d_-(k), \dplus]$ and  $\beta > \beta_-(k)$ are such that (\ref{eqProp_smmRemedy1}) holds. Then $$ \frac{1}{n} \Erw \ln \left [ Z_\beta(\vec \Phi) \right] \sim  \frac{1}{n} \ln \Erw \left [ Z_\beta(\vec \Phi) \right].$$ \end{lemma}
\begin{proof} 
We can apply \Lem~\ref{Lemma_antiPlanting} to find that
$$ \frac{\Erw[Z_{\rm tame}(\vec \Phi, \beta)]}{\Erw[Z_\beta(\vec \Phi)]} \geq \exp(o_n(n)) .$$
Hence \Lem~\ref{Lemma_smm} implies that 
$$\liminf_{n \to \infty}\frac{\Erw[Z_{\rm tame}(\vec \Phi, \beta)]^2}{\Erw[Z_{\rm tame}(\vec \Phi, \beta)^2]}>\exp(o_n(n)).$$ Using the Paley-Zigmund inequality we have
$$ \liminf_{n \to \infty} \Prob \left[Z_{\rm tame}(\vec \Phi, \beta)  \geq \Erw \left[Z_{\rm tame}(\vec \Phi, \beta)\right]/2 \right] \geq \exp(o_n(n)). $$ 
In particular, we have
$$ \liminf_{n \to \infty} \Prob \left[Z_\beta(\vec \Phi)  \geq \frac{c_1}{2} \Erw \left[Z_\beta(\vec \Phi)\right]  \right] \geq \exp(o_n(n)) .$$
In other words
$$ \liminf_{n \to \infty} \Prob \left[\frac{1}{n} \ln Z_\beta(\vec \Phi)  \geq \frac{1}{n} \ln \Erw \left[Z_\beta(\vec \Phi)\right] -o_n(1) \right] \geq \exp(o_n(n)) .$$
It follows from \Lem~\ref{lemma_concentration} that $$\frac{1}{n} \Erw \ln \left[ Z_\beta(\vec \Phi) \right] \geq  \frac{1}{n} \Erw \ln \left[ Z_\beta(\vec \Phi) \right] - o_n(1). $$
The proof is completed by Jensen's inequality which give us $$ \frac{1}{n} \Erw \ln \left[ Z_\beta(\vec \Phi) \right] \leq \frac{1}{n}  \ln\Erw \left[ Z_\beta(\vec \Phi) \right].$$
 \end{proof}
 
 The second part of the proposition will be a simple application of the following lemma, which is similar to \Lem~\ref{Lemma_antiPlanting}.
 
 \begin{lemma} \label{lemma_planting_adirect} Let $d \leq \dplus$ and $\beta \in \mathbb{R}$ be fixed. Assume that there exists a sequence of event $\hcE_n$ such that 
 $$\Prob \left[ \vec \Phi \in \hcE_n \right] =1-o_n(1) \qquad \textrm{and}\qquad \liminf_{n \to \infty} \Prob \left[ \vec \hPhi \in \hcE_n \right]^{1/n} <1.$$
 Then we have $$ \liminf_{n \to \infty} \frac{1}{n} \ln \Erw \left[ Z_\beta(\vec \Phi) \right] <  \limsup_{n \to \infty} \frac{1}{n} \Erw \ln \left [ Z_\beta(\vec \Phi) \right].$$  \end{lemma}
 
 \begin{proof}
Let $m(n)$ be a monotically increasing sequence and $\xi$ be such that $\cProb \left[ \vec \hPhi \in \hcE_n \right]^{1/m(n)} \leq \exp(-\xi m(n))$.
  We have, by the same steps as in the proof of \Lem~\ref{Lemma_antiPlanting}
 \begin{align*} \Erw \left[ Z_\beta(\vec \Phi) \mathbf{1}_{\vec \Phi \in \hcE_{m(n)}} \right]
  \leq  \cProb \left[ \vec \hPhi \in \hcE_{m(n)} \right] \Erw \left[ Z_\beta(\vec \Phi) \right] \leq \exp(-\xi m(n)) \Erw \left[ Z(\vec \Phi(m(n),k,d), \beta) \right].
 \end{align*}
 Thereby we obtain, using that $\frac{1}{n}\ln Z_\beta(\Phi) \in [-\beta + \ln 2, \ln 2]$ for all $\Phi$
 \begin{align*} \frac{1}{m(n)} \Erw \ln \left[  Z_\beta(\vec \Phi) \right] &=  \frac{1}{m(n)} \Erw \ln \left[  Z_\beta(\vec \Phi) \mathbf{1}_{\vec \Phi \in \hcE_{m(n)}}\right] +o_n(1)
 \\ & \leq  \frac{1}{m(n)} \ln \Erw \left[  Z_\beta(\vec \Phi) \mathbf{1}_{\vec \Phi \in \hcE_{m(n)}}\right] +o_n(1)
 \\ & \leq   \frac{1}{m(n)} \Erw \ln \left [ Z_\beta(\vec \Phi) \right] - \xi + o_n(1).
  \end{align*}
  \end{proof}
 
 \begin{proof}[Proof of \Lem~\ref{Lem_smmRemedy}] Assume that Eq. (\ref{eqProp_smmRemedy1}) holds. Then by \Lem~\ref{lemma_planting_direct} we have $ \frac{1}{n} \Erw \ln \left [ Z_\beta(\vec \Phi) \right] \sim \frac{1}{n} \ln \Erw \left[ Z_\beta(\vec \Phi) \right]$.
 
 Assume that Eq. (\ref{eqProp_smmRemedy1}) does not hold and let $\epsilon$ be such that $$\limsup_{n \to \infty} \frac{1}{n} \cErw \left[  \ln \cC_{\vec \hPhi,\vec {\hat \sigma}}(\beta) \right] \geq \liminf_{n \to \infty} \frac{1}{n} \ln \Erw \left[ Z_\beta(\vec \Phi) \right]+\epsilon $$
 Let $z = \frac{1}{n} \ln \Erw \left[ Z_\beta(\vec \Phi) \right]+\epsilon/2$ and $\cE_n$ be the event that $\frac{1}{n} \ln Z_\beta(\vec \Phi) > \frac{1}{n} \ln \Erw \left[ Z_\beta(\vec \Phi) \right]+\epsilon/2$. Then using Jensen's inequality and \Lem~\ref{lemma_concentration} we obtain $\Prob \left[ \vec \Phi \in \cE_n \right]^{1/n} \sim 1$ while $\liminf_{n \to \infty} \Prob \left[ \vec \hPhi \in \cE_n \right]^{1/n} < 1$. Therefore with \Lem~\ref{lemma_planting_adirect} we have $$\liminf_{n \to \infty} \frac{1}{n} \Erw \ln \left [ Z_\beta(\vec \Phi) \right] < \limsup_{n \to \infty} \frac{1}{n} \ln \Erw \left[ Z_\beta(\vec \Phi) \right].$$ 
  \end{proof}
 
 \subsection{Proof of \Lem~\ref{lemma_vanilla_non_local_0} and \Lem~\ref{lemma_vanilla_non_local}}
 \label{sec_proof_lemma_vanilla_non_local}
 
 We first need to study $f_{2,k}(d,\beta, \cdot)$ locally around $\alpha = 1/2$ and compare it with $f_{1,k}(d,\beta)$. This will be given by the two following lemmas.
 
\begin{lemma}\label{lemma_vanilla_exp_f1} We have, for $d \leq \dplus$ and $\beta \in \mathbb{R}$, $$f_{1,k}(d,\beta) =\ln2- \frac{d}{k} \left( c_\beta 2^{-k} +  2^{-1-2k} -k  2^{-1-2k} \right) + \tilde O_k(4^{-k}). $$ \end{lemma}
\begin{proof} The result follows from a direct computation, using the observation that $q = \frac{1}{2}+c_\beta 2^{-1-k}+ \tilde O_k(4^{-k})$.
\end{proof}
\begin{lemma} \label{lemma_vanilla_expansion_f2} Let $d \leq \dplus$ and $\beta \in \mathbb{R}$ be fixed. $f_{2,k}$ is of class $C^\infty$ on $\mathbb{R}^2 \times(0,1)$. It satisfies $f_{2,k}(d,\beta,1/2) = 2 f_{1,k}(d,\beta)$, $\frac{\partial}{\partial \alpha} f_{2,k}(d,\beta,1/2) = 0$ and $\sup_{|\alpha-1/2| \leq 2^{-k/3}}\frac{\partial^2}{\partial \alpha^2} f_{2,k}(d,\beta,\alpha) < 0.$ \end{lemma}
\begin{proof}
$\alpha \to H(\alpha)$ is clearly of class $C^\infty$ on $(0,1)$. Similarly, $(\beta, h,\hq) \to \ln z_{2,k}(\beta, h,\hq)$ is of class $C^\infty$ on $\cT$ and $D$ is of class $C^\infty$ on $(0,1) \times \cT$. The smoothness of $\alpha \to f_{2,k}(d,\beta,\alpha)$ therefore follows from the one of $\alpha \mapsto (h_\beta(\alpha), \hq_\beta(\alpha))$ granted by \Lem~\ref{lemma_vanilla_2_implicit}.

Because $(h_\beta(\alpha),\hq_\beta(\alpha))$ satisfy $g_{2,k,\beta}(h_\beta(\alpha),\hq_\beta(\alpha)) = ( \frac{1}{2},\frac{1-\alpha}{2})$, we have using the chain rule
\begin{align*} \frac{\partial z_{2,k}}{\partial h}(\beta,h_\beta(\alpha), \hq_\beta(\alpha)) &= \frac{\partial D_2}{\partial h} (\alpha, h_\beta(\alpha), \hq_\beta(\alpha)),
\\ \frac{\partial z_{2,k}}{\partial \hq}(\beta,h_\beta(\alpha), \hq_\beta(\alpha)) &= \frac{\partial D_2}{\partial \hq} (\alpha, h_\beta(\alpha), \hq_\beta(\alpha)) . \end{align*}
The differential of $f_{2,k}$ with respect to $\alpha$ then simplifies to
\begin{align*} \frac{\partial f_{2,k}}{\partial \alpha} (d,\beta,\alpha) &= H'(\alpha)  + d \frac{\partial}{\partial \alpha} D_2(\alpha, h_\beta(\alpha), \hq_\beta(\alpha))
\\ & = H'(\alpha) - d H'(\alpha) + \frac{d}{2} \ln \left(  \frac{\hq_\beta(\alpha) (1-2h_\beta(\alpha)+\hq_\beta(\alpha))}{(h_\beta(\alpha)-\hq_\beta(\alpha))^2} \right).
 \end{align*}
 In particular, for $\alpha=1/2$ we have $\hq_\beta(1/2) = h_\beta(1/2)^2$ and $\frac{\partial f_{2,k}}{\partial \alpha} (d,\beta,1/2) = 0$.
 
 Differentiating once more with respect to $\alpha$ yields
\begin{align*} \frac{\partial^2 f_{2,k}}{\partial^2 \alpha} (d,\beta,\alpha) &= H''(\alpha)  + d \frac{\partial^2}{\partial \alpha^2} D_2(\alpha, h_\beta(\alpha), \hq_\beta(\alpha))+ d \frac{\partial^2}{\partial \alpha \partial h} D_2(\alpha, h_\beta(\alpha), \hq_\beta(\alpha))+ d \frac{\partial^2}{\partial \alpha \partial \hq} D_2(\alpha, h_\beta(\alpha), \hq_\beta(\alpha))
\\ & = H''(\alpha) - d H''(\alpha) + \frac{d}{2}\left[ -\frac{2}{h_\beta(\alpha)-\hq_\beta(\alpha)}-\frac{2}{1-2h_\beta(\alpha)+\hq_\beta(\alpha)} \right] h_\beta'(\alpha)
\\ & \hspace{1cm} + \frac{d}{2}\left[ \frac{1}{\hq_\beta(\alpha)}+\frac{1}{1-2h_\beta(\alpha)+\hq_\beta(\alpha)} + \frac{2}{h_\beta(\alpha)-\hq_\beta(\alpha)}\right] \hq_\beta'(\alpha) .
 \end{align*}
 In particular for $|\alpha-1/2| \leq 2^{-k/3}$ we have with \Lem~\ref{lemma_vanilla_2_implicit}
 \begin{align*} \frac{\partial^2 f_{2,k}}{\partial^2 \alpha} (d,\beta,1/2) &= H''(\alpha) - d H''(\alpha) -8 d h_\beta'(\alpha) + 8d \hq_\beta'(\alpha)+ \tilde O_k(2^{-k/3})
\\&= -4 + 8 d \left[ \hq_\beta'(\alpha) - h_\beta'(\alpha) + \frac{1}{2} \right] +\tilde O_k(2^{-k/3}) 
\\ &= -4 + \tilde O_k(2^{-k/3}).  \end{align*}
\end{proof}

We now study $f_{2,k}(d,\beta, \alpha)$ when $|\alpha -1/2| > k^2 2^{-k/2}$.  For the sake of readability, we decompose this study in small steps. We first show that we can upper bound $f_{2,k}$ by a simpler function. Let $\ovz_{2,k} : \mathbb{R} \times (0,1) \to \mathbb{R}$ and $\ovf_{2,k} : \mathbb{R}^2 \times (0,1) \to \mathbb{R}$ be defined by
$$ \ovz_{2,k}(\beta,\alpha) =1-2 c_\beta 2^{-k} + c_\beta^2 \left( \frac{1-\alpha}{2} \right)^k , \qquad \ovf_{2,k}(d,\beta, \alpha) = \ln2 + H(\alpha) + \frac{d}{k} \ln \left(\ovz_{2,k}(d,\beta,\alpha)\right).$$

\begin{lemma}\label{simple_uper_bound_moments_f} For all $\alpha \in (0,1)$ we have $f_{2,k}(d,\beta,\alpha) \leq \ovf_{2,k}(d,\beta,\alpha)$. \end{lemma}

\begin{proof} Consider the event $B^{(2)},S^{(2)}$ defined in the proof of \Prop~\ref{prop_second_moment_rate_function} and their probability $\mathbb{P}'$ under the distribution where (with the notations of the proof of \Prop~\ref{prop_second_moment_rate_function}) 
$$\phi_{al}=
\begin{cases}
(1,1), &\text{with probability } (1-\alpha)/2\\
(1,-1), &\text{with probability } \alpha/2\\
(-1,1), &\text{with probability } \alpha/2\\
(-1,-1), &\text{with probability }  (1-\alpha)/2\\
\end{cases}$$
independently for all $a \in [m], l \in [k]$. We have
\begin{align*} \frac{1}{m} \ln \mathbb{P}'[S^{(2)}] &\sim \ln \left(1-2 c_\beta 2^{-k} + c_\beta^2 \left( \frac{1-\alpha}{2} \right) \right) - \ln \left(1 + \exp(-\beta) \right),
\\ \frac{1}{m} \ln \mathbb{P}'[B^{(2)}] &\sim 1,
\\\frac{1}{m} \ln \mathbb{P}'[S^{(2)}|B^{(2)}] &\sim \frac{1}{m} \ln \mathbb{P}[S^{(2)}|B^{(2)}]. \end{align*}
In particular, with Bayes' theorem 
$$ \ovf_{2,k}(d,\beta,\alpha) - f_{2,k}(d,\beta,\alpha) \sim \frac{1}{m} \ln \mathbb{P}'[S^{(2)}]- \frac{1}{m} \ln \mathbb{P}[S^{(2)} | B^{(2)}] \sim -\frac{1}{m} \mathbb{P}'\beta[B^{(2)}|S^{(2)}] \geq 0 $$
 \end{proof}

\begin{lemma} \label{lemma_upb_ovf_2} Assume that $d \leq \dplus$ and $\beta \in \mathbb{R}$. Then $$ \ovf_{2,k}(d,\beta,1/2-k^2 2^{-k/2}) < f_{2,k}(d,\beta,1/2).$$ \end{lemma}
\begin{proof} We compute 
\begin{align*}\ovf_{2,k}(d,\beta,1/2- k^2 2^{-k/2}) &= 2 \ln2 - k^4 2^{1-k} + \frac{d}{k} \left( - c_\beta 2^{1-k} - c_\beta^2 2^{-2k}  \right) + \tilde O_k(2^{-4k/3})
 \end{align*}
On the other hand using \Lem~\ref{lemma_vanilla_expansion_f2} and \Lem~\ref{lemma_vanilla_exp_f1} we have
\begin{align*} f_{2,k}(d,\beta,1/2) = 2 f_{1,k}(d,\beta) = 2 \ln2- 2\frac{d}{k} \left( c_\beta 2^{-k} +  2^{-1-2k} -k  2^{-1-2k} \right) + \tilde O_k(4^{-k}).  \end{align*}
It follows that 
$$\ovf_{2,k}(d,\beta,1/2- k^2 2^{-k/2}) -f_{2,k}(d,\beta,1/2) \leq - k^4 2^{1-k}+  d  O_k(4^{-k}) + \tilde O_k(4^{-k}) <0.$$ \end{proof}

In order to prove \Lem~\ref{lemma_vanilla_non_local_0}-\ref{lemma_vanilla_non_local}, it will be convenient to restrict the range of $(d,\beta)$ that we need to consider.  We define $\ovf_{2,k}:\mathbb{R} \times (0,1) \to \mathbb{R}$ and $\ovz_{2,k}: (0,1) \to \mathbb{R}$ by $$\ovf_{2,k}(d,\alpha) = \lim_{\beta \to \infty}\ovf_{2,k}(d,\beta,\alpha), \qquad \ovz_{2,k}(\alpha) = \lim_{\beta \to \infty} \ovz_{2,k}(\beta,\alpha).$$

The following claim is immediate, once one observes that $\frac{\partial}{\partial d}\ovf_{2,k}(d,\alpha) = \frac{1}{k} \ln \left( \bar{z}_{2,k}(d,\beta,\alpha) \right)$.
  \begin{claim} \label{claim_aux_vanilla_mon_1} Assume that $d \leq \dplus$. Then for $\alpha \in (0, 1/2-k^2 2^{-k/2})$ we have $$ \frac{\partial}{\partial d}\ovf_{2,k}(d,\alpha) \geq \frac{\partial}{\partial d}\ovf_{2,k}(d,1/2-k^2 2^{-k/2}).$$ \end{claim}
  
Similarly, we have the following.

\begin{claim} \label{claim_aux_vanilla_mon_2} Assume that $d \leq \dplus$. Then for $\alpha \in (0, 1/2- k^22^{-k/2})$ we have $$ \frac{\partial}{\partial \beta}\ovf_{2,k}(d,\beta,\alpha) \geq \frac{\partial}{\partial \beta}\ovf_{2,k}(d,\beta,1/2-k^2 2^{-k/2}).$$ \end{claim}

\begin{proof} We compute
$$ \frac{\partial}{\partial \beta}\ovf_{2,k}(d,\beta,\alpha) = - \frac{\exp(-\beta)}{2^{k-1}} \frac{d}{k} \frac{ 1 -  c_\beta \left( {1-\alpha} \right)^k }{1-2 c_\beta 2^{-k} + c_\beta^2 \left( \frac{1-\alpha}{2} \right)^k}.$$
In particular 
$$ \frac{\partial^2}{\partial \alpha \partial \beta}\ovf_{2,k}(d,\beta,\alpha)  = - \frac{d\exp(-\beta)}{2^{k-1}}  c_\beta (1-\alpha)^{k-1} (1+\tilde O_k(2^{-k}))<0.$$
  \end{proof}
  
  Therefore, in order to prove \Lem~\ref{lemma_vanilla_non_local_0} we can assume that $d= \dplus$ and $\beta \to \infty$, and to prove \Lem~\ref{lemma_vanilla_non_local}, we can focus on the following two cases. \begin{itemize} \item $d = d_-(k)$ and $\beta \to \infty$, \item $d = \dplus$ and $\beta = \beta_-(k)$. \\ \end{itemize}
  
  \begin{lemma}\label{lemma_aux_vanilla_ovf_inter} We have $$\sup_{\alpha \in [2^{-k+10},1/2- k^2 2^{-k/2}]} \ovf_{2,k}(\dplus,\alpha) \leq \ovf_{2,k}(\dplus,1/2-k^2 2^{-k/2}).$$ \end{lemma}
\begin{proof}
We first compute
\beq \label{eq_aux_vanilla_second_a0_bis}  \ovf_{2,k}(\dplus, 1/2-k^2 2^{-k/2}) = - k^4 2^{1-k}+ \tilde O_k(4^{-k}). \eeq
We differentiate $\ovf_{2,k}(\dplus,\alpha)$ with respect to $\alpha$.
$$ \frac{\partial \ovf_{2,k}}{\partial \alpha}(\dplus,  \alpha) = - \ln \left( \frac{\alpha}{1-\alpha} \right) - \frac{\dplus}{2^k} \frac{(1-\alpha)^{k-1}}{\ovz_{2,k}(\alpha)}. $$
Assume that $\alpha \in [1/2-k^2 2^{k/2}, 1/2-2^{-k/3}]$. Then we have
\beq \label{eq_aux_vanilla_second_a1_bis} \frac{\partial \ovf_{2,k}}{\partial \alpha}(\dplus, \alpha) \geq - \ln \left( \frac{1/2-k^2 2^{-k/2}}{1/2+k^2 2^{-k/2}} \right) -  (k \ln2 ) 2^{1-k} \left(1 + O_k(2^{-k})\right) >0. \eeq
Assume that $\alpha \in [0.4, 1/2-2^{-k/3}]$. Then we have
\beq \label{eq_aux_vanilla_second_a1} \frac{\partial \ovf_{2,k}}{\partial \alpha}(\dplus, \alpha) \geq - \ln \left( \frac{1/2-2^{-k/3}}{1/2+2^{-k/3}} \right) -  (k \ln2 ) {(0.6)^{k-1}} \left(1 + O_k(2^{-k})\right) >0. \eeq
Similarly, for $\alpha \in [2 (\ln k)/k, 0.4]$, we have
\beq \label{eq_aux_vanilla_second_a2} \frac{\partial \ovf_{2,k}}{\partial \alpha} \left (\dplus, \alpha \right) \geq - \ln \left( \frac{0.4}{0.6} \right) -  \frac{\ln2}{k}+ O_k((\ln k) k^{-2}) >0. \eeq
For $\alpha \in [2^{-k/10}, 2 (\ln k)/k]$, we compute with the help of (\ref{eq_aux_vanilla_second_a0_bis}), and using $-(1- \exp(-x)) \leq -x/2$ for $0<x< 1$
\beq \label{eq_aux_vanilla_second_a3_ter} \ovf_{2,k} \left (\dplus, \alpha \right) \leq \alpha \left(-\ln \alpha -\frac{ k \ln 2}{2}\right) +  O_k(\alpha) < \ovf_{2,k}(\dplus, 1/2-k^22^{-k/2}). \eeq
The lemma follows from Eq.(\ref{eq_aux_vanilla_second_a1_bis}-\ref{eq_aux_vanilla_second_a2}) and (\ref{eq_aux_vanilla_second_a3_ter}).

 \end{proof}
  
\begin{lemma}\label{lemma_border_case_1} We have $$\sup_{\alpha \in (0,1/2- k^2 2^{-k/2}]} \ovf_{2,k}(d_-(k),\alpha) \leq \ovf_{2,k}(d_-(k),1/2-k^2 2^{-k/2}).$$ \end{lemma}

\begin{proof}
We first compute
\beq \label{eq_aux_vanilla_second_b0}  \ovf_{2,k}(d_-(k), 1/2-k^2 2^{-k/2}) = k^5  2^{1-k}+ O_k(k^4 2^{-k}). \eeq
By \Lem~\ref{lemma_aux_vanilla_ovf_inter} and \Claim~\ref{claim_aux_vanilla_mon_1} we also have
\beq \label{eq_aux_vanilla_second_b1} \sup_{\alpha \in (2^{-k/10}, 1/2-k^2 2^{-k/2}]}  \ovf_{2,k}(d_-(k), \alpha) \leq  \ovf_{2,k}(d_-(k),1/2-k^2 2^{-k/2}). \eeq
Let $\alpha^\star$ be a maximum of $\ovf_{2,k}(d_-(k), \cdot)$ over $(0,2^{-k/10})$.The equation $\frac{\partial \ovf_{2,k}(d_-(k),\alpha)}{\partial \alpha}=0$ reads
\beq \label{eq_maximizer_boundary_vanilla_second_1} -\ln \left( \frac{\alpha^\star}{1-\alpha^\star} \right) = k \ln2 (1+o_k(1)). \eeq
Expanding for $\alpha \leq 2^{-k/10}$, we obtain that 
$\alpha^\star  \sim 2^{-k} $. Using (\ref{eq_maximizer_boundary_vanilla_second_1}) once again yields
$$ \alpha^\star = 2^{-1-k}  + \tilde O_k(4^{-k}).$$
In particular \begin{align*} \ovf_{2,k}(d_-(k), \alpha^\star) &= \ln2 + k 2^{-k} +2^{-k} + \frac{d_-(k)}{k} \left( -2^{-k}-2^{-1-2k} -k 2^{-2k} \right)+ \tilde O_k(4^{-k}) 
\\ & = k^5 2^{-k}+ \tilde O_k(4^{-k}). \end{align*}
Noting that $\lim_{\alpha \to 0} \frac{\partial}{\partial \alpha} \ovf_{2,k}(d_-(k),\alpha) = \infty$ and using (\ref{eq_aux_vanilla_second_b0}), this gives
\beq \label{eq_aux_vanilla_second_b4} \sup_{\alpha \in (0,2^{-k/10})} \ovf_{2,k}(d_-(k), \alpha) \leq \ovf_{2,k}(d_-(k),1/2-k^2 2^{-k/2}). \eeq
Collecting (\ref{eq_aux_vanilla_second_b1}) and (\ref{eq_aux_vanilla_second_b4}) ends the proof of the lemma.
 \end{proof}

\begin{lemma}\label{lemma_border_case_2} We have $$ \sup_{\alpha \in (0,1/2-k^2 2^{-k/2}]} \ovf_{2,k}(\dplus, \beta_-(k), \alpha) \leq \ovf_{2,k}(\dplus,\beta_-(k), 1/2-k^2 2^{-k/2}).$$ \end{lemma}

\begin{proof} By combining \Lem~\ref{lemma_aux_vanilla_ovf_inter}, \Claim~\ref{claim_aux_vanilla_mon_1} and \Claim~\ref{claim_aux_vanilla_mon_2} we obtain
\beq \label{eq_aux_vanilla_second_c1} \sup_{\alpha \in [2^{-k/10}, 1/2-k^22^{-k/2}]} \ovf_{2,k}(\dplus,\beta_-(k), \alpha) = \ovf_{2,k}(\dplus,\beta_-(k), 1/2-k^2 2^{-k/2}). \eeq
Let $\alpha^\star$ be a maximum of $\ovf_{2,k}(\dplus,\beta_-(k), \cdot)$ over $(0,2^{-k/10})$.The equation $\frac{\partial \ovf_{2,k}(\dplus,\beta_-(k),\alpha)}{\partial \alpha}=0$ is again given by (\ref{eq_maximizer_boundary_vanilla_second_1}) and hence
$$ \alpha^\star = \frac{1}{2} - 2^{-1-k}  + \tilde O_k(4^{-k}).$$
In particular \begin{align*} \ovf_{2,k}(\dplus,\beta_-(k), \alpha^\star) &= \ln2 + k 2^{-k} +2^{-k} + \frac{\dplus}{k} \left( -c_\beta 2^{-k}-2^{-1-2k} -k 2^{-2k} \right)+ \tilde O_k(4^{-k}) 
\\ & = \tilde O_k(4^{-k}), \end{align*}
while $$ \ovf_{2,k}(\dplus,\beta_-(k), 1/2-k^2 2^{-k/10}) = k^{10}  2^{1-k} + O_k(k^4 2^{-k}).$$
Noting that $\lim_{\alpha \to 0} \frac{\partial}{\partial \alpha} \ovf_{2,k}(d,\beta_-(k), \alpha) = \infty$, this gives
\beq \label{eq_aux_vanilla_second_c2} \sup_{\alpha \in (0,2^{-k/10})} \ovf_{2,k}(\dplus,\beta_-(k), \alpha) < \ovf_{2,k}(\dplus,\beta_-(k), 1/2-k^2 2^{-k/2}). \eeq
This concludes the proof of the lemma.
\end{proof}

\begin{proof}[Proof of \Lem~\ref{lemma_vanilla_non_local_0}] Let $d$ and $\beta$ be as in \Lem~\ref{lemma_vanilla_non_local_0}.
We first observe that for $\alpha \in (0,1/2)$, $f_{2,k}(d,\beta,\alpha) \geq f_{2,k}(d,\beta,1-\alpha)$. Therefore, we can restrict ourselves to $\alpha \in (0,1/2)$. By \Lem~\ref{lemma_vanilla_expansion_f2} we have $$ \sup_{1/2 - 2^{-k/3} \leq \alpha \leq 1/2} f_{2,k}(d,\beta,\alpha) \leq f_{2,k}(d,\beta, 1/2) .$$
Combining \Claim~\ref{claim_aux_vanilla_mon_1}-\ref{claim_aux_vanilla_mon_1} with \Lem~\ref{lemma_aux_vanilla_ovf_inter} we obtain
$$\sup_{\alpha \in [2^{-k/10},1/2- k^2 2^{-k/2}]} \ovf_{2,k}(d,\beta,\alpha) \leq \ovf_{2,k}(d,\beta,1/2-k^2 2^{-k/2}).$$
Using in addition \Lem~\ref{simple_uper_bound_moments_f} and \Lem~\ref{lemma_upb_ovf_2} ends the proof of the lemma.
  \end{proof}

\begin{proof}[Prood of \Lem~\ref{lemma_vanilla_non_local}] The proof follows by combining the previous results with \Lem~\ref{lemma_border_case_1} and \Lem~\ref{lemma_border_case_2}. \end{proof}

\end{document}